\documentclass[reqno]{amsart}
\newtheorem{theorem}{Theorem}[section]
\newtheorem{lemma}[theorem]{Lemma}
\newtheorem{proposition}[theorem]{Proposition}

\theoremstyle{definition}

\newtheorem{example}[theorem]{Example}

\usepackage[a4paper, left=3.25cm, right=3.25cm, top=3.5cm, bottom=3.5cm]{geometry}
\theoremstyle{remark}
\usepackage{caption}
\newtheorem{remark}[theorem]{Remark}
\numberwithin{equation}{section}
\newcommand{\be}{\begin{eqnarray}}
\newcommand{\en}{\end{eqnarray}}
\newcommand{\nn}{\nonumber}
\newcommand{\bmf}{\boldsymbol}
\newcommand{\mr}{\mathrm}

\newcommand{\corrauthor}[1]{#1$^{*}$}
\usepackage{graphicx}
\usepackage{amssymb}
\usepackage{subfigure}
\usepackage{amsmath}
\usepackage{epstopdf}
\usepackage{dsfont}
\usepackage{varioref}
\usepackage{hyperref}
\usepackage{cleveref}
\usepackage{color}
\usepackage{mathrsfs,amsfonts}
\usepackage{tikz}
\usepackage{graphicx}
\usetikzlibrary{patterns}
\usepackage{circuitikz}
\usepackage{tkz-base,tkz-euclide}
\usepackage{cite}
\usepackage{ulem}

\graphicspath{{figures/}}

\begin{document}
\title{Finite volume element method for Landau-Lifshitz equation}
\author{Yunjie Gong}
\address{School of Mathematical Sciences, Soochow University, Suzhou 215006, China}
\email{gyj1564763342@163.com}

\author{Jingrun Chen}
\address{School of Mathematical Sciences and Suzhou Institute for Advanced Research, University of Science and Technology of China}
\address{Suzhou Big Data \& AI Research and Engineering Center, Suzhou, China}
\email{jingrunchen@ustc.edu.cn}

\author{\corrauthor{Rui Du}}
\address{School of Mathematical Sciences, Soochow University, Suzhou 215006, China}
\email{durui@suda.edu.cn}
\author{\corrauthor{Panchi Li}}
\address{School of Mathematical Sciences, Soochow University, Suzhou 215006, China}
\email{lipch@suda.edu.cn}



\thanks{$^{*}$ Corresponding authors.}

\begin{abstract}
The Landau-Lifshitz equation describes the dynamics of magnetization in ferromagnetic materials. Due to the essential nonlinearity and nonconvex constraint, it is typically solved numerically. In this paper, we developed a finite volume element method (FVEM) with the Gauss-Seidel projection method (GSPM) for the micromagnetics simulations. We provide the approximation error in space and depict the energy law when the FVEM is adopted. Owing to the GSPM for time-marching, the discrete system is decoupled component by component, making the computational complexity comparable to that of solving the scalar heat equation implicitly. This significantly accelerates real simulations. We present several numerical experiments to validate the theoretical analysis and the efficiency gain. Additionally, we study the blow-up solution and efficiently simulate the 2D magnetic textures using the proposed method.
\end{abstract}


\maketitle

\section{Introduction}
In micromagnetics, the fundamental quantity of interest is the magnetization,
a three-dimensional vector field with constant magnitude at each lattice point when the system is below the Curie temperature.
Subject to various external controls,
the system exhibits rich dynamic behaviors and evolves toward equilibrium guided by the Landau-Lifshitz~(LL) equation~\cite{Landau,Gilbert1955}.
The LL equation consists of a gyromagnetic term and a damping term, which model the precession motion and energy dissipation, respectively. They are both nonlinear and conserve the constant length as time evolves, which poses challenges for analytical dynamics studies and numerical simulations of magnetic materials.

The mathematical theory and numerical methods for the LL equation have attracted considerable attention in the past few decades (see, e.g., \cite{Kruz,Maekawa,Shinjo,Melcher,Zhou}).
The existence and uniqueness of weak solutions have been established in various works \cite{Alouges,Carbou1,Guo}.
Alouges and Soyeur established the existence of a global weak solution in the three-dimensional setting \cite{Alouges}.
Guo and Hong established the global existence of solutions and revealed new connections between harmonic maps and  solutions of the LL equation \cite{Guo}.
Meanwhile,
substantial efforts also have been devoted to develop accurate and efficient numerical approaches.
Building on the finite difference method (FDM),
Miltat and Donahue \cite{Miltat} proposed efficient approaches for micromagnetic simulations based on both field-driven and energy-based formulations.
Fuwa et al. \cite{Fuwa} proposed an implicit nonlinear scheme and provided its unique solvability.
Recently,
Chen et al. \cite{chen} developed a second-order semi-implicit projection method for micromagnetics simulations.
Meanwhile, numerical approaches based on the finite element method (FEM) have been continuously developed. Bartels and Prohl \cite{Bartels} proposed a fully implicit nonlinear scheme and demonstrated its unconditional stability.
Shortly thereafter,
Alouges et al. \cite{Alouges1,Alouges2,Alouges3} developed a fully discrete linearized $\theta$-scheme and provided the corresponding stability condition.
Gao \cite{Gao} proposed a linearized backward Euler method and established optimal error estimates in both the $L^2$ and $H^1$ norms.
Besides, An \cite{an} gave the optimal error estimates for the linearized Crank-Nicolson Galerkin method.
Among these approaches, either a nonlinear system or a linearized system with variable coefficients must be solved due to the nonlinearity of the equation.
To improve simulations' efficiency, a series of Gauss-Seidel projection methods (GSPMs) have been developed in the past decades~\cite{GSPM2001JCP,Li,li2024enhanced}, in which only heat equations with constant coefficients need to be solved.

Our motivation is adopting a robust spatial discretisation method, the finite volume element method (FVEM), for the micromagnetics simulations.
The FVEM is to approximate the discrete fluxes in the conventional finite volume method by replacing the finite element approximation of the solution for PDEs \cite{cai1}.
Similar to the FEM, the FVEM exhibits a capability of handling complex boundary conditions and irregular geometries.
Moreover, it generally features a simpler solver structure and facilitates the implementation of advanced techniques with low computational costs.
Recent developments of the FVEM for various complex problems have been made to further demonstrate its capabilities and advantages, such as for moving-domain convection-diffusion problems \cite{GAO2021113537}, the Darcy flow in fractured media \cite{chen2022finite}, and nonlinear parabolic equations \cite{AAMM-17-3}. Another aspect of the developments of the FVEM is the high-order scheme. Its construction is not a direct extension of the high-order FEM, and some significant progresses have been made in recent years (see, for example \cite{Zhang2015,Zhang2014,zhou2023three,wu2021finite}). 

In this paper, we develop a novel numerical approach for the strongly nonlinear LL equation that combines the FVEM in space and GSPM in time. Due to the inherent difficulty of error analysis for GSPM, the theoretical analysis only considers the spatial approximation error. Nevertheless, it is still a challenging task to analyze the approximation error of FVEM because of the nonlinearity of the LL equation. In order to achieve the target,
we construct a linear auxiliary problem with variable coefficients and prove the approximation error of the FVEM.
In a specific setting,
the auxiliary problem becomes a linearized implicit scheme of the LL equation,
which enables us to estimate the approximation error for the original problem.
Thereafter, the discrete energy dissipation is naturally obtained.
We conduct several numerical experiments to verify the
theoretical analysis. To further explore the proposed numerical method, we study the blow-up solution of the LL equation in cases where the error analysis near the blow-up time becomes unrealistic. Besides, we also simulate the static magnetic textures in 2D ferromagnetic films subject to the Dirichlet boundary condition. Thanks to the implementation of the GSPM, we can study the micromagnetic dynamics at a cost comparable to that of solving the heat equation with constant coefficients. These demonstrate that the proposed method offers a promising and efficient approach to simulating complex micromagnetic systems.

The rest of this paper is organized as follows.
In \Cref{Sec2:Basic Model and numerical methods}, we introduce the basis model and the FVEM.
In \Cref{Sec3:Some auxiliary lemmas and error analysis for FVE method}, we present the error estimate and derive the discretized energy dissipation law.
The numerical approach is detailed and numerical tests are conducted as in \Cref{sec4: Temporal discretization and numerical experiments}.
Conclusions are drawn in \Cref{sec5:Conclusions}.

\section{Mathematical Model and numerical methods \label{Sec2:Basic Model and numerical methods}}
\subsection{The LL equation}
The fundamental form of the  LL equation is given by
\begin{equation*}
\partial_t\bmf{M}
=
-\mu_0\gamma\bmf{M} \times \mathcal{H}
- \frac{ \mu_0\gamma \alpha } {{M}_s} \bmf{M} \times (\bmf{M} \times \mathcal{H}),
\label{LL1}
\end{equation*}
where $\bmf{M} = (M_1, M_2, M_3)^T$ is the magnetization that satisfies $|\bmf{M}| = M_s$ with $M_s$ being the saturation magnetization,
$\gamma$ is the gyromagnetic ratio,
$\alpha$ is the damping parameter,
$\mu_0$ is the magnetic permeability of vacuum,
and $\mathcal{H}$ is the effective field given by the variation calculation $\mathcal{H} = -(1/\mu_0)\delta\mathcal{F}/\delta\bmf{M}$ of the magnetic free energy
\begin{equation*}
\mathcal{F}[\bmf{M}]
=
\int_{\Omega}\frac{A}{M_s^2}|\nabla \bmf{M}|^2 \,\mr{d}\bmf{x}
+
\int_{\Omega}\Phi\left(\frac{\bmf{M}}{M_s}\right)\,\mr{d}\bmf{x}
-
\mu_0\int_{\Omega}\bmf{H}_{\mr{e}}\cdot\bmf{M}\,\mr{d}\bmf{x}
+
\frac{\mu_0}{2}\int_{\Omega}\bmf{H}_{\mr{s}} \cdot \bmf{M}\,\mr{d}\bmf{x}.
\label{LL2}
\end{equation*}
In the energy functional,
$\Omega$ denotes the volume occupied by the ferromagnetic body,
$A$ is the exchange constant,
$\Phi\left(\frac{\bmf{M}}{M_s}\right)$ is a smooth function that defines the anisotropy,
and $\bmf{H}_{\mr{e}}$ is the external magnetic field.
In the last term,
$\bmf{H}_{\mr{s}}$ denotes the stray field, which is simplified to $\bmf{H}_{\mr{s}} \cdot \bmf{M} = {M}_3^2$  under the assumption of ultra-thin films.
Let the uniaxial anisotropy be along with $\bmf{e}_1 = (1, 0, 0)^T$,
and the effective field is calculated as
\begin{equation*}
\mathcal{H}
=
-\frac{\delta \mathcal{F}}{\delta \bmf{M}}
=
\frac{2 A}{M_s^2}\Delta \bmf{M}
-
\frac{2 K_{u}}{M_s^2}(M_2\bmf{e_2}+M_3\bmf{e_3})
-
\mu_0 M_3 \bmf{e}_3
+
\mu_0 \bmf{H}_{\mr{e}},
\end{equation*}
where $\bmf{e}_2=(0,1,0)^T$, $\bmf{e}_3=(0,0,1)^T$.
Define $\bmf{m}=M_s \bmf{M}$, $\bmf{H}_{\mr{s}} = M_s \bmf{h}_{\mr{s}}$ and $\bmf{H}_{\mr{e}}=M_s \bmf{h}_{\mr{e}}$.
Apply the spatial rescaling $\bmf{x}\rightarrow L \bmf{x}$ and
then, the magnetic free energy can be written as $\mathcal{F} = (\mu_0 M_s^2)F$ with
\begin{equation*}
F[\bmf{m}]
=
\frac{\epsilon}{2} \int_{\Omega}|\nabla \bmf{m}|^2 \,\mr{d}\bmf{x}
+
\frac{q}{2} \int_{\Omega}({m}_2^2+{m}_3^2) \,\mr{d}\bmf{x}
-
\int_{\Omega}\bmf{h}_{\mr{e}}\cdot\bmf{m} \,\mr{d}\bmf{x}
+
\frac{1}{2}\int_{\Omega}{m}_3^2 \,\mr{d}\bmf{x},
\end{equation*}
where
\begin{equation*}
    \epsilon=\frac{2 A}{\mu_0 M_s^2 L^2},\quad q=\frac{2 K_{u}}{\mu_0 M_s^2}.
\end{equation*}

Furthermore, the dimensionless LL equation is induced by applying the time rescaling $ t\rightarrow (\mu_0 \gamma M_s)^{-1}t$
\begin{equation}
\partial_t\bmf{m}
=
-
\bmf{m} \times {\bmf{h}}
-
\alpha \bmf{m} \times (\bmf{m} \times {\bmf{h}}),
\label{LL4}
\end{equation}
where
\begin{equation*}
\bmf{h} = \epsilon\Delta \bmf{m}-q(m_2\bmf{e_2}+m_3\bmf{e_3})-m_3 \bmf{e_3}+ \bmf{h}_{\mr{e}}.
\end{equation*}
The above LL form \eqref{LL4} will be used in the micromagnetics simulations, while to 
illustrate the main ideas of the numerical approach, we consider its simplified form
\begin{equation*}
  \partial_t\bmf{m}
  =
  -
  \bmf{m} \times \Delta\bmf{m}
  -
  \alpha \bmf{m} \times (\bmf{m} \times \Delta\bmf{m}).
\end{equation*}
Due to the non-convex constraint $|\bmf{m}| = 1$, this equation can be rewritten into
\begin{equation}
  \partial_t\bmf{m}
  -
  \alpha\Delta\bmf{m}
  +
  \bmf{m} \times \Delta\bmf{m}
  =
  \alpha |\nabla\bmf{m}|^2\bmf{m}.
  \label{equ:simplicity-LL-dimensionless}
\end{equation}
The numerical analysis of this equation will then be conducted in the sequel.
For the well-posedness, we adopt the initial condition and the Dirichlet boundary condition
\begin{gather}
    \bmf{m}(\bmf{x}, 0) = \bmf{m}_0(\bmf{x} ),\quad x\in \Omega\label{initial-condition}\\
    \bmf{m}(\bmf{x} ) = \bmf{g}(\bmf{x} ), \quad x\in\partial\Omega,\label{boundary}
\end{gather}
in which $|\bmf{m}_{0}| = |\bmf{g}| = 1$ holds in a point-wise sense.

\subsection{FVEM approximation}
We begin by introducing some notations that will be used throughout this paper.
The standard notation of Sobolev spaces $W^{s,p}(\Omega)$ ($1\leq p\leq\infty$) consists of functions that have generalized derivatives of order $s$ in $L^p(\Omega)$ that equips the norm
\begin{equation*}
\|m\|_{W^{s,p}}
=
\left( \int_{\Omega}\sum_{|\alpha|\leq s}|D^{\alpha}m|^{p} \, \mr{d}x \right)^{\frac{1}{p}}
\end{equation*}
with the standard modification for $p=\infty$. We then use
\( \bmf{L}^p \) and \( \bmf{W}^{k,r} \) to denote the vector Sobolev space $ (L^p(\Omega))^3 $ and $(W^{k,r}(\Omega))^3$, respectively.
In particular, we write \( \bmf{H}^k :=\bmf{W}^{k,2} \), and particularly define
\begin{equation*}
\bmf{X} = L^{\infty}([0, T]; \bmf{H}^{2}(\Omega)) \cap L^2([0, T]; \bmf{H}^3(\Omega)).
    \label{X}
\end{equation*}

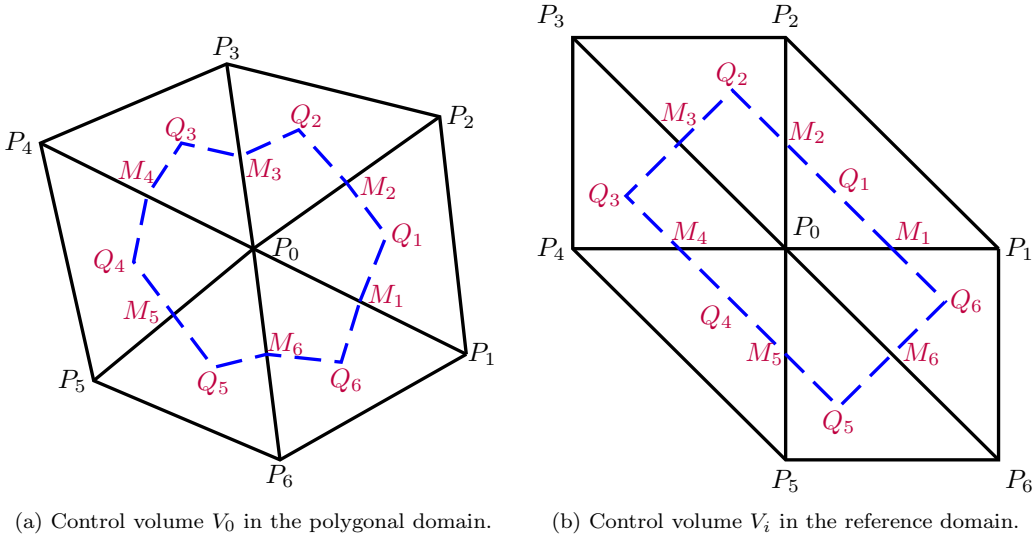
\begin{figure}[htbp]
\begin{tikzpicture}[>=latex,scale=0.7]
\begin{scope}
		 \draw[line width = 1.25pt] (0,0)node[right=0.1cm]{$P_0$} --  (4,-2);node
		 \draw[line width = 1.25pt] (0,0) -- (3.5,2.5);
		 \draw[line width = 1.25pt] (0,0) -- (-0.5,3.5);
         \draw[line width = 1.25pt] (0,0) -- (-4,2);
         \draw[line width = 1.25pt] (0,0) -- (-3,-2.5);
         \draw[line width = 1.25pt](0,0) -- (0.5,-4);
        \draw[line width = 1.25pt] (4,-2) to (3.5,2.5) to (-0.5,3.5) to(-4,2) to (-3,-2.5) to (0.5,-4) to (4,-2);
         \draw[dash pattern = on 10pt off 5pt, line width = 1.25pt, color = blue]
         (2,-1)
         to (2.5,0.25)
         to (1.75,1.25)
         to(0.85,2.25)
         to (-0.25,1.75)
         to (-1.35,2)
         to (-2,1)
         to (-2.25,-0.25)
         to (-1.5,-1.25)
         to (-0.75,-2.25)
         to (0.25,-2)
         to (1.65,-2.15)
         to (2,-1);
         \node at (4.3,-2){$P_1$};
         \node at (3.9,2.5){$P_2$};
         \node at (-0.5,3.8){$P_3$};
         \node at (-4.4,2){$P_4$};
         \node at (-3.4,-2.5){$P_5$};
         \node at (0.5,-4.3){$P_6$};
         \node[text=purple] at (2.5,-0.9){$M_1$};
         \node[text=purple] at (2.35,1.15){$M_2$};
         \node[text=purple] at (0.2,1.5){$M_3$};
         \node[text=purple] at (-2.2,1.45){$M_4$};
         \node[text=purple] at (-2.1,-1.2){$M_5$};
         \node[text=purple] at (0.6,-1.8){$M_6$};
         \node[text=purple] at (2.9,0.25){$Q_1$};
         \node[text=purple] at (1,2.5){$Q_2$};
         \node[text=purple] at (-1.35,2.25){$Q_3$};
         \node[text=purple] at (-2.7,-0.25){$Q_4$};
         \node[text=purple] at (-0.75,-2.55){$Q_5$};
         \node[text=purple] at (1.75,-2.5){$Q_6$};
          \end{scope}
  \end{tikzpicture}
 \caption{The control volume $V_0$ centered at $P_0$ in the polygonal domain.}
\label{Fig1}
\end{figure}

Let $\Omega$ be a polygonal domain.
We consider the quasi-uniform regular triangulation $T_h$ of $\Omega$, consisting of closed triangle elements $K$ such that $\bar{\Omega}=\cup_{K\in T_h}K$. Let $\mathcal{N}_h$ denote the set of all nodes with $N = |\mathcal{N}_h|$, and then define $\mathcal{N}_h^{0} = \mathcal{N}_h\cap\Omega$.
In~\Cref{Fig1}, the control volume associated with the node $P_0$ in the FVEM is represented by the polygon $M_1 Q_1 \cdots M_6 Q_6$,
where $M_i$ and $Q_i$ $(i = 1, \cdots, 6)$ are the points on edges and in the interior of elements, respectively.

For any arbitrary point $P_i \in \mathcal{N}_h$, control volumes $V_i$ associated with each node $P_i$ form the dual mesh $T_h^*$ in the polygonal region.

Let $h$ denote the maximum diameter of all elements in $T_h$.
We assume that the dual mesh $T_h^*$ is quasi-uniformly regular, i.e., there exists a positive constant $C$ such that
\begin{equation*}
C^{-1}h^2 \leq \mr{meas}(V_i) \leq Ch^2, \quad \forall\; V_i\in T_h^*.
\end{equation*}

Next, we introduce the discrete trial and test spaces.
Let $P_1(K)$ denote the space of linear polynomials on $K \in T_h$.
The trial space is defined as the conforming piecewise linear finite element space with homogeneous Dirichlet boundary condition:
\begin{equation*}
S_h
:=
\left\{
v_h \in H_0^1(\Omega)
:\;
v_h|_K \in P_1(K), \quad \forall\, K\in T_h
\right\}.
\end{equation*}
Equivalently,
\begin{equation*}
S_h=\operatorname{span}\{\phi_i(\bmf{x}) : \bmf{x}_i \in \mathcal N_h^0\},
\end{equation*}
where $\phi_i(\bmf{x})$ is the standard nodal basis function associated with the node $\bmf{x}_i$.
The test space is defined as the space of piecewise constant functions on the dual mesh associated with interior vertices:
\begin{equation*}
S_h^*
:=
\operatorname{span}\{\phi_i^*(\bmf{x}_i) : \bmf{x}_i \in \mathcal N_h^0\},
\end{equation*}
where $\phi_i^*(\bmf{x}_i)$ denotes the characteristic function of $V_i$, i.e.,
\begin{equation*}
\phi_i^*(\bmf{x}_i)
=
\begin{cases}
1, & \bmf{x}_i\in V_i,\\
0, & \text{otherwise}.
\end{cases}
\end{equation*}

For the LL equation, we use $\bmf{S}_h$ and $\bmf{S}_h^*$ denote $(S_h)^3$ and $(S_h^*)^3$,
respectively. Define the interpolation operator $I_h:\bmf{H}_0^1\cap \bmf{H}^2\rightarrow \bmf{S}_h$ such that \cite{thomee2013galerkin}
\begin{equation*}
    I_h \bmf{v}
    =
    \sum_{\bmf{x}_i\in\mathcal{N}_h^0}\bmf{v}(\bmf{x}_i)\phi_i(\bmf{x}),
\end{equation*}
then for all $\bmf{v}\in \bmf{H}_0^1(\Omega)\cap \bmf{H}^2(\Omega)$, we have \cite{brenner}
\begin{equation}
    \|I_h\bmf{v} - \bmf{v}\|_{\bmf{L}^2} \leq Ch^2\|\bmf{v}\|_{\bmf{H}^2},
    \label{In}
\end{equation}
\begin{equation}\label{In-prop}
\|I_h\bmf{v}\|_{\bmf{L}^{\infty}} \leq C \|\bmf{v}\|_{\bmf{L}^{\infty}}.
\end{equation}
Meanwhile, for any $\bmf{v}_h\in \bmf{S}_h$, we define another interpolation operator $I_h^*:\bmf{S}_h\rightarrow \bmf{S}_h^*$,
such that
$$I_h^*\bmf{v}_h=\sum_{\bmf{x}_i\in\mathcal{N}_h^0}\bmf{v}_h(\bmf{x}_i)\phi_i^*(\bmf{x}).$$

For $p>1$, It satisfies \cite{chen2010two,chou2003lp,chat1}
\begin{gather} \label{I_h^*}
    \|I_h^* \bmf{v}_h\|_{W^{0,p}}\leq\|\bmf{v}_h\|_{W^{0,p}}, \qquad
     \|\bmf{v}_h - I_h^* \bmf{v}_h\|_{W^{0,p}} \leq C h^s \|\bmf{v}_h\|_{W^{s,p}}, 0 \leq s \leq 1.
\end{gather}

For the numerical analysis, we assume that within finite time the LL equation \eqref{equ:simplicity-LL-dimensionless}-\eqref{boundary} has a unique local weak solution $\bmf{m}$ satisfying the regularity condition \cite{Gao}:
\begin{equation}
    \|\mathbf{m}\|_{L^{\infty}(0,T;\mathbf{H}^{3})}
   \! + \!
    \|\partial_t \mathbf{m}\|_{L^2(0,T;\mathbf{H}^2)}
   \! + \!
    \|\partial_t \mathbf{m}\|_{L^{\infty}(0,T;\mathbf{H}^1)}
   \! + \!
    \|\partial_{tt} \mathbf{m} \|_{L^2(0,T;\mathbf{L}^2)}
    \! \leq \!
    C_{\mathrm{reg}}.
    \label{regularity-condition}
\end{equation}
Meanwhile, the following bound of the numerical solution is also required \cite{Alouges}
\begin{equation}
        \|\nabla \bmf{m}_h\|_{\bmf{L}^{\infty}} \leq C.
    \label{working_set}
\end{equation}
We remark that with a certain initialization, the solution to the LL equation may blow up in finite time as shown in \cite{bartels2008numerical,Bartels}. In such cases, assumptions of strong regularity, such as those mentioned above, cannot be introduced. Error analysis for the problem near the blow up is still not realistic. Therefore, the theoretical analysis presented in this work is only applicable to scenarios before the blow-up time.

The FVEM approximation begins by
integrating \eqref{equ:simplicity-LL-dimensionless} over a control volume $V_i$ and applying Green's formula, which yields
\begin{equation}
  \int_{V_i} \partial_t\bmf{m}_h \,\mr{d}x \mr{d}y
  -
  \int_{\partial V_i} \alpha \nabla \bmf{m}_h \cdot\bmf{n} \,\mr{d}S
  +
  \int_{\partial V_i} \bmf{m}_h \times (\nabla \bmf{m}_h \cdot\bmf{n}) \,\mr{d}S
  =
  \int_{V_i} \alpha |\nabla \bmf{m}_h|^2 \bmf{m}_h \,\mr{d}x\mr{d}y,
  \label{equ:simplicity-semi-discretized}
\end{equation}
where $\bmf{n}$ denotes the unit outward normal vector of $V_i$. 
Multiplying both sides of \eqref{equ:simplicity-semi-discretized} by $\bmf{v}_h(\bmf{x}_i)$ and taking the sum of all terms $\bmf{x}_i \in \mathcal{N}_h$ yields the discrete scheme for \eqref{equ:simplicity-LL-dimensionless}:
Find $\bmf{m}_h \in \bmf{S}_h$,
such that for all $\bmf{v}_h \in \bmf{S}_h$,
\begin{equation}
\left\{
\begin{aligned}
&(\partial_t\bmf{m}_h, I_h^*\bmf{v}_h)
 +
 \mathcal{A}_h(\bmf{m}_h;\bmf{m}_h,I_h^*\bmf{v}_h)
 =
 \alpha(|\nabla\bmf{m}_h|^2\bmf{m}_h,I_h^*\bmf{v}_h), \\
&\bmf{m}_h(0) = I_h \bmf{m}_0, \\
&\bmf{m}_h(t)|_{\partial\Omega} = \bmf{0},
\end{aligned}
\label{equ:LL-semi-discretized1}
\right.
\end{equation}
where $\mathcal{A}_h(\bmf{m}_h;\bmf{m}_h,I_h^*\bmf{v}_h)
    =
    a_h(\bmf{m}_h,I_h^*\bmf{v}_h)
    +
    b_h(\bmf{m}_h;\bmf{m}_h,I_h^*\bmf{v}_h)$ with
\begin{equation*}
 a_h(\bmf{m}_h,I_h^*\bmf{v}_h)
 =
 - \sum_{\bmf{x}_i\in\mathcal{N}_h}\!\int_{\partial V_i}\alpha\,(\nabla \bmf{m}_h\cdot\bmf{n})\cdot I_h^*\bmf{v}_h \,\mr{d}S
 =
 - \sum_{\bmf{x}_i\in\mathcal{N}_h}\bmf{v}_h(\bmf{x}_i)\int_{\partial
V_i}\!\alpha \,\nabla \bmf{m}_h\cdot\bmf{n}\,\mr{d}S,
\end{equation*}
\begin{equation*}
 b_h(\bmf{m}_h;\bmf{m}_h,I_h^*\bmf{v}_h)
 \!=
 \!\!\!\sum_{\bmf{x}_i\in\mathcal{N}_h}\!\!\int_{\partial V_i}\!\!\!(\bmf{m}_h \! \times \! (\nabla \bmf{m}_h \cdot\bmf{n})) \cdot I_h^*\bmf{v}_h\,\mr{d}S\!
 =
 \!\!\!\sum_{\bmf{x}_i\in\mathcal{N}_h}\!\!\!\bmf{v}_h(\bmf{x}_i)\!\!\!\int_{\partial V_i}\!\!\!\bmf{m}_h \!\times\! (\nabla \bmf{m}_h \cdot\bmf{n}) \,\mr{d}S.
\end{equation*}
To establish the spatial error estimate for FVEM  and derive the discretized energy law,  we define $|\|\bmf{m}_h|\|_{0}^2 = (\bmf{m}_h, I_h^*\bmf{m}_h)$ in $\bmf{S}_h$.
It is straightforward to verify that there exist two positive constants $C_*, C^*$ independent of $h$ such that
\begin{equation}\label{norm-equ}
 C_*\|\bmf{m}_h\|_{\bmf{L}^2}\leq \||\bmf{m}_h|\|_{0}\leq C^*\|\bmf{m}_h\|_{\bmf{L}^2},\quad
\forall\, \bmf{m}_h\in \bmf{S}_h.
\end{equation}

Before the formal error analysis is given,
we state some inequalities.
In the two-dimensional setting,
there exists a constant $C>0$ such that for all $\bmf{v}\in \bmf{H}^2$,
the following inequalities \cite{adams1975sobolev,carbou2001regular} hold:
\begin{eqnarray}\label{Sobolev-inequality}
\begin{aligned}
&\|\bmf{v}\|_{\bmf{L}^r}
\leq
C \|\bmf{v}\|_{\bmf{H}^1} \quad (2 \leq r \leq 6), \\
&\|\bmf{v}\|_{\bmf{L}^4}
\leq
C \|\bmf{v}\|_{\bmf{H}^1}^{\frac{1}{2}}\|\bmf{v}\|_{\bmf{L}^2}^{\frac{1}{2}},\\
&\|\bmf{v}\|_{\bmf{L}^{\infty}}
\leq
C (\|\bmf{v}\|_{\bmf{L}^2}^2 + \|\Delta \bmf{v}\|_{\bmf{L}^2}^2)^{\frac{1}{2}},\\
&\|\nabla \bmf{v}\|_{\bmf{L}^{\infty}}
\leq
C \|\nabla \bmf{v}\|_{\bmf{L}^2}^{\frac{1}{2}}
(\|\nabla \bmf{v}\|_{\bmf{L}^2}^2 + \|\Delta \bmf{v}\|_{\bmf{L}^2}^2 + \|\nabla \Delta \bmf{v}\|_{\bmf{L}^2}^2)^{\frac{1}{4}},\\
&\|\nabla\bmf{v}\|_{\bmf{L}^4}
\leq
C \|\nabla\bmf{v}\|_{\bmf{L}^2}^{\frac{1}{2}}
( \|\nabla\bmf{v}\|_{\bmf{L}^2}^{2} + \|\Delta \bmf{v}\|_{\bmf{L}^2}^{2})^{\frac{1}{4}}.
\end{aligned}
\end{eqnarray}
In what follows, $C$ denotes a generic positive constant independent of $h$, whose value may change from line to line for notational simplicity.

\section{Auxiliary lemmas and error estimates\label{Sec3:Some auxiliary lemmas and error analysis for FVE method}}

\subsection{An auxiliary problem}\label{auxiliary-problem}
We first define the auxiliary problem as follows:
Given a function $ \bmf{\Phi} \in \bmf{X}$ satisfying $|\bmf{\Phi}| = 1$, where $\bmf{X}$ is defined in \eqref{X}, find $\tilde{\bmf{m}}$ such that
\begin{equation}
    \left\{\begin{aligned}
    &\partial_t\tilde{\bmf{m}}
    -
    \alpha \Delta\tilde{\bmf{m}}
    +
    \bmf{\Phi} \times \Delta \tilde{\bmf{m}}
    +
    \nabla\bmf{\Phi} \times \nabla \tilde{\bmf{m}}
    =
    \alpha |\nabla\bmf{\Phi}|^2 \tilde{\bmf{m}}, \\
    &\tilde{\bmf{m}} (x,y,0)
    =
    \tilde{\bmf{m}}_0, \\
    &\tilde{\bmf{m}}|_{\partial\Omega}
    =
    \bmf{0}.
    \end{aligned}\right.
    \label{equ:LL-auxiliary-problem}
\end{equation}
Let $\bmf{\Phi} = \tilde{\bmf{m}}$, and it becomes the LL equation.
For any given $\bmf{\Phi}$,
we estimate the approximation error of the FVEM for this auxiliary problem.
Then, we apply the fixed-point iteration and derive the spatial error estimate for the LL equation. We note that the homogeneous Dirichlet boundary condition is used in the defined auxiliary problem.
\begin{lemma}
    Given $ \bmf{\Phi} \in \bmf{X}$, there exists a positive constant $C$ such that
\begin{equation*}
        \|\tilde{\bmf{m}}\|_{\bmf{H}^2} \leq C \|\tilde{\bmf{m}}_0\|_{\bmf{H}^2},
\end{equation*}
where $C$ depends on $\alpha$, $ C_{\mathrm{reg}}$ and $\|\nabla\bmf{\Phi}\|_{\bmf{L}^{\infty}}$.
    \label{lem-m-H2}
\end{lemma}
\begin{proof}
We multiply both sides by \(\tilde{\bmf{m}}\) and integrate over \(\Omega\) for \eqref{equ:LL-auxiliary-problem}. Applying Green's formula yields
\begin{equation*}
    (\partial_t \tilde{\bmf{m}},\tilde{\bmf{m}})
    +
    \alpha (\nabla\tilde{\bmf{m}}, \nabla\tilde{\bmf{m}})
    =
    \alpha(|\nabla\bmf{\Phi}|^2\tilde{\bmf{m}},\tilde{\bmf{m}}),
\end{equation*}
where we have utilized the fact that $(\bmf{\Phi} \times \nabla\tilde{\bmf{m}},\nabla\tilde{\bmf{m}}) = 0$.

Since $\|\nabla\bmf{\Phi}\|_{\bmf{L}^{\infty}} \leq C$,
applying H\"{o}lder and Young's inequalities, we arrive at
\begin{equation}
    \frac{1}{2} \partial_t \|\tilde{\bmf{m}}\|_{\bmf{L}^2}^2
    +
    \alpha \|\nabla\tilde{\bmf{m}}\|_{\bmf{L}^2}^2
   \leq
    C \|\nabla\bmf{\Phi}\|_{\bmf{L}^{\infty}}^2 \|\tilde{\bmf{m}}\|_{\bmf{L}^{2}}^2
   \leq
    C \|\tilde{\bmf{m}}\|_{\bmf{L}^{2}}^2,
\label{equ:LL-auxiliary-problem-stable-L2-1}
\end{equation}
where $C$ depends on $\alpha$ and $\|\nabla\bmf{\Phi}\|_{\bmf{L}^{\infty}}$.

Repeating the above procedure with the test functions $\Delta\tilde{\bmf{m}}$ and $\Delta^2\tilde{\bmf{m}}$,
we obtain
\begin{equation}
\begin{aligned}
& \frac{1}{2} \partial_t \|\nabla\tilde{\bmf{m}}\|_{\bmf{L}^2}^2
+
\alpha \|\Delta \tilde{\bmf{m}}\|_{\bmf{L}^2}^2 \\
\leq
& \|\nabla\bmf{\Phi}\|_{\bmf{L}^{\infty}}  \|\nabla\tilde{\bmf{m}}\|_{\bmf{L}^{2}} \|\Delta \tilde{\bmf{m}}\|_{\bmf{L}^2}
+
C \|\nabla\bmf{\Phi}\|_{\bmf{L}^{\infty}}^2 \|\nabla\tilde{\bmf{m}}\|_{\bmf{L}^{2}}^2 \\
& +
C \|\nabla\bmf{\Phi}\|_{\bmf{L}^{\infty}} \|\Delta\bmf{\Phi}\|_{\bmf{L}^{4}} \|\tilde{\bmf{m}}\|_{\bmf{L}^{4}}\|\nabla\tilde{\bmf{m}}\|_{\bmf{L}^2}\\
\leq
& C \|\tilde{\bmf{m}}\|_{\bmf{H}^1}^2
+
C \epsilon^{-1} \|\nabla\tilde{\bmf{m}}\|_{\bmf{L}^2}^2
+
\epsilon \|\Delta \tilde{\bmf{m}}\|_{\bmf{L}^2}^2
\end{aligned}
\label{equ:LL-auxiliary-problem-stable-H1-1}
\end{equation}
and
\begin{equation}
\begin{aligned}
& \frac{1}{2} \partial_t \|\Delta\tilde{\bmf{m}}\|_{\bmf{L}^2}^2
+
\alpha\|\nabla\Delta\tilde{\bmf{m}}\|_{\bmf{L}^2}^2 \\
\leq
& C \|\Delta\bmf{\Phi}\|_{\bmf{L}^{4}} \|\nabla\tilde{\bmf{m}}\|_{\bmf{L}^{4}} \|\nabla\Delta\tilde{\bmf{m}}\|_{\bmf{L}^2}
+
C \|\nabla\bmf{\Phi}\|_{\bmf{L}^{\infty}} \|\Delta \tilde{\bmf{m}}\|_{\bmf{L}^2} \|\nabla\Delta\tilde{\bmf{m}}\|_{\bmf{L}^2} \\
& +
\alpha \|\nabla\bmf{\Phi}\|_{\bmf{L}^{\infty}}^2 \|\nabla\tilde{\bmf{m}}\|_{\bmf{L}^{2}}\|\nabla\Delta\tilde{\bmf{m}}\|_{\bmf{L}^{2}}
+ C
\|\nabla\bmf{\Phi}\|_{\bmf{L}^\infty} \|\Delta \bmf{\Phi}\|_{\bmf{L}^4} \|\tilde{\bmf{m}}\|_{\bmf{L}^{4}} \|\nabla\Delta \tilde{\bmf{m}}\|_{\bmf{L}^2}\\
\leq &
C \epsilon^{-1}(\|\Delta\bmf{\Phi}\|_{\bmf{L}^{4}}^2 \|\nabla\tilde{\bmf{m}}\|_{\bmf{L}^{4}}^2
+
\|\Delta\tilde{\bmf{m}}\|_{\bmf{L}^{2}}^2
+
 \|\nabla \tilde{\bmf{m}}\|_{\bmf{L}^{2}}^2
+
 \|\tilde{\bmf{m}}\|_{\bmf{L}^{4}}^2)
+
4\epsilon \|\nabla\Delta\tilde{\bmf{m}}\|_{\bmf{L}^{2}}^2
\\
\leq &
C \epsilon^{-1} \|\tilde{\bmf{m}}\|_{\bmf{H}^2}^2
+
4  \epsilon \|\nabla\Delta\tilde{\bmf{m}}\|_{\bmf{L}^{2}}^2. 
\end{aligned}
\label{equ:LL-auxiliary-problem-stable-H2-1}
\end{equation}
Then, summing the inequalities \eqref{equ:LL-auxiliary-problem-stable-L2-1}-\eqref{equ:LL-auxiliary-problem-stable-H2-1} yields
\begin{equation*}
\frac{1}{2} \partial_t \|\tilde{\bmf{m}}\|_{\bmf{H}^2}^2
+
\| \tilde{\bmf{m}}\|_{\bmf{H}^3}^2
\leq
C \|\tilde{\bmf{m}}\|_{\bmf{H}^1}^2
+
C \epsilon^{-1} \|\tilde{\bmf{m}}\|_{\bmf{H}^2}^2
+
4 \epsilon \|\tilde{\bmf{m}}\|_{\bmf{H}^3}^2.
\end{equation*}
Choosing $\epsilon < \frac{1}{6}$ and applying the Gronwall inequality, we further get
\begin{equation*}
 \|\tilde{\bmf{m}}(t)\|_{\bmf{H}^{2}} \leq C \|\tilde{\bmf{m}}(0)\|_{\bmf{H}^{2}}.
\end{equation*}
This completes the proof.
\end{proof}

For the high-order estimation of $\tilde{\bmf{m}}$,  the higher-order boundary conditions must be imposed. The corresponding proof is similar to that in \cite{ma2026optimalerrorestimateslinearized} (Lemma 5.3), hence we omit it here.

The regularity stated in the above lemma can be extended to the LL equation under the condition $\|\nabla \bmf{m}\|_{\bmf{L}^{\infty}} \leq C$. And the solution of the LL equation is obtained by the fixed-point iteration. A direct and simple example is the backward Euler method on a temporal interval $[T_1,T_2]$. Specifically, let $\tau = T_2-T_1$, which is sufficiently small, and then we have if there is a sequence $\{ {\bmf{m}}^{\,l}\}_{l\ge0}\subset \bmf{X}$ obtained by solving equations
\begin{equation}
    \frac{ {\bmf{m}}^{\,l+1}- {\bmf{m}}(T_1)}{\tau}
    -
    \alpha\Delta {\bmf{m}}^{\,l+1}
    +
     {\bmf{m}}^{\,l}\times\Delta {\bmf{m}}^{\,l+1}
    +
    \nabla {\bmf{m}}^{\,l}\times\nabla {\bmf{m}}^{\,l+1}
    =
    \alpha|\nabla {\bmf{m}}^{\,l}|^2 {\bmf{m}}^{\,l+1},
    \label{equ:LL-backward-Euler}
\end{equation}
the sequence $\{ {\bmf{m}}^{\,l}\}_{l\ge0}$ converges in $\bmf{H}^2$. This equation can be treated as the discretized form using the backward Euler scheme for the continuous equation
\begin{equation}
    \partial_t {\bmf{m}}^{\,l+1}
    -
    \alpha\Delta {\bmf{m}}^{\,l+1}
    +
     {\bmf{m}}^{\,l}\times\Delta {\bmf{m}}^{\,l+1}
    +
    \nabla {\bmf{m}}^{\,l}\times\nabla {\bmf{m}}^{\,l+1}
    =
    \alpha|\nabla {\bmf{m}}^{\,l}|^2 {\bmf{m}}^{\,l+1}
    \label{equ:LL-linearized}
\end{equation}
with the initial data $ {\bmf{m}}(T_1)$. Furthermore, we can derive a regularity estimate similar to that in Lemma \ref{lem-m-H2}. Next, we prove the convergence of the iteration sequence $\{ {\bmf{m}}^{\,l}\}_{l\ge0}$ provided by \eqref{equ:LL-backward-Euler}.

\begin{lemma}\label{lem-con}
Assume that the regularity condition \eqref{regularity-condition} holds. And a sequence $\{ \bmf{m}^{l+1}\}_{l \geq 0}$ solves the equation \eqref{equ:LL-backward-Euler} with a sufficiently small $\tau$. Then, there holds
$$ \lim_{l \rightarrow \infty} \| {\bmf{m}}^{l+1} - {\bmf{m}}^l\|_{\bmf{H}^2} = 0.$$
\end{lemma}
The proof is provided in Appendix \ref{app:lem-con}.

\subsection{A discrete auxiliary problem.}\label{discrete-auxiliary-problem}
We now estimate the approximation error of the FVEM for the auxiliary problem \eqref{equ:LL-auxiliary-problem}.
The corresponding semi-discrete scheme is given by :
Find $\tilde{\bmf{m}}_h \in \bmf{S}_h$ such that for all $\bmf{v}_h \in \bmf{S}_h$
\begin{equation}
    (\partial_t\tilde{\bmf{m}}_h, I_h^*\bmf{v}_h)
  +
  \mathcal{A}_h(\bmf{\Phi}; \tilde{\bmf{m}}_h, I_h^*\bmf{v}_h)
  =
  \alpha (|\nabla\bmf{\Phi}|^2\tilde{\bmf{m}}_h, I_h^*\bmf{v}_h),
  \label{auxiliary}
\end{equation}
where
\begin{equation*}
  \mathcal{A}_h(\bmf{\Phi}; \tilde{\bmf{m}}_h, I_h^*\bmf{v}_h)
  =
  a_h(\tilde{\bmf{m}}_h,I_h^*\bmf{v}_h)
  +
  b_h(\bmf{\Phi};\tilde{\bmf{m}}_h,I_h^*\bmf{v}_h).
\end{equation*}
For this semi-discrete scheme, we have the following lemma.
\begin{lemma}
Let $\bmf{\Phi} \in \bmf{X}$. For sufficiently small $h$,
there exists a positive constant $C$ such that
for all $\tilde{\bmf{m}}_h, \bmf{v}_h\in \bmf{S}_h$,
the coercivity
\begin{equation}
\mathcal{A}_h(\bmf{\Phi};\tilde{\bmf{m}}_h,I_h^*\tilde{\bmf{m}}_h)
\geq
\alpha \|\tilde{\bmf{m}}_h\|_{\bmf{H}^1}^2,
\label{Ah-coe}
\end{equation}
and the upper bound
\begin{equation}
|\mathcal{A}_h(\bmf{\Phi};\tilde{\bmf{m}}_h,I_h^*\bmf{v}_h)|
\leq
C \|\tilde{\bmf{m}}_h\|_{\bmf{H}^1}\|\bmf{v}_h\|_{\bmf{H}^1}.
\label{Ah-boundary}
\end{equation}
\label{lem-Ah}
\end{lemma}
The proof is provided in Appendix \ref{app:lem-Ah}.
This lemma gives the upper and lower bounds of $\mathcal{A}_h(\bmf{\Phi};\cdot,I^*_h\cdot)$.
Notice that $\mathcal{A}_h(\bmf{\Phi};\cdot, I^*_h\cdot)$ is generally unsymmetric
and this would affect the well-posedness of this problem.
We present the bound related to the symmetry as follows.
\begin{lemma}
Let $\bmf{\Phi} \in \bmf{X}$. For sufficiently small $h$,
there exists a positive constant $C$ such that
\begin{equation}
|\mathcal{A}_h(\bmf{\Phi};\tilde{\bmf{m}}_h,I_h^*\bmf{v}_h)
-
\mathcal{A}_h(\bmf{\Phi};\bmf{v}_h,I_h^*\tilde{\bmf{m}}_h)|
\leq
C h\|\tilde{\bmf{m}}_h\|_{\bmf{H}^1}\|\bmf{v}_h\|_{\bmf{H}^1}, \quad \forall \tilde{\bmf{m}}_h,\bmf{v}_h\in \bmf{S}_h.
\label{Ah-relation}
\end{equation}
\label{lem-Ah1}
\end{lemma}
The proof of this lemma is similar to that of Lemma 2.4 in \cite{chou} and is presented concisely in Appendix \ref{app:lem-Ah1}.

Let $R_h\tilde{\bmf m}(t)\in \bmf S_h$ be the elliptic projection defined by
	\begin{equation}
		\mathcal A(\bmf\Phi(t);R_h\tilde{\bmf m}(t)-\tilde{\bmf m}(t),\bmf v_h)=0,
		\qquad \forall \bmf v_h\in \bmf S_h.
		\label{R_h}
	\end{equation}
	Here,
	\begin{equation*}
		\mathcal{A}(\bmf{\Phi};\tilde{\bmf{m}},\bmf{v})
		=
		a(\tilde{\bmf{m}},\bmf{v})
		-
		b(\bmf{\Phi};\tilde{\bmf{m}},\bmf{v})
		:=
		\int_{\Omega} \nabla \tilde{\bmf{m}} \cdot \nabla \bmf{v} \, \mr{d} x \mr{d} y
		-
		\int_{\Omega} \bmf{\Phi} \times \nabla \tilde{\bmf{m}} \cdot \nabla \bmf{v} \, \mr{d} x \mr{d} y.
	\end{equation*}
	Let $\bmf\Phi\in \bmf{X}$, and then
	$\mathcal A(\bmf\Phi;\cdot,\cdot)$ is bounded and coercive on
	$\bmf H_0^1(\Omega)\times \bmf H_0^1(\Omega)$, which implies that the projection
	$R_h\tilde{\bmf m}(t)$ is well defined.
    
For $\tilde{\bmf{m}} \in \bmf{S}_h$, it holds
	\begin{equation}
		\|\tilde{\bmf{m}}-R_h\tilde{\bmf{m}}\|_{\bmf{L}^2} \leq C h^{2} \|\tilde{\bmf{m}}\|_{\bmf{H}^2},
		\label{R_h_p}
	\end{equation}
	and
	\begin{equation}
		\|R_h\tilde{\bmf{m}}\|_{\bmf{L}^{\infty}} \leq C \|\tilde{\bmf{m}}\|_{\bmf{L}^{\infty}}.
		\label{R_h_p1}
\end{equation}
Before exhibiting the approximation error of the FVEM for the auxiliary problem, we define the following two error functions:
\begin{equation}
\varepsilon_h(\bmf{f},\bmf{v}_h)
=
(\bmf{f},\bmf{v}_h)
-
(\bmf{f},I_h^*\bmf{v}_h),
\qquad \forall \bmf{v}_h\in \bmf{S}_h,
\label{erfun1}
\end{equation}
\begin{equation}
\varepsilon_A(\bmf{\Phi};R_h\tilde{\bmf{m}},\bmf{v}_h)
=
\mathcal{A}(\bmf{\Phi};R_h\tilde{\bmf{m}},\bmf{v}_h)
-
\mathcal{A}_h(\bmf{\Phi};R_h\tilde{\bmf{m}},I_h^*\bmf{v}_h),
\ \forall
\tilde{\bmf{m}},\bmf{v}_h\in \bmf{S}_h.
\label{erfun-A}
\end{equation}
The bounds for the error functions defined in \eqref{erfun1} and \eqref{erfun-A} are established in the following lemma.

\begin{lemma}
Given $\bmf{\Phi} \in \bmf{X}$ and $\bmf{v}_h\in \bmf{S}_h$,
then
\begin{equation}
|\varepsilon_h(\bmf{f},\bmf{v}_h)|
\leq
Ch^{i+j}\|\bmf{f}\|_{\bmf{H}^i}\|\bmf{v}_h\|_{\bmf{H}^j},
\quad
\bmf{f}\in \bmf{H}^i, i,j=0,1,
\label{erfun-h1}
\end{equation}
\begin{equation}
|\varepsilon_A(\bmf{\Phi};R_h \tilde{\bmf{m}},\bmf{v}_h)|
\leq Ch^{i+j}\|\tilde{\bmf{m}}\|_{\bmf{H}^{1+i}}\|\bmf{v}_h\|_{\bmf{H}^j},
\quad
\tilde{\bmf{m}}\in \bmf{H}^{1+i}\cap \bmf{H}^1_0,i,j=0,1.
\label{erfun-A1}
\end{equation}
\label{erfunction}
\end{lemma}
The proof of this lemma is given in the Appendix \ref{app:lem-erfun}.

We have now prepared the conditions for estimating the approximation error of FVEM for the auxiliary problem. The approximation error is given below.
\begin{proposition}\label{pro-au}
Let $\tilde{\bmf{m}}$ and $\tilde{\bmf{m}}_h$ be the solutions of \eqref{equ:LL-auxiliary-problem} and \eqref{auxiliary}, respectively.
Assume that $\tilde{\bmf{m}}$ satisfies the regularity condition \eqref{regularity-condition}.
For sufficiently small $h$, let $\tilde{\bmf{m}}_h^0=R_h \tilde{\bmf{m}}_0$. Then, we have
\begin{eqnarray*}
\|\tilde{\bmf{m}}-\tilde{\bmf{m}}_h\|_{\bmf{L}^2}
\leq
C h^2, \quad
\|\tilde{\bmf{m}}-\tilde{\bmf{m}}_h\|_{\bmf{H}^1}
\leq
C h,
\end{eqnarray*}
where $C>0$ is a constant independent of $h$, but depends on 
$\|\nabla \bmf{\Phi}\|_{\bmf{L}^{\infty}}$ 
and $C_{\mathrm{reg}}$.
\label{pro2}
\end{proposition}
\begin{proof}
Let $\tilde{\bmf{m}} - \tilde{\bmf{m}}_h = ( \tilde{\bmf{m}} - R_h \tilde{\bmf{m}}) + (R_h \tilde{\bmf{m}} - \tilde{\bmf{m}}_h) =: \bmf{\bmf{\eta}} + \bmf{\bmf{\xi}}$.
According to \eqref{equ:LL-auxiliary-problem} and \eqref{auxiliary}, we have the error equation
\begin{equation*}
(\partial_t\bmf{\xi},I_h^*\bmf{v}_h)
        +
        \mathcal{A}_h(\bmf{\Phi};\bmf{\xi},I_h^*\bmf{v}_h)
        =
        -
        (\partial_t\bmf{\eta},I_h^*\bmf{v}_h)
        -
        \mathcal{A}_h(\bmf{\Phi};\bmf{\eta},I_h^*\bmf{v}_h)
        +
        \alpha (|\nabla \bmf{\Phi}|^2 (\bmf{\xi} + \bmf{\eta}),I_h^*\bmf{v}_h).
\end{equation*}
By \eqref{R_h}, \eqref{erfun1} and \eqref{erfun-A},
we have
\begin{equation*}
    \begin{aligned}
       & \mathcal{A}_h(\bmf{\Phi};\bmf{\eta},I_h^*\bmf{v}_h)\nn
        =
       \mathcal{A}_h(\bmf{\Phi};\tilde{\bmf{m}},I_h^*\bmf{v}_h)
       -
       \mathcal{A}_h(\bmf{\Phi};R_h \tilde{\bmf{m}},I_h^*\bmf{v}_h)\nn\\
        =&
       -(\partial_t \tilde{\bmf{m}} - \alpha |\nabla\bmf{\Phi}|^2 \tilde{\bmf{m}}, I_h^*\bmf{v}_h)
       -
       \mathcal{A}(\bmf{\Phi};R_h \tilde{\bmf{m}},\bmf{v}_h)
       +
       [\mathcal{A}(\bmf{\Phi};R_h \tilde{\bmf{m}},\bmf{v}_h)-\mathcal{A}_h(\bmf{\Phi};R_h \tilde{\bmf{m}}, I_h^*\bmf{v}_h)]\nn\\
       =&
       -(\partial_t \tilde{\bmf{m}} - \alpha |\nabla\bmf{\Phi}|^2 \tilde{\bmf{m}}, I_h^*\bmf{v}_h)
       +
       (\partial_t \tilde{\bmf{m}} - \alpha |\nabla\bmf{\Phi}|^2 \tilde{\bmf{m}}, \bmf{v}_h)
       +
       [\mathcal{A}(\bmf{\Phi};R_h \tilde{\bmf{m}},\bmf{v}_h)-\mathcal{A}_h(\bmf{\Phi};R_h \tilde{\bmf{m}},I_h^*\bmf{v}_h)]\nn\\
        =&
       -(\partial_t \tilde{\bmf{m}} - \alpha |\nabla\bmf{\Phi}|^2 \tilde{\bmf{m}}, I_h^*\bmf{v}_h-\bmf{v}_h)
       +
       [\mathcal{A}(\bmf{\Phi};R_h \tilde{\bmf{m}},\bmf{v}_h)-\mathcal{A}_h(\bmf{\Phi};R_h \tilde{\bmf{m}} ,I_h^*\bmf{v}_h)]\nn\\
       =&
       \varepsilon_h(\partial_t \tilde{\bmf{m}} - \alpha |\nabla\bmf{\Phi}|^2 \tilde{\bmf{m}},\bmf{v}_h)
       +
       \varepsilon_A(\bmf{\Phi};R_h \tilde{\bmf{m}},\bmf{v}_h).
    \end{aligned}
\end{equation*}
Then, it holds
\begin{equation}
\begin{aligned}
        &(\partial_t\bmf{\xi},I_h^*\bmf{v}_h)
        +
        \mathcal{A}_h(\bmf{\Phi};\bmf{\xi},I_h^*\bmf{v}_h)\\
         = &
        \!-
        (\partial_t\bmf{\eta},I_h^*\bmf{v}_h)
        \!-\!
        \varepsilon_h(\partial_t \tilde{\bmf{m}} - \alpha |\nabla\bmf{\Phi}|^2 \tilde{\bmf{m}},\bmf{v}_h)
        \!-\!
       \varepsilon_A(\bmf{\Phi};R_h \tilde{\bmf{m}},\bmf{v}_h)
        +
        \alpha (|\nabla \bmf{\Phi}|^2 (\bmf{\xi} \!+\! \bmf{\eta}),I_h^*\bmf{v}_h).
\end{aligned}
\label{equ:error1}
\end{equation}
According to the definition of $\bmf{\xi}$, it has the form
\begin{equation*}
    \bmf{\xi} = \sum_{\mathbf{x}_i\in\mathcal{N}_h^0} \bmf{\xi}_h(\bmf{x}_i) \phi_i(\bmf{x}).
\end{equation*}
We have the equality
\begin{equation*}
    (\partial_t\bmf{\xi},I_h^*\bmf{\xi}) = \frac{1}{2}\partial_t\||\bmf{\xi}|\|_{0}^2.
\end{equation*}
Then, choose $\bmf{v}_h = \bmf{\xi}$ in \eqref{equ:error1},
and applying Lemma \ref{lem-Ah} and Lemma \ref{erfunction} yield
\begin{equation*}
    \begin{aligned}
        &\frac{1}{2} \partial_t \||\bmf{\xi}|\|_{0}^2
        +
        \alpha \|\bmf{\xi}\|_{\bmf{H}^1}^2\\
         \leq &
        -
        (\partial_t\bmf{\eta},I_h^*\bmf{\xi})
        -
        \varepsilon_h(\partial_t \tilde{\bmf{m}} - \alpha |\nabla\bmf{\Phi}|^2 \tilde{\bmf{m}},\bmf{\xi})
        -
       \varepsilon_A(\bmf{\Phi};R_h \tilde{\bmf{m}},\bmf{\xi})
        +
        \alpha (|\nabla \bmf{\Phi}|^2 (\bmf{\xi} + \bmf{\eta}),I_h^*\bmf{\xi})\\
        \leq &
        C \|\partial_t \bmf{\eta}\|_{\bmf{L}^2} \||\bmf{\xi}|\|_{0}
        +
        C h^2 \|\partial_t \tilde{\bmf{m}} - \alpha |\nabla\bmf{\Phi}|^2 \tilde{\bmf{m}}\|_{\bmf{H}^1} \|\bmf{\xi}\|_{\bmf{H}^1}
        +
        C h^2 \|\tilde{\bmf{m}}\|_{\bmf{H}^2} \|\bmf{\xi}\|_{\bmf{H}^1}\\
        & +
        C \|\nabla \bmf{\Phi}\|_{\bmf{L}^{\infty}} ( \||\bmf{\xi}|\|_{0}+\|\bmf{\eta}\|_{\bmf{L}^2} ) \|\bmf{\xi}\|_{\bmf{L}^2}\\
         \leq &
        C (\|\partial_t \bmf{\eta}\|_{\bmf{L}^2}^2
        + \epsilon^{-1} h^4 \|\partial_t \tilde{\bmf{m}} - \alpha |\nabla\bmf{\Phi}|^2 \tilde{\bmf{m}}\|_{\bmf{H}^1}^2
        + \epsilon^{-1} h^4 \|\tilde{\bmf{m}}\|_{\bmf{H}^2}^2
        + \|\bmf{\eta}\|_{\bmf{L}^2}^2)
        \\
        & +
        C \||\bmf{\xi}|\|_{0}^2
        +
        2 \epsilon \|\bmf{\xi}\|_{\bmf{H}^1}^2,
    \end{aligned}
\end{equation*}
in which we use the H\"{o}lder inequality and Young's inequality.

Applying the interpolation inequality \eqref{R_h_p} for $\bmf{\mathbf{\eta}}$ and owing to the assumption $\bmf{\Phi}\in\bmf{X}$, we obtain 
\begin{equation}\label{eta-L-2}
\begin{aligned}
    \|\partial_t \boldsymbol{\eta}\|_{\mathbf{L}^2}^2
        + \| \boldsymbol{\eta} \|_{\mathbf{L}^2}^2 
        &=  \|\partial_t ( \tilde{\mathbf{m}} - R_h \tilde{\mathbf{m}})\|_{\mathbf{L}^2}^2
        + \| \tilde{\mathbf{m}} - R_h \tilde{\mathbf{m}} \|_{\mathbf{L}^2}^2 \\
        &\leq C_0 (\|\partial_t \tilde{\mathbf{m}}\|_{\mathbf{H}^2}^2 + \|\tilde{\mathbf{m}}\|_{\mathbf{H}^2}^2) h^4,
\end{aligned}    
\end{equation}
and 
\begin{equation}\label{nolinear-f}
    h^4 \|\partial_t \tilde{\mathbf{m}} - \alpha |\nabla\mathbf{\Phi}|^2\tilde{\mathbf{m}}\|_{\mathbf{H}^1}^2
    \leq C_1 (\|\partial_t \tilde{\mathbf{m}}\|_{\mathbf{H}^1}^2 + \alpha \|\nabla\bmf{\Phi}\|_{\bmf{L}^{\infty}}^2\|\tilde{\mathbf{m}}\|_{\mathbf{H}^1}^2) h^4, 
\end{equation}
where $C_0$ and $C_1$ are independent of $h$. We now arrive at 
\begin{equation}\label{equ:error-L2-2}
    \frac{1}{2} \partial_t \||\mathbf{\xi}|\|_{0}^2
        +
        \alpha \|\mathbf{\xi}\|_{\mathbf{H}^1}^2
        \leq
          C_2 (\|\partial_t \tilde{\mathbf{m}}\|_{\mathbf{H}^2}^2 + \|\tilde{\mathbf{m}}\|_{\mathbf{H}^2}^2 + \alpha \|\nabla\bmf{\Phi}\|_{\bmf{L}^{\infty}}^2\|\tilde{\mathbf{m}}\|_{\mathbf{H}^1}^2) h^4
         +
        C \||\mathbf{\xi}|\|_{0}^2
        +
        2 \epsilon \|\mathbf{\xi}\|_{\mathbf{H}^1}^2,
\end{equation}
where $C_2$ only relies to $C_0$ and $C_1$.

Next, choosing $\epsilon = \alpha/2$ and using the Gronwall's lemma \cite{evans2022partial} yield
\begin{equation*}
 \||\bmf{\xi}|\|_{0} ^2
\leq
e^{CT} (\|\bmf{\xi}(0)\|_{0}^2 + C_2(\|\partial_t \tilde{\mathbf{m}}\|_{{L}^{\infty}(0, T; \mathbf{H}^2)}^2 + \|\tilde{\mathbf{m}}\|_{{L}^{\infty}(0, T; \mathbf{H}^2)}^2 + \alpha \|\nabla\bmf{\Phi}\|_{\bmf{L}^{\infty}}^2\|\tilde{\mathbf{m}}\|_{{L}^{\infty}(0, T; \mathbf{H}^1)}^2) T h^4).
\end{equation*}
Since $\tilde{\mathbf{m}}(0) = R_h \tilde{\mathbf{m}}(0)$, we have $\bmf{\xi}(0)=0$.
Hence, 
\begin{equation*}
  \||\bmf{\xi}|\|_{0} 
  \leq 
  C h^2. 
\end{equation*} 
Owing to \eqref{norm-equ}, we have
\begin{equation}\label{estimate-xi}
  \|\bmf{\xi}\|_{\bmf{L}^2} 
  \leq 
  C h^2.    
\end{equation}
In the above estimations, we have used the bounds of $\|\partial_t \tilde{\bmf{m}}\|_{{L}^{\infty}(0, T; \bmf{H}^2)}$ 
and $\|\tilde{\bmf{m}}\|_{{L}^{\infty}(0, T; \bmf{H}^2)}$, which have been given by the regularity assumption \eqref{regularity-condition}.
So, the constant $C$ in \eqref{estimate-xi} depends on $C_{\mathrm{reg}}$ and $\|\nabla \bmf{\Phi}\|_{\bmf{L}^{\infty}}$ but is independent of $h$.

Meanwhile, taking $\bmf{v}_h = \partial_t \bmf{\xi}$ in \eqref{equ:error1},
we also have
\begin{equation*}
    \begin{aligned}
        &(\partial_t\bmf{\xi},I_h^*\partial_t \bmf{\xi})
        +
        \mathcal{A}_h(\bmf{\Phi};\bmf{\xi},I_h^*\partial_t \bmf{\xi})\\
         = &
        \!-
        (\partial_t\bmf{\eta},I_h^*\partial_t \bmf{\xi})
        \!-\!
        \varepsilon_h(\partial_t \tilde{\bmf{m}} - \alpha |\nabla\bmf{\Phi}|^2 \tilde{\bmf{m}},\partial_t \bmf{\xi})
        \!-\!
       \varepsilon_A(\bmf{\Phi};R_h \tilde{\bmf{m}},\partial_t \bmf{\xi})
        \!+\!
        \alpha (|\nabla \bmf{\Phi}|^2 (\bmf{\xi} \!+\! \bmf{\eta}),I_h^*\partial_t \bmf{\xi}).
    \end{aligned}
\end{equation*}
Since
\begin{equation*}
\mathcal{A}_h(\bmf{\Phi};\bmf{\xi},I_h^*\partial_t \bmf{\xi})
\geq
\frac{\alpha}{2} \partial_t \|\bmf{\xi}\|_{\bmf{H}^1}^2
-
\frac{\alpha}{2}[\mathcal{A}_h(\bmf{\Phi};\partial_t\bmf{\xi},I_h^*\bmf{\xi})
-
\mathcal{A}_h(\bmf{\Phi};\bmf{\xi},I_h^*\partial_t\bmf{\xi})].
\end{equation*}
Applying Lemma \ref{erfunction} and Lemma \ref{lem-Ah},
we obtain
\begin{equation*}
\begin{aligned}
&\||\partial_t \bmf{\xi}|\|_{0}^2
+
\frac{\alpha}{2} \partial_t \|\bmf{\xi}\|_{\bmf{H}^1}^2\\
 \leq &
(\alpha |\nabla \bmf{\Phi}|^2(\bmf{\eta}+\bmf{\xi}),I_h^*\partial_t \bmf{\xi})
-
(\partial_t\bmf{\eta},I_h^*\partial_t \bmf{\xi})
-
\varepsilon_h(\partial_t \tilde{\bmf{m}} - \alpha |\nabla\bmf{\Phi}|^2 \tilde{\bmf{m}},\partial_t \bmf{\xi}) \\
& -
\varepsilon_A(\bmf{\Phi};R_h \tilde{\bmf{m}},\partial_t \bmf{\xi})
+
\frac{1}{2}[\mathcal{A}_h(\bmf{\Phi};\partial_t\bmf{\xi},I_h^*\bmf{\xi})
-
\mathcal{A}_h(\bmf{\Phi};\bmf{\xi},I_h^*\partial_t\bmf{\xi})]\\
\leq &
C \|\nabla\bmf{\Phi}\|_{\bmf{L}^{\infty}}^2 (\| \bmf{\eta} \|_{\bmf{L}^2} + \| \bmf{\xi} \|_{\bmf{L}^2}) \|| \partial_t \bmf{\xi} |\|_{0}
+
C \| \partial_t\bmf{\eta} \|_{\bmf{L}^2} \|| \partial_t \bmf{\xi} |\|_{0}\\
&+
C h \|\partial_t \tilde{\bmf{m}} - \alpha |\nabla\bmf{\Phi}|^2 \tilde{\bmf{m}}\|_{\bmf{H}^1} \|| \partial_t \bmf{\xi} |\|_{0}
+
C h \|\tilde{\bmf{m}}\|_{\bmf{H}^2} \|| \partial_t \bmf{\xi} |\|_{0}
+
C h \|\bmf{\xi}\|_{\bmf{H}^1} \|\partial_t\bmf{\xi}\|_{\bmf{H}^1}\\
\leq &
C \epsilon^{-1}\| \bmf{\eta} \|_{\bmf{L}^2}^2 + C\epsilon^{-1} \| \bmf{\xi} \|_{\bmf{L}^2}^2
+
\epsilon \|| \partial_t \bmf{\xi} |\|_{0}^2
+
C \epsilon^{-1} \| \partial_t\bmf{\eta} \|_{\bmf{L}^2}^2
+
\epsilon \|| \partial_t \bmf{\xi} |\|_{0}^2
+
C \epsilon^{-1} h^2 \|\partial_t \tilde{\bmf{m}} - \alpha |\bmf{\Phi}|^2 \tilde{\bmf{m}}\|_{\bmf{H}^1}^2 \\
& +
\epsilon \|| \partial_t \bmf{\xi} |\|_{0}^2
+
C \epsilon^{-1} h^2 \|\tilde{\bmf{m}}\|_{\bmf{H}^2}^2
+
\epsilon \|| \partial_t \bmf{\xi} |\|_{0}^2
+
C \epsilon^{-1} \|\bmf{\xi}\|_{\bmf{H}^1}^2
+
\epsilon \|| \partial_t \bmf{\xi} |\|_{0}^2\\
\leq &
C \epsilon^{-1} (\|\bmf{\eta}\|_{\bmf{L}^2}^2 + \|\partial_t \bmf{\eta}\|_{\bmf{L}^2}^2
+ h^2 \|\partial_t \tilde{\bmf{m}} - \alpha |\nabla\bmf{\Phi}|^2 \tilde{\bmf{m}}\|_{\bmf{H}^1}^2
+
h^2 \|\tilde{\bmf{m}}\|_{\bmf{H}^2}^2 
+ 
\|\bmf{\xi}\|_{\bmf{H}^1}^2)
+ 5 \epsilon \|| \partial_t \bmf{\xi} |\|_{0}^2,
\end{aligned}
\end{equation*}
in which the H\"{o}lder inequality, Young's inequality and inverse inequality are employed.
Similar to the proof of \eqref{equ:error-L2-2}, we also have 
\begin{equation*}
\begin{aligned}
\||\partial_t \mathbf{\xi}|\|_{0}^2
+
\frac{\alpha}{2} \partial_t \|\mathbf{\xi}\|_{\mathbf{H}^1}^2
\leq&
C_0 h^4 (\|\partial_t \tilde{\mathbf{m}}\|_{\mathbf{H}^2}^2 + \|\tilde{\mathbf{m}}\|_{\mathbf{H}^2}^2) 
+ 5 \epsilon \||\partial_t\mathbf{\xi}|\|_{0}^2
+
C \|\mathbf{\xi}\|_{\mathbf{H}^1}^2\\
&+
C_1 h^2 (\|\partial_t \tilde{\mathbf{m}}\|_{\mathbf{H}^1}^2 + \alpha \|\nabla\bmf{\Phi}\|_{\bmf{L}^{\infty}}^2\|\tilde{\mathbf{m}}\|_{\mathbf{H}^1}^2)\\
\leq &
C_3 h^2 (h^2\|\partial_t \tilde{\mathbf{m}}\|_{\mathbf{H}^2}^2 + h^2\|\tilde{\mathbf{m}}\|_{\mathbf{H}^2}^2
+
\|\partial_t \tilde{\mathbf{m}}\|_{\mathbf{H}^1}^2 + \alpha \|\nabla\bmf{\Phi}\|_{\bmf{L}^{\infty}}^2\|\tilde{\mathbf{m}}\|_{\mathbf{H}^1}^2) \\
&+ 
5 \epsilon \||\partial_t\mathbf{\xi}|\|_{0}^2
+
C \|\mathbf{\xi}\|_{\mathbf{H}^1}^2\\
\leq &
C_4 h^2 (\|\partial_t \tilde{\mathbf{m}}\|_{\mathbf{H}^2}^2 + \|\tilde{\mathbf{m}}\|_{\mathbf{H}^2}^2
+ \alpha \|\nabla\bmf{\Phi}\|_{\bmf{L}^{\infty}}^2\|\tilde{\mathbf{m}}\|_{\mathbf{H}^1}^2) \\
&+ 
5 \epsilon \||\partial_t\mathbf{\xi}|\|_{0}^2
+
C \|\mathbf{\xi}\|_{\mathbf{H}^1}^2,    
\end{aligned}
\end{equation*}
where $C_4$ only relies to $C_0$ and $C_1$.

Choosing $\epsilon = 1/5$ and using the Gronwall's lemma yield 
\begin{equation*}
\|\bmf{\xi}(t)\|_{\bmf{H}^1}^2
\leq
e^{CT} (\|\bmf{\xi}(0)\|_{\bmf{H}^1}^2 + C_4(\|\partial_t \tilde{\mathbf{m}}\|_{L^{\infty}(0,T;\mathbf{H}^2)}^2 + \|\tilde{\mathbf{m}}\|_{L^{\infty}(0,T;\mathbf{H}^2)}^2
+ \alpha \|\nabla\bmf{\Phi}\|_{\bmf{L}^{\infty}}^2\|\tilde{\mathbf{m}}\|_{L^{\infty}(0,T;\mathbf{H}^1)}^2) T h^2).
\end{equation*}
Since $\bmf{\xi}(0)=0$, we have
\begin{equation*}
\|\bmf{\xi}(t)\|_{\bmf{H}^1} \leq C h.
\end{equation*}
Here we used the bounds of $\|\partial_t \tilde{\bmf{m}}\|_{L^{\infty}(0,T;\bmf{H}^2)}$,
$\|\nabla \bmf{\Phi}\|_{\bmf{L}^{\infty}}$,
and $\|\tilde{\bmf{m}}\|_{L^{\infty}(0,T;\bmf{H}^2)}$.  
Hence, the constant $C$ depends on $\|\nabla \bmf{\Phi}\|_{\bmf{L}^{\infty}}$ and $C_{\mathrm{reg}}$ as well. At last, 
using \eqref{R_h_p} and the triangle inequality yields Proposition \ref{pro2}, which completes the proof.
\end{proof}

Consequently,
the approximation error of the FVEM applied to the LL equation can be readily obtained.
\begin{theorem}
Let $\bmf{m}$ and $\bmf{m}_h$ be the solutions of \eqref{equ:simplicity-LL-dimensionless} with homogeneous Dirichlet boundary conditions and \eqref{equ:LL-semi-discretized1}, respectively.
Assume the regularity condition \eqref{regularity-condition} holds.
For sufficiently small $h$, let $\bmf{m}_h^0 = I_h \bmf{m}_0$.
Then, we have
\begin{equation}
    \|\bmf{m}-\bmf{m}_h\|_{\bmf{L}^2}
    \leq
    C h^2
    \label{error-L2}
\end{equation}
and
\begin{equation}
    \|\bmf{m} - \bmf{m}_h\|_{\bmf{H}^1}
    \leq
    C h,
    \label{error-H1}
\end{equation}
where $ C > 0 $ is independent of $h$.
\label{thm-error}
\end{theorem}

\begin{proof}
Recall the auxiliary problem
\begin{equation*}
    \partial_t\tilde{\bmf{m}}
    -
    \alpha \Delta\tilde{\bmf{m}}
    +
    \bmf{\Phi} \times \Delta \tilde{\bmf{m}}
    +
    \nabla\bmf{\Phi} \times \nabla \tilde{\bmf{m}}
    =
    \alpha |\nabla\bmf{\Phi}|^2 \tilde{\bmf{m}}.
\end{equation*}
Let $\bmf{\Phi} = \tilde{\bmf{m}}$, and it becomes the LL equation. Set $\bmf{\Phi} = \bmf{m}^l$ and construct the iteration equation over the time interval $[T_1, T_2]$ as \eqref{equ:LL-linearized}. Then, by Lemma \ref{lem-con}, the sequence $\{ {\bmf{m}}^l\}$ converges in $\bmf{H}^2$. Denote $\bmf{m}$ the limit of the sequence, and there holds
\begin{equation*}
\lim_{l\to\infty}\| {\bmf{m}}^l - \bmf{m}\|_{\bmf{H}^2} = 0.
\end{equation*}
For the equation \eqref{equ:LL-linearized}, we construct the discrete weak form
\begin{equation}\label{equ:weak-form-iteration}
\big( \partial_t {\bmf{m}}_h^{l+1},
I_h^* \bmf{v}_h \big)
+ \mathcal{A}_h\big( {\bmf{m}}_h^l;  {\bmf{m}}_h^{l+1}, I_h^* \bmf{v}_h\big)
= \alpha\big(|\nabla {\bmf{m}}_h^l|^2  {\bmf{m}}_h^{l+1}, I_h^* \bmf{v}_h\big), \qquad \forall \bmf{v}_h \in \bmf{S}_h,
\end{equation}
where the initial value is given by $\bmf{m}_h(\cdot, T_1)$.
Similarly, it satisfies
\begin{equation*}
\lim_{l\to\infty}\| {\bmf{m}}_h^l - \bmf{m}_h\|_{\bmf{H}^2} = 0,
\quad
\|\nabla {\bmf{m}}_h^l\|_{\bmf{L}^\infty}\leq C.
\end{equation*}

Next, for any given $l\ge 0$, we have known that $\|\nabla\bmf{m}^l\|_{\bmf{L}^\infty}\leq C$. And then, for the problem
\begin{equation}\label{equ:weak-form-iteration2}
\big( \partial_t\tilde{\bmf{m}}_h^{l+1},
I_h^* \bmf{v}_h \big)
+ \mathcal{A}_h\big( {\bmf{m}}^l; \tilde{\bmf{m}}_h^{l+1}, I_h^* \bmf{v}_h\big)
= \alpha\big(|\nabla {\bmf{m}}^l|^2 \tilde{\bmf{m}}_h^{l+1}, I_h^* \bmf{v}_h\big), \qquad \forall \bmf{v}_h \in \bmf{S}_h,
\end{equation}
using the proposition \ref{pro-au} yields
\begin{equation*}
\| {\bmf{m}}^{l+1} - \tilde{\bmf{m}}_h^{l+1}\|_{\bmf{L}^2} \le C h^2,
\quad
\| {\bmf{m}}^{l+1} - \tilde{\bmf{m}}_h^{l+1}\|_{\bmf{H}^1} \le C h,    
\end{equation*}
where $C>0$ depends only on $\|\nabla\bmf{m}^l\|_{\bmf{L}^{\infty}}$ and $C_{\rm reg}$ and is independent of $h$.
We now decompose the total error into
\begin{equation*}
\bmf{m} - \bmf{m}_h
= (\bmf{m} - {\bmf{m}}^l)
+ ({\bmf{m}}^l - {\bmf{m}}_h^l)
+ ({\bmf{m}}_h^l - \bmf{m}_h),
\end{equation*}
in which ${\bmf{m}}^l - {\bmf{m}}_h^l = ({\bmf{m}}^l - \tilde{\bmf{m}}_h^l) + (\tilde{\bmf{m}}_h^l - {\bmf{m}}_h^l)$. Denote $\bmf{\rho}^l := {\bmf{m}}^l - \tilde{\bmf{m}}_h^l$ and $\bmf{\theta}^l = \tilde{\bmf{m}}_h^l - {\bmf{m}}_h^l$, and we only need to estimate the error of
$\bmf{\theta}^l$.

Subtract \eqref{equ:weak-form-iteration} from \eqref{equ:weak-form-iteration2} and we have
	\begin{equation}
		(\partial_t\bmf{\theta}^{l+1}, I_h^*\bmf{v}_h)
		+
		\mathcal{A}_h(\bmf{m}^l;\bmf{\theta}^{l+1}, I_h^*\bmf{v}_h)
		=
		\mathcal{R}_1(\bmf{v}_h)
		+
		\mathcal{R}_2(\bmf{v}_h),
		\qquad \forall\, \bmf{v}_h\in \bmf{S}_h,
		\label{thm37-theta-eq}
	\end{equation}
	where
	\begin{equation}
		\mathcal{R}_1(\bmf{v}_h)
		:=
		-b_h(\bmf{m}^l;\bmf{m}_h^{l+1},I_h^*\bmf{v}_h)
		+
		b_h(\bmf{m}_h^l;\bmf{m}_h^{l+1},I_h^*\bmf{v}_h),
		\label{thm37-R1}
	\end{equation}
    and
	\begin{equation}
		\mathcal{R}_2(\bmf{v}_h)
		:=
		\alpha (|\nabla\bmf{m}^l|^2\tilde{\bmf{m}}_h^{l+1}-|\nabla\bmf{m}_h^l|^2\bmf{m}_h^{l+1}, I_h^*\bmf{v}_h).
		\label{thm37-R2}
	\end{equation}
	Choosing $\bmf{v}_h=\bmf{\theta}^{l+1}$ in \eqref{thm37-theta-eq}, and using the Lemma \ref{lem-Ah}, we have
	\begin{equation}
    \begin{aligned}
    	&\frac12\partial_t \||\bmf{\theta}^{l+1}|\|_0^2
		+
		\alpha\|\bmf{\theta}^{l+1}\|_{\bmf{H}^1}^2 \\
		\leq&
        C \|\bmf{m}_h^l-\bmf{m}^l\|_{\bmf{L}^2}\|\bmf{m}_h^{l+1}\|_{\bmf{L}^{\infty}}\|\bmf{\theta}^{l+1}\|_{\bmf{H}^1} 
        + 
        \alpha (|\nabla\bmf{m}^l|^2\bmf{\theta}^{l+1}, I_h^*\bmf{\theta}^{l+1})\\
        \leq
        & \frac{C}{2\epsilon} \|\bmf{m}_h^l-\bmf{m}^l\|^2_{\bmf{L}^2}\|\bmf{m}_h^{l+1}\|_{\bmf{L}^{\infty}}^2 +
        \frac{\epsilon}{2} \|\bmf{\theta}^{l+1}\|_{\bmf{H}^1}^2 +
        \alpha \|\nabla\bmf{m}^l\|_{\bmf{L}^{\infty}}^2 \||\bmf{\theta}^{l+1}|\|_0^2,
    \end{aligned}
		\label{thm37-basic-energy}
	\end{equation}
	where we used the assumption \eqref{working_set}, the regularity condition \eqref{regularity-condition}, and Young's inequality. Choose $\epsilon = 2\alpha$ and we arrive at
    \begin{align*}
        \frac12\partial_t \||\bmf{\theta}^{l+1}|\|_0^2
        &\leq
        C \big(\|\bmf{m}_h^l-\bmf{m}^l\|^2_{\bmf{L}^2} + \||\bmf{\theta}^{l+1}|\|_0^2 \big) \\
        &\leq C \big(\|\bmf{\rho}^{l}\|_{\bmf{L}^{2}}^2 + \||\bmf{\theta}^{l}|\|_0^2 + \||\bmf{\theta}^{l+1}|\|_0^2 \big).
    \end{align*}
    Then, taking summation with respect to $l$, and using Gronwall's lemma yield
\begin{equation*}
\sum_{k=0}^l \||\bmf{\theta}^{k}|\|_0^2 \leq e^{CT}(Ch^4 + C\||\bmf{\theta}^0|\|_{0}^2) . 
\end{equation*}
Since $\bmf{\theta}^0=0$, we have
\begin{equation*}
\sum_{k=0}^l \||\bmf{\theta}^{k}|\|_0^2 \leq C h^4 , 
\end{equation*}
which implies that $\|\bmf{\theta}^{l}\|_{\bmf{L}^2} \leq C h^2$.

As for the $\bmf{H}^1$ estimation, taking $\bmf{v}_h = \partial_t \bmf{\theta}^{l+1}$ in \eqref{thm37-theta-eq} yields
	\begin{equation}
    \begin{aligned}
    	&\||\partial_t\bmf{\theta}^{l+1}|\|_0^2
		+
		\frac{\alpha}{2}\partial_t\|\bmf{\theta}^{l+1}\|_{\bmf{H}^1}^2 \\
		\leq&
        C h\|\bmf{\theta}^{l+1}\|_{\bmf{H}^1}\|\partial_t\bmf{\theta}^{l+1}\|_{\bmf{H}^1}
        + C(\|\bmf{\rho}^{l}\|_{\bmf{L}^2}+\|\bmf{\theta}^l\|_{\bmf{L}^2})\|\partial_t\bmf{\theta}^{l+1}\|_{\bmf{H}^1} \\
        & + C\|\nabla\bmf{m}^l\|_{\bmf{L}^{\infty}}\|\bmf{\theta}^{l+1}\|_{\bmf{L}^2}\|\partial_t \bmf{\theta}^{l+1}\|_{\bmf{L}^2} \\
        \leq&
        C \|\bmf{\theta}^{l+1}\|_{\bmf{H}^1}\|\partial_t\bmf{\theta}^{l+1}\|_{\bmf{L}^2} +
        Ch^{-1} (\|\bmf{\rho}^{l}\|_{\bmf{L}^2}+\|\bmf{\theta}^l\|_{\bmf{L}^2})\|\partial_t\bmf{\theta}^{l+1}\|_{\bmf{L}^2} \\
        & + C \|\bmf{\theta}^{l+1}\|_{\bmf{L}^2}\|\partial_t \bmf{\theta}^{l+1}\|_{\bmf{L}^2} \\
        \leq&
        C \|\bmf{\theta}^{l+1}\|_{\bmf{H}^1}\||\partial_t\bmf{\theta}^{l+1}|\|_{0} +
        Ch^{-1} (\|\bmf{\rho}^{l}\|_{\bmf{L}^2}+\|\bmf{\theta}^l\|_{\bmf{L}^2})\||\partial_t\bmf{\theta}^{l+1}|\|_{0} \\
        & + C \|\bmf{\theta}^{l+1}\|_{\bmf{L}^2}\||\partial_t \bmf{\theta}^{l+1}|\|_{0} \\
        \leq &
        \frac{C}{2\epsilon_1} \|\bmf{\theta}^{l+1}\|_{\bmf{H}^1}^2 + \frac{\epsilon_1}{2}\||\partial_t\bmf{\theta}^{l+1}|\|_{0}^2 +
        \frac{C}{2\epsilon
        _2} (h^{-2}\|\bmf{\rho}^{l}\|_{\bmf{L}^2}^2 + \|\bmf{\theta}^l\|_{\bmf{H}^1}^2) \\
        & + \frac{\epsilon_2}{2} \||\partial_t\bmf{\theta}^{l+1}|\|_{0}^2 + \frac{C}{2\epsilon_3}\|\bmf{\theta}^{l+1}\|_{\bmf{L}^2}^2 + \frac{\epsilon_3}{2}\||\partial_t \bmf{\theta}^{l+1}|\|_{0}^2,
    \end{aligned}
		\label{thm37-basic-energy-1}
	\end{equation}
    in which the H\"older inequality, Young’s inequality and inverse inequality are used. Then, choose $\epsilon_1 + \epsilon_2 + \epsilon_3 = 2$ and we arrive at
    \begin{align*}
        \frac{\alpha}{2}\partial_t\|\bmf{\theta}^{l+1}\|_{\bmf{H}^1}^2
        \leq & C\big(  \|\bmf{\theta}^{l+1}\|_{\bmf{H}^1}^2 +
         h^{-2}\|\bmf{\rho}^{l}\|_{\bmf{L}^2}^2 + \|\bmf{\theta}^l\|_{\bmf{H}^1}^2 \big) \\
         \leq &
         C\big( h^2 + \|\bmf{\theta}^{l+1}\|_{\bmf{H}^1}^2 +
         \|\bmf{\theta}^l\|_{\bmf{H}^1}^2 \big).
    \end{align*}
    Taking summation with respect to $l$, and using Gronwall's lemma yield
\begin{equation*}
\sum_{k=0}^l \|\bmf{\theta}^{k}\|_{\bmf{H}^1}^2 \leq e^{CT}(Ch^4 + C\|\bmf{\theta}^0\|_{H^1}^2). 
\end{equation*}
Owing to $\bmf{\theta}^0=0$, 
\begin{equation*}
\sum_{k=0}^l \|\bmf{\theta}^{k}\|_{\bmf{H}^1}^2 \leq C h^2 , 
\end{equation*}
which also implies $\|\bmf{\theta}^{l}\|_{\bmf{H}^1} \leq C h$.

Next, using the triangle inequality produces
\begin{equation*}
\begin{aligned}
\|\bmf{m} - \bmf{m}_h\|_{L^2}
&\leq \|\bmf{m} - \tilde{\bmf{m}}^l\|_{\bmf{L}^2}
+ \|\tilde{\bmf{m}}^l - \tilde{\bmf{m}}_h^l\|_{\bmf{L}^2}
+ \|\tilde{\bmf{m}}_h^l - \bmf{m}_h\|_{\bmf{L}^2},
\\
\|\bmf{m} - \bmf{m}_h\|_{\bmf{H}^1}
&\leq \|\bmf{m} - \tilde{\bmf{m}}^l\|_{\bmf{H}^1}
+ \|\tilde{\bmf{m}}^l - \tilde{\bmf{m}}_h^l\|_{\bmf{H}^1}
+ \|\tilde{\bmf{m}}_h^l - \bmf{m}_h\|_{\bmf{H}^1}.
\end{aligned}    
\end{equation*}
Let $l \to \infty$, and the first and third components of the above equations vanish, while the second terms remain bounded. We therefore reach
\begin{equation*}
\|\bmf{m} - \bmf{m}_h\|_{\bmf{L}^2} \le C h^2,
\quad
\|\bmf{m} - \bmf{m}_h\|_{\bmf{H}^1} \le C h.
\end{equation*}
At last, divide the whole interval $[0,T]$ into $N = T/\tau$ subintervals with step length $\tau$. On each subinterval, the initial error is at most $O(h^2)$ in $\bmf{L}^2$ and $O(h)$ in $\bmf{H}^1$ with a uniform constant $C$. By mathematical induction, the error estimates hold true on the entire interval $[0,T]$, which completes the proof.
\end{proof}

\begin{remark}
In many applications, the nonhomogeneous Dirichlet boundary condition $\bmf{m}|_{\partial\Omega} = \bmf{g}$, where $|\bmf{g}| = 1$ in a point-wise sense, is considered.
We emphasize that in this work, the error analysis is conducted for the LL equation and the auxiliary problem introduced with the homogeneous boundary condition.
Nevertheless, this analysis can be extended to the nonhomogeneous case, where the boundary condition is given by the function $\bmf{m}|_{\partial\Omega} = \bmf{g}$, provided that $\bmf{g}$ has sufficient regularity.
\end{remark}

At the end of the error estimate, we remark that the model considered in the above sections is a degenerate form of the full LL equation, since only the exchange field is included in the effective field, and the length constraint is omitted as well. Meanwhile, when the nonlinear problem is analyzed, a fixed-point iteration is introduced, in which the time derivative is discretized using a linearized implicit Euler method in the iteration. Hence, the analysis in this work can also be extended to the fully discretized form.

Next, we consider the energy law based on the FVEM discretization. Since the continuous energy law does not hold for the degenerate equation, we will consider the full LL equation in the sequel.

\subsection{Discretized energy dissipation}
Taking inner product of \eqref{LL4} with $-\bmf{h}$ yields
\begin{align*}
    -(\partial_t\bmf{m}, \bmf{h})
    =
    (\bmf{m} \times \bmf{h}, \bmf{h})
    +
    \alpha(\bmf{m} \times(\bmf{m} \times \bmf{h}), \bmf{h})
    =
    -\alpha(\bmf{m} \times \bmf{h})^2.
\end{align*}
Since
\begin{align*}
    -(\partial_t\bmf{m}, \bmf{h})
    & =
    -\int_{\Omega}\partial_t\bmf{m}\cdot(\epsilon\Delta \bmf{m}-q(m_2\bmf{e_2}+m_3\bmf{e_3})-m_3 \bmf{e_3}+ \bmf{h}_{\mr{e}}) \, \mr{d}x\mr{d}y \\
    & =
     \partial_t F[\bmf{m}] - \int_{\partial\Omega} \partial_t\bmf{m}\cdot(\nabla\bmf{m}\cdot\bmf{n}) \,\mr{d}S,
\end{align*}
the energy dissipation can be maintained when
\begin{equation}
    \int_{\partial\Omega} \partial_t\bmf{m}\cdot(\nabla\bmf{m}\cdot\bmf{n}) \,\mr{d}S = 0. \label{equ:general-boundary-condtion}
\end{equation}
This indicates that the energy dissipation law holds for systems with homogeneous Neumann boundary conditions and time-independent Dirichlet boundary conditions.

To conveniently illustrate the discrete energy dissipation by the FVEM, we rewrite the LL equation \eqref{LL4} into the Landau-Lifshitz-Gilbert (LLG) form
\begin{equation*}
\alpha \,\partial_t \bmf{m} + \bmf{m} \times \partial_t \bmf{m} - (1+\alpha^2)\bmf{h}
= -(1+\alpha^2) (\bmf{m}\cdot\bmf{h})\bmf{m}.
\end{equation*}
Define the discrete norm
\begin{equation}
\||\bmf{m}_h|\|_h^2 := \sum_{K_Q \in T_h} |\bmf{m}_h|_{\bmf{H}^1(K_Q)}^2,
    \label{discrete-norm}
\end{equation}
the magnetic free energy is therefore discretized into
\begin{equation*}
F[\bmf{m}_h]
=
\frac{\epsilon}{2} \||\bmf{m}_h|\|_h^2
+
(\hat{F}(\bmf{m}_h),1),
\end{equation*}
where
\begin{equation*}
    \hat{F}(\bmf{m}_h) = \frac{q}{2} (|\bmf{m}_{2h}|^2 +|\bmf{m}_{3h}|^2)
-
\bmf{h}_e\cdot \bmf{m}_h
+
\frac{1}{2} |\bmf{m}_{3h}|^2.
\end{equation*}
The numerical solution $\bmf{m}_h$ can also be solved from the weak form of the LLG equation
\begin{align}
&\alpha(\partial_t\bmf{m}_h, I_h^*\bmf{v}_h)
 +
 (\bmf{m}_h \times \partial_t \bmf{m}_h, I_h^*\bmf{v}_h)
 +
 \epsilon(1+\alpha^2) (c_h(\bmf{m}_h, I_h^*\bmf{v}_h)
 -d_h(\bmf{m}_h,I_h^*\bmf{v}_h)) \nonumber\\
 = &
(1+\alpha^2)[(\epsilon|\nabla\bmf{m}_h|^2\bmf{m}_h,I_h^*\bmf{v}_h)
 +
 (\hat{\bmf{f}}(\bmf{m}_h)
 \!-\!
 (\bmf{m}_h\!\cdot\!\hat{\bmf{f}}(\bmf{m}_h))\bmf{m}_h,I_h^*\bmf{v}_h)],
\label{equ:LL-semi-discretized-energy}
\end{align}
where
\begin{equation*}
 c_h(\bmf{m}_h,I_h^*\bmf{v}_h)
 =
 - \sum_{\bmf{x}_i\in\mathcal{N}_h^0}\!\int_{\partial V_i}\nabla \bmf{m}_h\cdot\bmf{n}I_h^*\bmf{v}_h \,\mr{d}S
 =
 - \sum_{\bmf{x}_i\in\mathcal{N}_h^0}\bmf{v}_h(\bmf{x}_i)\int_{\partial
V_i}\nabla \bmf{m}_h\cdot\bmf{n}\,\mr{d}S,
\end{equation*}
\begin{equation*}
    d_h(\bmf{m}_h,I_h^*\bmf{v}_h) = \int_{\partial \Omega}\nabla\bmf{m}_h\cdot\bmf{n}I_h^*\bmf{v}_h\,\mr{d}S,
\end{equation*}
and
\begin{equation*}
   \hat{\bmf{f}}(\bmf{m}_h) = -q(m_{2h}\bmf{e_2}+m_{3h}\bmf{e_3})-m_{3h} \bmf{e_3}+ \bmf{h}_{\mr{e}}.
\end{equation*}
For the term $c_h(\bmf{m}_h,I_h^*\bmf{v}_h)$, we have the following estimate.

\begin{lemma}
For the sufficiently small $h$, there exists a positive constant $C$ such that
for all $\bmf{m}_h, \partial_t\bmf{m}_h\in \bmf{S}_h$, it holds
\begin{equation*}
    c_h(\bmf{m}_h,I_h^*\partial_t\bmf{m}_h)
    \geq
    C \partial_t \||\bmf{m}_h|\|_{h}^2.
\end{equation*}
\label{lem-t}
\end{lemma}
The proof is detailed in Appendix \ref{app:lem-ah-t}.
Consequently, we obtain the discrete energy law as follows.
\begin{proposition}
    Denote $\bmf{m}_h^0 \in \bmf{S}_h$ the initial condition. Assume that the boundary condition \eqref{equ:general-boundary-condtion} holds true and the discrete magnetization $\bmf{m}_h$ satisfies $|\bmf{m}_h| = 1$ in a point-wise sense. We then obtain
    $d_tF[\bmf{m}_h] \leq 0$.
    \label{energy-dissipation}
\end{proposition}
\begin{proof}
Choose $\bmf{v}_h = \partial_t \bmf{m}_h$ in \eqref{equ:LL-semi-discretized-energy}, and
we have
\begin{equation*}
\begin{aligned}
&\alpha(\partial_t\bmf{m}_h, I_h^*\partial_t \bmf{m}_h)
 +
 (\bmf{m}_h \times \partial_t \bmf{m}_h, I_h^*\partial_t \bmf{m}_h)
 +
 \epsilon(1+\alpha^2) (c_h(\bmf{m}_h, I_h^*\partial_t \bmf{m}_h)-d_h(\bmf{m}_h,I_h^*\partial_t \bmf{m}_h))\\
 = &
(1+\alpha^2)[(\epsilon|\nabla\bmf{m}_h|^2\bmf{m}_h,I_h^*\partial_t \bmf{m}_h)
 +
 (\hat{\bmf{f}}(\bmf{m}_h)
 \!-\!
 (\bmf{m}_h\!\cdot\!\hat{\bmf{f}}(\bmf{m}_h))\bmf{m}_h,I_h^*\partial_t \bmf{m}_h)].
\end{aligned}
\end{equation*}
Owing to the Lemma \ref{lem-t} and the fact that $\bmf{m}_h \cdot \partial_t \bmf{m}_h =0$, it is obvious that
\begin{equation*}
\alpha\|\partial_t\bmf{m}_h\|_{\bmf{L}^2}^2
+
 \epsilon(1+\alpha^2) \partial_t \||\bmf{m}_h|\|_{h}^2
 \leq
 (1+\alpha^2)((\hat{\bmf{f}}(\bmf{m}_h),I_h^*\partial_t \bmf{m}_h)+d_h(\bmf{m}_h,I_h^*\partial_t \bmf{m}_h)).
\end{equation*}
According to the definition of discrete energy,
we get
\begin{equation*}
 \partial_t F[\bmf{m}_h]
=
\epsilon \partial_t \||\bmf{m}_h|\|_h^2
-
(\bmf{\hat{f}}(\bmf{m}_h),I_h^*\partial_t\bmf{m}_h),
\end{equation*}
which further induces that
\begin{equation*}
 \partial_t F[\bmf{m}_h] \leq - \frac{\alpha}{1+\alpha^2} \|\partial_t\bmf{m}_h\|_{\bmf{L}^2}^2
 +
 d_h(\bmf{m}_h,I_h^*\partial_t \bmf{m}_h) \leq 0.
\end{equation*}
This completes the proof of the discrete energy dissipation law.
\end{proof}

\section{Temporal discretization and numerical experiments\label{sec4: Temporal discretization and numerical experiments}}
\label{sec4}

Among the currently available time-marching methods for the LL equation, one of the most efficient is the GSPM proposed in 2001.
In the first-order accurate version,
only heat equations with constant coefficients are solved per time step.
This gives advantages in the application of FEM and FVEM.

\subsection{Temporal discretization}
GSPM solves the LL equation~\eqref{LL4} into the following procedures:
\begin{align}
    &\partial_t\bmf{m}^* = \bmf{h},
    \label{equ:GSPM_step1} \\
    &\partial_t\bmf{m} =
    -\bmf{m}\times\partial_t\bmf{m}^*
    - \alpha\bmf{m}\times\left(\bmf{m}\times \partial_t\bmf{m}^*\right).
    \label{equ:GSPM_step2}
\end{align}
The components of the nonlinear LL equation are decoupled using the above time-splitting step and are updated in the Gauss-Seidel manner within a time step.

We solve the heat equation discretized by the FVEM:
Find $\bmf{m}_h \in \bmf{S}_h$ for control volume $V_i$ such that
\begin{equation}
    \int_{V_i} \partial_t\bmf{m}^* \, \mr{d}x\mr{d}y
    - \epsilon\int_{\partial V_i}\nabla\bmf{m}_h\cdot \bmf{n}_i \, \mr{d}S
    = \int_{V_i}\hat{\bmf{f}}(\bmf{m}_h) \, \mr{d}x\mr{d}y,
    \label{equ:FVEM-heat-equ}
\end{equation}
where $\bmf{n}_i$ is the unit outward normal vector of $V_i$.
Multiplying \eqref{equ:FVEM-heat-equ} by $\bmf{v}_h(\bmf{x}_i)$ and taking summation of $\bmf{x}_i\in\mathcal{N}_h$ yields
\begin{equation}
 \left( \partial_t\bmf{m}_h,I_h^*\bmf{v}_h\right)
 +
 a(\bmf{m}_h,I_h^*\bmf{v}_h)
 =
 (\hat{\bmf{f}}(\bmf{m}_h), I_h^*\bmf{v}_h),
\label{2.8}
\end{equation}
where
\begin{equation}
 a(\bmf{m}_h,I_h^*\bmf{v}_h)
 =
 -\epsilon \!\sum_{\bmf{x}_i\in\mathcal{N}_h}\!\int_{\partial V_i}\!(\nabla \bmf{m}_h)\cdot\bmf{n}I_h^*\bmf{v}_h \, \mr{d}S
 =
 - \epsilon \!\sum_{\bmf{x}_i\in\mathcal{N}_h}\bmf{v}_h(\bmf{x}_i)\int_{\partial V_i}\!(\nabla \bmf{m}_h)\cdot\bmf{n} \, \mr{d}S.
\label{2.9}
\end{equation}

The semi-implicit Euler method applied, the approach that combines the GSPM and FVEM (GSPM-FVEM) is:
\begin{itemize}
    \item Find $ \bmf{m}_h^{*} \in \bmf{S}_h $ such that
\begin{align}
    &\left(\frac{{m}_{i}^* - {m}_{i}^n}{\Delta t},I_h^*{v}_{h}\right)
    + \epsilon a({m}_{i}^*,I_h^*{v}_{h}) = (\hat{f}_i(\bmf{m}_{h}^n),I_h^*{v}_{h}), \quad i = 1,2,3.
    \label{equ:GSPM1}
\end{align}
\item Evolve the ODE with the Gauss-Seidel update,
\begin{align}
\hat{m}_{1}^{n+1}
&=
{m}_{1}^{n}\!
-\!({m}_{2}^{n}{m}_{3}^{*}\!-\!{m}_{3}^{n}{m}_{2}^{*})\!
-\!\alpha({m}_{1}^{n}{m}_{1}^{*}\!+\!{m}_{2}^{n}{m}_{2}^{*}\!+\!{m}_{3}^{n}{m}_{3}^{*}){m}_{1}^{n}\!
+\!\alpha {m}_{1}^{*}, \nonumber\\
\hat{m}_{2}^{n+1}
&=
{m}_{2}^{n}\!
-\!({m}_{3}^{n}{g}_{1}\!-\!\hat{m}_{1}^{n+1}{m}_{3}^{*})\!
-\!\alpha(\hat{m}_{1}^{n+1}{g}_{1}\!+\!{m}_{2}^{n}{m}_{2}^{*}\!+\!{m}_{3}^{n}{m}_{3}^{*}){m}_{2}^{n}\!
+\!\alpha {m}_{2}^{*}, \label{2.17}\\
\hat{m}_{3}^{n+1}
&=
{m}_{3}^{n}\!
-\!(\hat{m}_{1}^{n+1}{g}_2\!-\!\hat{m}_{2}^{n+1}{g}_1)\!
-\!\alpha(\hat{m}_{1}^{n+1}{g}_1\!+\!\hat{m}_{2}^{n+1}{g}_2\!+\!{m}_{3}^{n}{m}_{3}^{*}){m}_{3}^{n}\!
+\!\alpha {m}_{3}^{*}, \nonumber
\end{align}
where $ {g}_1$ and $ {g}_2 $ are the solutions to
\begin{equation}\label{equ:GSPM1-2}
    \left(\frac{{g}_i-{m}_{i}^{n+1}}{\Delta t},I_h^{*}v_{h}\right)
    +
    \epsilon a({g}_{i},I_h^{*}v_{h})
    =
    (\hat{f}_i(\bmf{m}_{h}^{*}),I_h^{*}v_{h}),
    \quad i=1,2.
\end{equation}
\item Projection onto $\mathcal{S}^2$,
\begin{equation}
    \bmf{m}^{n+1}
    =
    \frac{\bmf{\hat{m}_h}}{|\bmf{\hat{m}_h}|}.
\end{equation}
\end{itemize}
As shown in the above algorithm, the heat equations with constant coefficients are solved at each time step in \eqref{equ:GSPM1} and \eqref{equ:GSPM1-2}, in which the degrees of freedom (dofs) are $N$. Meanwhile, we consider the linearized backward Euler method (such as \eqref{equ:LL-linearized}) for time-stepping and the FVEM in space, a discrete system with variable coefficients and $3N$ dofs must be solved at each step. The latter approach usually requires more computational memory and time. Additionally, the GSPM allows for fast computational methods such as the Fast Fourier Transform.

\begin{remark}
    It should be emphasized that,
    although the underlying model is the degenerate LL equation,
    the primary focus of this work is the application and analysis of the FVEM.
    Meanwhile, to reach a fast simulation,
    the first-order GSPM is adopted for time discretization.
    The scheme is unconditionally stable,
    and decouples the system into simpler and more stable subproblems \cite{Li,GSPM2001JCP}.
    The second-order version has also been recently developed, in which biharmonic equations must be solved \cite{li2024enhanced}.
    Hence, the application of the second-order GSPM requires a high-order FVEM, which will be investigated in the near future.
\end{remark}

\subsection{Accuracy check}
Here we exhibit numerical examples to verify the convergence rate of the FVEM approximation.
We consider the equation
\begin{equation}
    \partial_t\bmf{m}
    =
    -\bmf{m}\times\Delta\bmf{m}
    - \alpha\bmf{m}\times(\bmf{m}\times\Delta\bmf{m}) + \bmf{f}_{\mr{source}},
    \label{equ:LL-equ-sorce}
\end{equation}
where $\bmf{f}_{\mr{source}}$ is the source term provided by the exact solution, i.e., $\bmf{f}_{\mr{source}} = \bmf{m}_t + \bmf{m} \times \Delta \bmf{m} + \alpha \bmf{m} \times (\bmf{m} \times \Delta \bmf{m})$.

\begin{example}
Consider the exact solution
$$\bmf{m}=(\sin(x)\cos(y+t), \cos(x)\cos(y+t), \sin(y+t))^T$$ over the domain $\Omega = [0, 1]^2$.
Notice that the values of the magnetization on boundaries are exactly given, and hence the time-dependent Dirichlet boundary condition is used.
We choose different damping parameters $\alpha = 0.1, 0.05$, and record errors at the terminal time $T = 1.0$ as the temporal step size and spatial mesh size vary in \Cref{Table 5} and \Cref{Table 7}.
\begin{table}[htbh]
  \centering
  \caption{Numerical convergence rates for $\alpha = 0.1$.}
  \label{Table 5}
  \begin{tabular}{||c|c|c|c|c|c|c||}
  \hline
  \multicolumn{7}{||c||}{$ \Delta t = \Delta x = \Delta y$} \\
\hline
$\Delta t$ & $\|\bmf{m}^n-\bmf{m}_h^n\|_{\bmf{L}^{\infty}}$ & $\mr{order}$ & $\|\bmf{m}^n-\bmf{m}_h^n\|_{\bmf{L}^2}$ & $\mr{order}$ & $\|\bmf{m}^n-\bmf{m}_h^n\|_{\bmf{H}^1}$ & $\mr{order}$ \\
\hline
1/32  & 4.00e-02  &   \quad  &    2.93e-02     &   \quad  &    6.55e-02   &   \quad\\
1/64  & 2.04e-02  &   0.97   &    1.47e-02     &   1.00   &    3.03e-02   &   1.11 \\
1/128 & 1.03e-02  &   0.98   &    7.48e-03     &   0.97   &    1.51e-02   &   1.00 \\
1/256 & 5.02e-03  &   1.04   &    3.33e-03     &   1.16   &    7.19e-03   &   1.07 \\
\hline
\multicolumn{7}{||c||}{$ \Delta t = \Delta x^2 = \Delta y^2$} \\
\hline
  $\Delta t$ & $\|\bmf{m}^n-\bmf{m}_h^n\|_{\bmf{L}^{\infty}}$ & $\mr{order}$ & $\|\bmf{m}^n-\bmf{m}_h^n\|_{\bmf{L}^2}$ & $\mr{order}$ & $\|\bmf{m}^n-\bmf{m}_h^n\|_{\bmf{H}^1}$ & $\mr{order}$ \\
\hline
  1/64  & 2.07e-02 & \quad& 1.55e-02 &  \quad  &  8.44e-02   &   \quad  \\
  1/256 & 5.06e-03 & 2.03 & 3.67e-03 &  2.08   &  4.00e-02   &   1.07   \\
  1/576 & 2.24e-03 & 2.01 & 1.58e-03 &  2.08  &  2.64e-02   &   1.02   \\
  1/1024 & 1.26e-03 & 2.00 & 8.80e-04 &  2.03   &  1.98e-02   &   1.01   \\
\hline
\end{tabular}
\end{table}

\begin{table}[!tbh]
  \centering
  \caption{Numerical convergence rates for $\alpha = 0.05$.}
  \label{Table 7}
  \begin{tabular}{||c| c| c| c| c| c| c||}
  \hline
  \multicolumn{7}{||c||}{$ \Delta t = \Delta x = \Delta y$} \\
\hline
$\Delta t$ & $\|\bmf{m}^n-\bmf{m}_h^n\|_{\bmf{L}^{\infty}}$ & $\mr{order}$ & $\|\bmf{m}^n-\bmf{m}_h^n\|_{\bmf{L}^2}$ & $\mr{order}$ & $\|\bmf{m}^n-\bmf{m}_h^n\|_{\bmf{H}^1}$ & $\mr{order}$\\
\hline
1/32  &  4.24e-02  &   \quad  &   2.58e-02    &   \quad  &    8.04e-02   &   \quad \\
1/64  &  2.00e-02  &   1.08   &   1.24e-02    &   1.05   &    3.86e-02   &   1.06  \\
1/128 &  1.00e-02  &   1.00   &   6.34e-03    &   0.97   &    1.72e-02   &   1.17  \\
1/256 &  4.97e-03  &   1.01   &   3.24e-03    &   0.97   &    7.86e-03   &   1.13  \\
\hline
\multicolumn{7}{||c||}{$ \Delta t = \Delta x^2 = \Delta y^2$} \\
\hline
$\Delta t$ & $\|\bmf{m}^n-\bmf{m}_h^n\|_{\bmf{L}^{\infty}}$ & $\mr{order}$ & $\|\bmf{m}^n-\bmf{m}_h^n\|_{\bmf{L}^2}$ & $\mr{order}$ & $\|\bmf{m}^n-\bmf{m}_h^n\|_{\bmf{H}^1}$ & $\mr{order}$\\
\hline
  1/64   & 2.04e-02 &  \quad  &  1.57e-02  & \quad  &  8.65e-02 &  \quad  \\
  1/256  & 4.95e-03 &  2.04   &  3.37e-03  & 2.21   &  4.01e-02 &   1.11   \\
  1/1024 & 2.19e-03 &  2.01   &  1.27e-03  & 2.41   &  2.64e-02 &   1.13   \\
  1/4096 & 1.23e-04 &  2.01   &  6.77e-04  & 2.19   &  1.98e-02 &   1.00   \\
\hline
\end{tabular}
\end{table}

Furthermore, we fix $\alpha = 0.1$ and compare the CPU time of the GSPM-FVEM and the backward Euler method (BE-FVEM), and the results are depicted in Figure \ref{CPU-time-com}. Figure \ref{CPU-time-com}(a) shows the comparison when the spatial mesh size is fixed at $1/32$ and the temporal size is varied, while Figure \ref{CPU-time-com}(b) shows the comparison when the temporal size is fixed at $1/1024$ and the spatial size is varied. Both plots demonstrate that GSPM-FVEM requires less CPU time than BE-FVEM for comparable accuracy. This demonstrates the computational advantage of the GSPM approach.

\begin{figure}[htbp]
\centering
\subfigure[CPU time versus the numerical error ($\Delta t$)]{
\label{CPU-time-com.sub.1}
\includegraphics[width=0.45\linewidth]{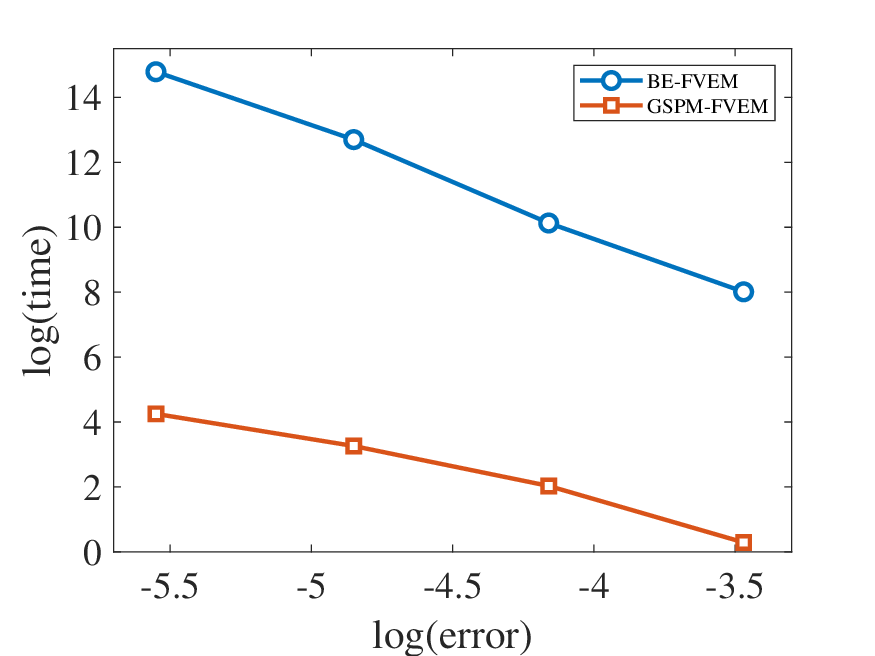}}
\subfigure[CPU time versus the numerical error ($\Delta x$)]{
\label{CPU-time-com.sub.2}
\includegraphics[width=0.45\linewidth]{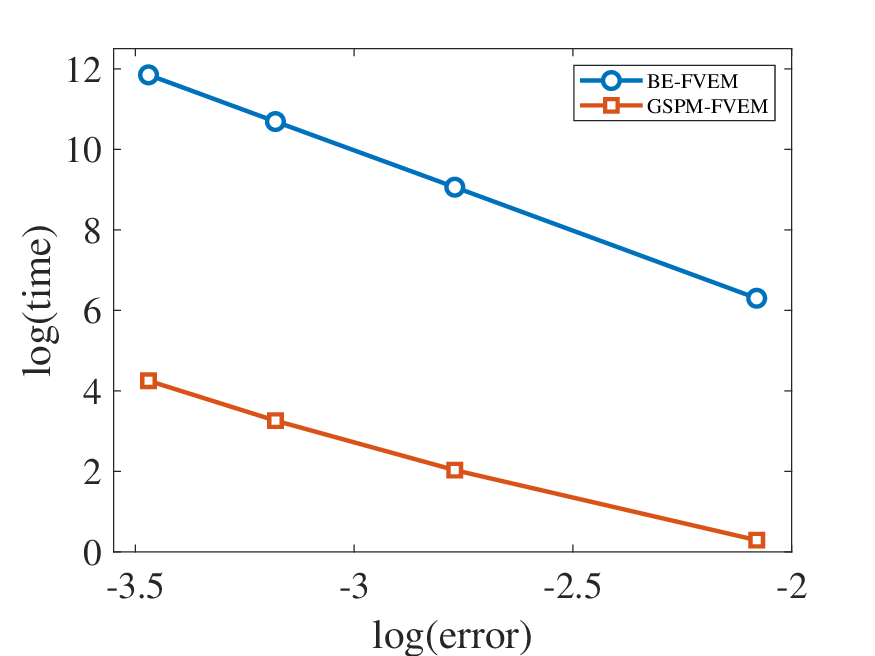}}
\caption{Comparison of CPU time versus the numerical errors for GSPM-FVEM and BE-FVEM.}
\label{CPU-time-com}
\end{figure}
\end{example}

\subsection{Discrete energy laws}
When the GSPM is adopted for the time-marching, the discretized energy law cannot be proved theoretically. Instead, numerical verifications were often conducted. We hereby consider the following two classes of Dirichlet boundary conditions:
\begin{gather}
    \bmf{m} = (\sin(x)\cos(y), \cos(x)\cos(y), \sin(y))^T, \label{equ:boundary-I}\\
    \bmf{m} = (\sin(x)\cos(y+t), \cos(x)\cos(y+t), \sin(y+t))^T.\label{equ:boundary-II}
\end{gather}
The first one satisfies the condition of the energy dissipation law, whereas the second, being time-dependent, does not preserve energy dissipation. The corresponding energy evolutions are depicted numerically as follows.

\begin{example}
We let $\Omega = [0, 1]^2$ and the mesh size $0.2\times 0.2$. Evolve the dynamics to the terminal time $T = 10$ with a time step size $0.1$. For the consistency of the model at $t = 0$, we adopt the initial condition $\bmf{m}_0 = (\sin(x)\cos(y), \cos(x)\cos(y), \sin(y))^T$, and fix the damping parameter $\alpha = 0.1$.
Numerical energy behaviors are recorded as in \Cref{fig4}\,(a)-(b).
\begin{figure}[htbp]
\centering
\subfigure[Boundary condition \eqref{equ:boundary-I}.]{
\label{Fig4.sub.1}
\includegraphics[width=0.45\linewidth]{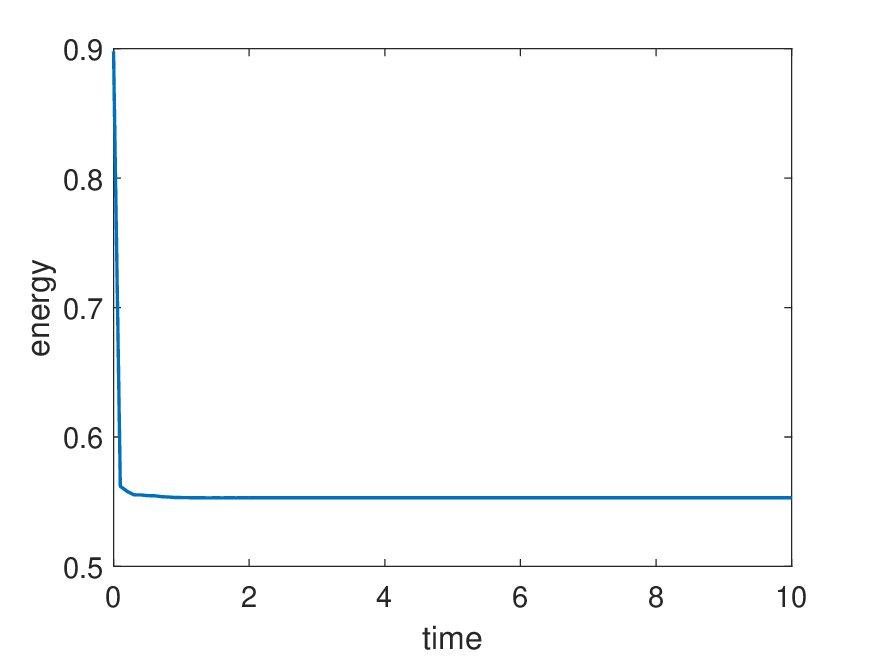}}
\subfigure[Boundary condition \eqref{equ:boundary-II}.]{\label{Fig4.sub.2}
\includegraphics[width=0.45\linewidth]{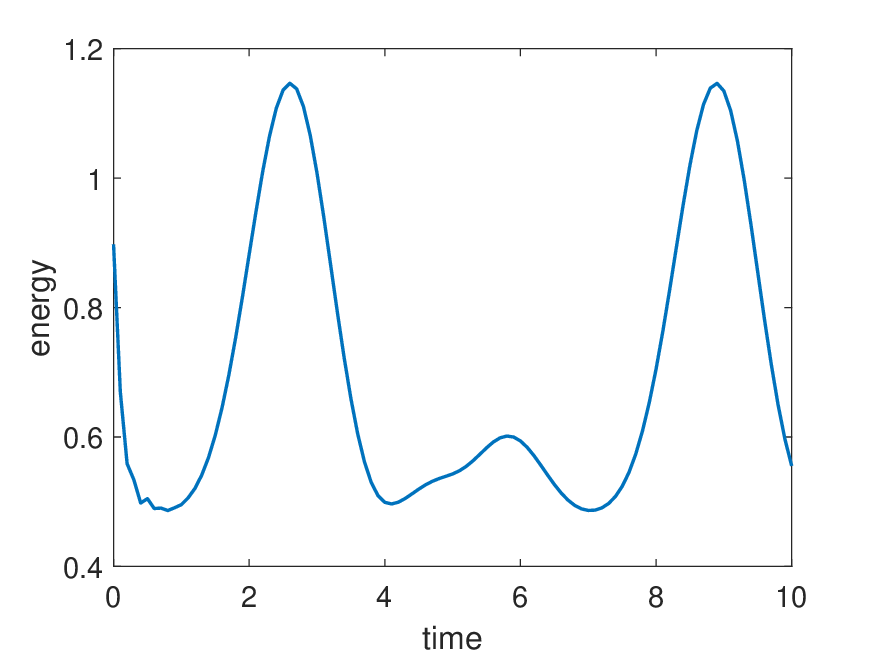}}
\caption{Energy behaviors of the system with different Dirichlet boundary conditions.}
\label{fig4}
\end{figure}
\par
\end{example}

\subsection{Phenomenon of the blow-up}
With a smooth initialization, the solution of the LL equation may blow up in a finite-time evolution.
This interesting phenomenon has been studied in~\cite{an,Bartels,Qing},
where the norm of $\|\nabla\bmf{m}_h\|_{\bmf{L}^{\infty}}$ is the monitor.
On a uniform mesh with mesh size $h$,
the magnitude of $\|\nabla\bmf{m}_h\|_{\bmf{L}^{\infty}}$ is $\mathcal{O}(1/h)$.
Hence one of the feasible approaches to depict the blow-up is observing the direction of magnetization around the origin.
Next, we study this interesting behavior of the LL equation using the proposed method.
We emphasise again that our error analysis relies heavily on the regularity assumption \eqref{regularity-condition}, in which it has assumed $\|\nabla\bmf{m}_h\|_{\bmf{L}^{\infty}}$ to be bounded. This contradiction implies that the error analysis holds before the presence of blow up. On the other hand, this test demonstrates that the proposed approach accurately and effectively studies the intrinsic characteristics of the LL equation.

\begin{example}
Consider $\Omega = [-0.5,0.5]^2$, and the initial condition
\begin{align}
		\bmf{m}_0(\bmf{x})=\left\{
		\begin{aligned}
		    &(0,0,-1)^T, & \quad  |\bmf{x}|\geq 0.5,\\
		    &\left(\frac{2 x A}{A^2+|\bmf{x}|^2},\frac{2 y A}{A^2+|\bmf{x}|^2},\frac{A^2-|\bmf{x}|^2}{A^2+|\bmf{x}|^2}\right)^T,	& \quad \mr{otherwise},
		\end{aligned}\right.
\end{align}
where $A=(1-2|\bmf{x}|^2)^4$.
The Dirichlet boundary condition is fixed and given by the initial condition on boundaries. The other parameters are: $h = \frac{1}{256}$, $\Delta t = $ 1.0e-04 and $\alpha = 1.0$.
Some snapshots are shown in \Cref{fig10}.
Meanwhile, to depict the blow-up,
we also provide close-up pictures of magnetization near the origin at the corresponding time in \Cref{fig11}. Near the origin,
the magnetization $\bmf{m}_h$ turns down to $(0,0,-1)^T$,
while the magnetization at origin maintains $(0,0,1)^T$, which indicates the singularity of $\nabla\bmf{m}_h$ at the origin.
Furthermore,
we record the evolution of the system's energy and $\|\nabla\bmf{m}_h\|_{\bmf{L}^{\infty}}$ in \Cref{fig12} to verify the observation.
\begin{figure}[htbp]
\centering
\subfigure[$t = 0$.]{
\label{Fig10.sub.1}
\includegraphics[width=4.0cm]{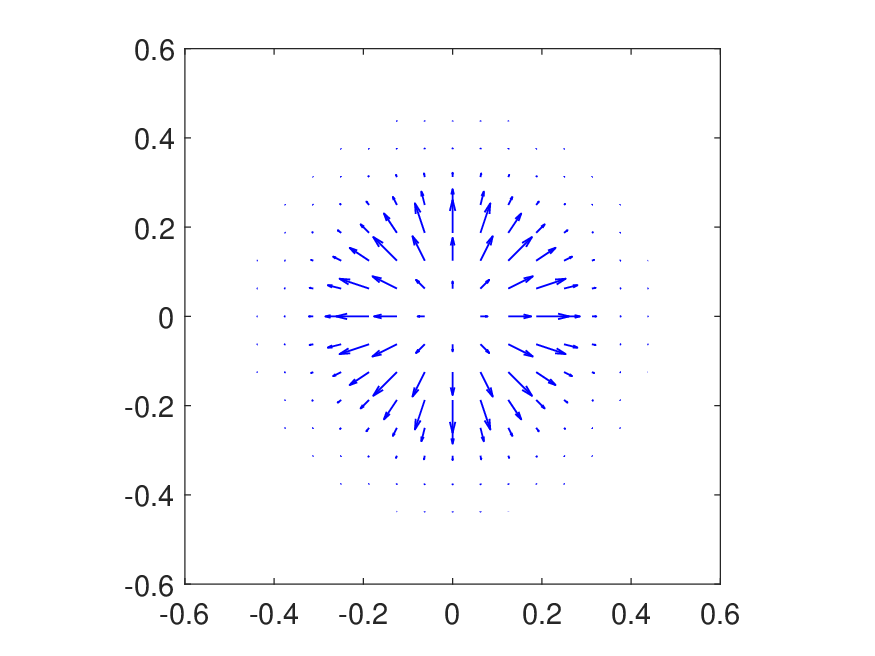}}
\hspace{-10mm}
\subfigure[$t = 0.001$.]{
\label{Fig10.sub.2}
\includegraphics[width=4.0cm]{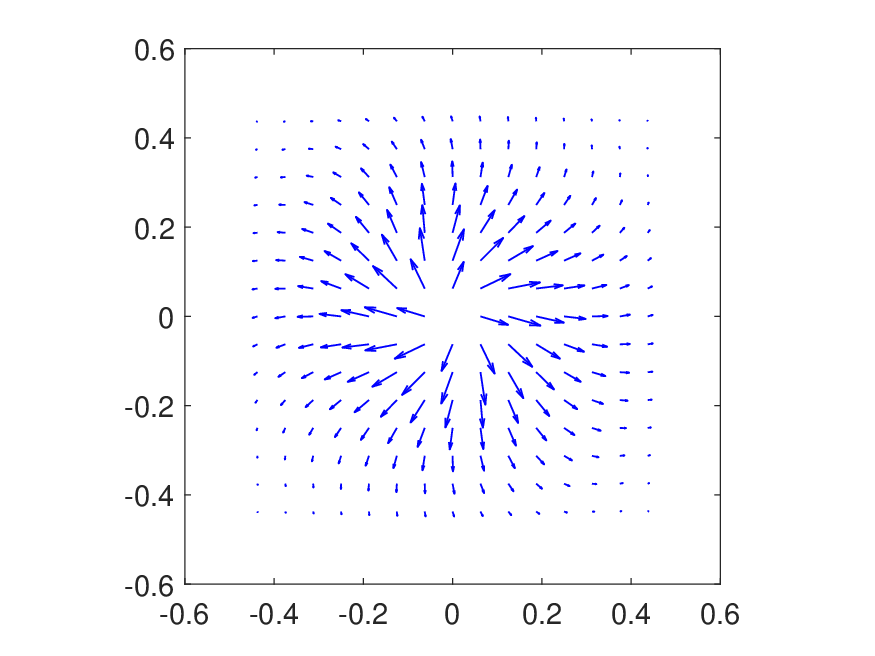}}
\hspace{-10mm}
\subfigure[$t = 0.05$.]{\label{Fig10.sub.4}
\includegraphics[width=4.0cm]{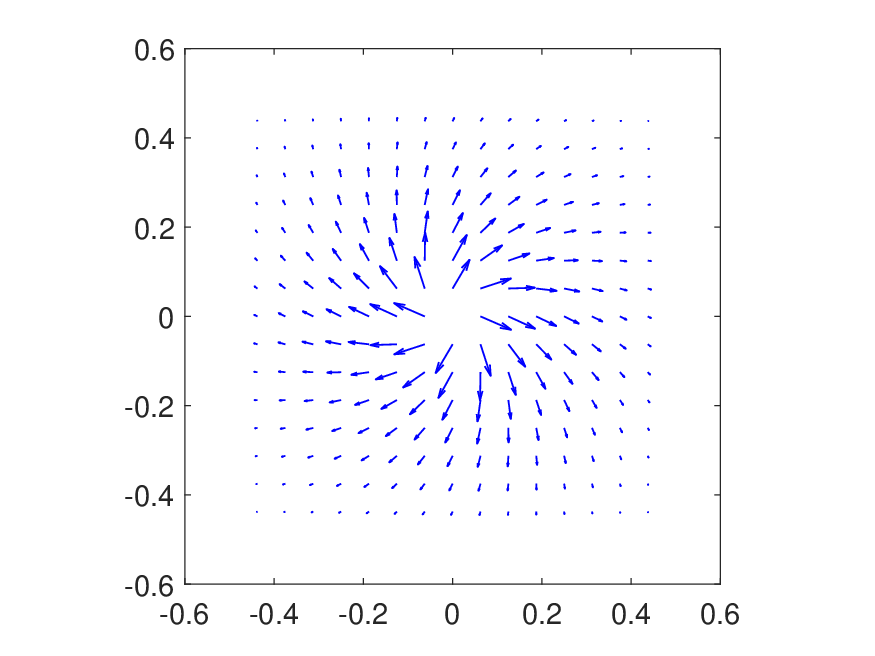}}
\hspace{-10mm}
\subfigure[$t = 0.1$.]{\label{Fig10.sub.5}
\includegraphics[width=4.0cm]{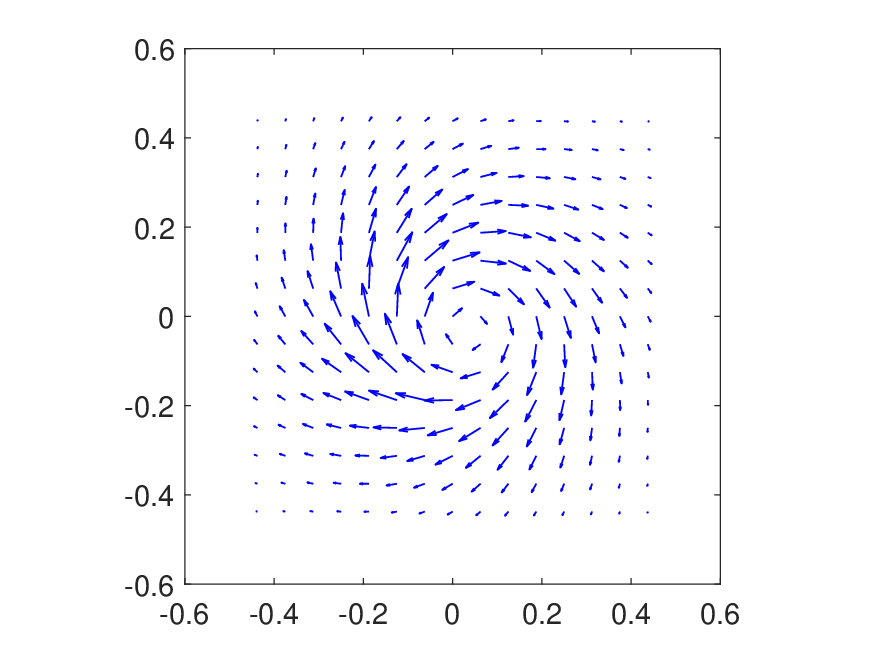}}
\\
\subfigure[$t = 0.2$.]{
\label{Fig10.sub.6}
\includegraphics[width=4.0cm]{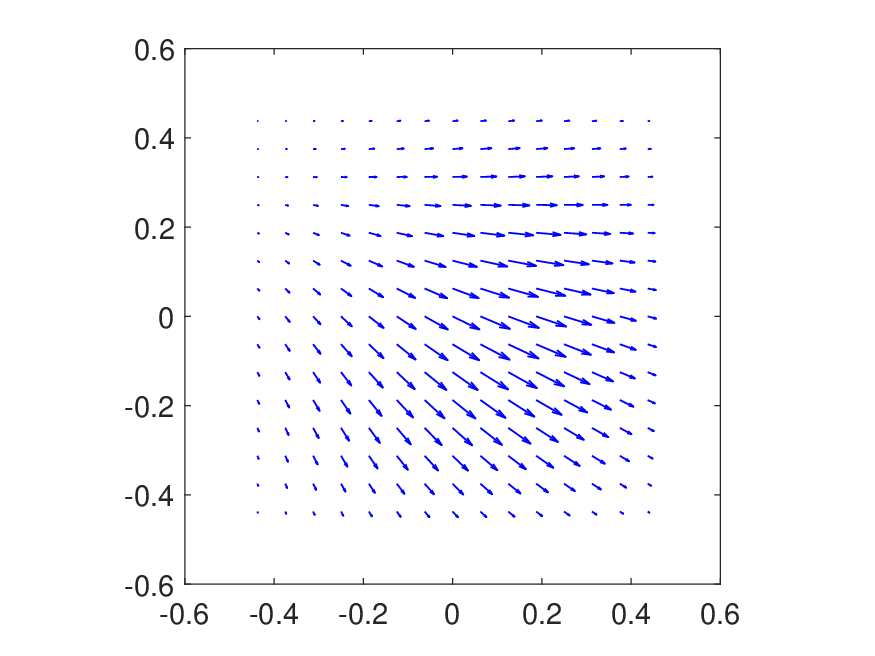}}
\hspace{-10mm}
\subfigure[ $ t = 0.4 $.]{\label{Fig10.sub.7}
\includegraphics[width=4.0cm]{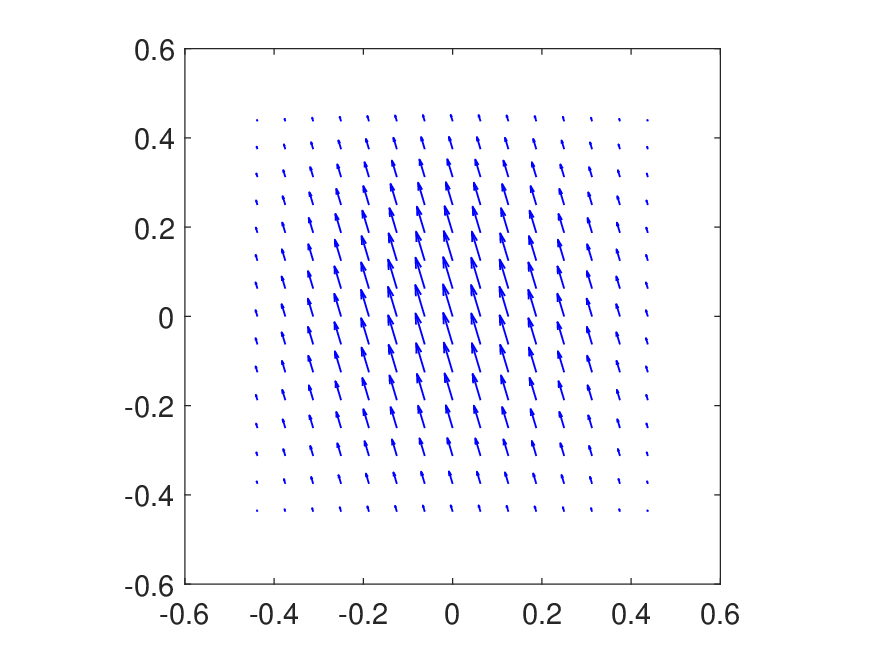}}
\hspace{-10mm}
\subfigure[$t = 0.5$.]{
\label{Fig10.sub.8}
\includegraphics[width=4.0cm]{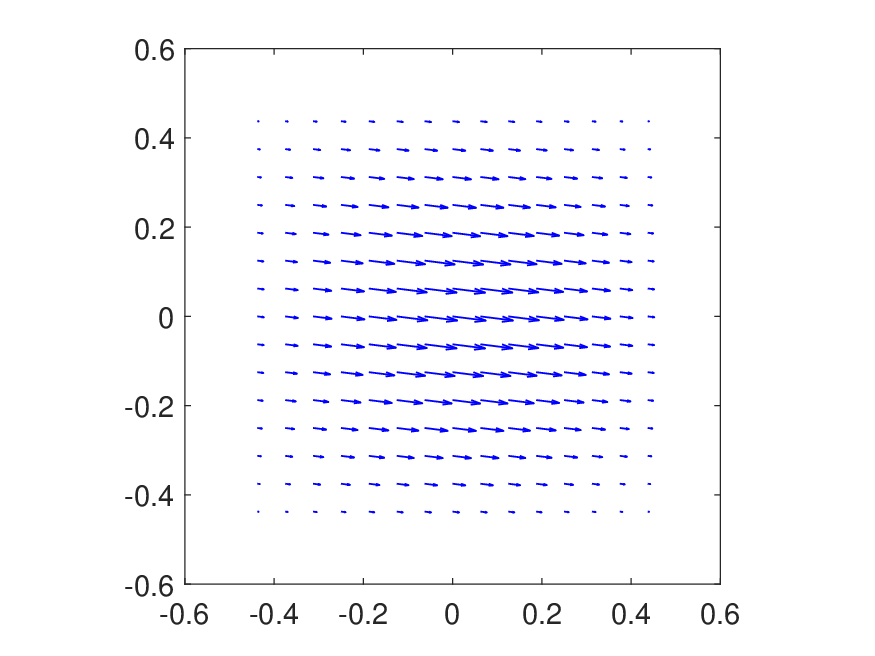}}
\hspace{-10mm}
\subfigure[$t = 0.6$.]{
\label{Fig10.sub.9}
\includegraphics[width=4.0cm]{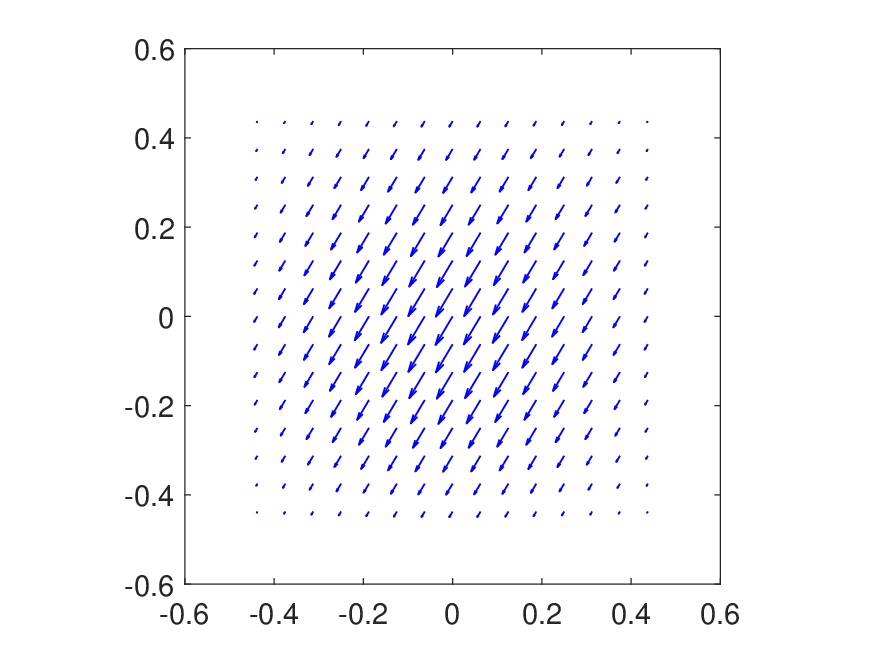}}
\caption{The magnetization textures at different times during the blow-up.}
\label{fig10}
\end{figure}

\begin{figure}[htbp]
\centering
\subfigure[$t = 0$.]{
\label{Fig11.sub.1}
\includegraphics[width=3.3cm]{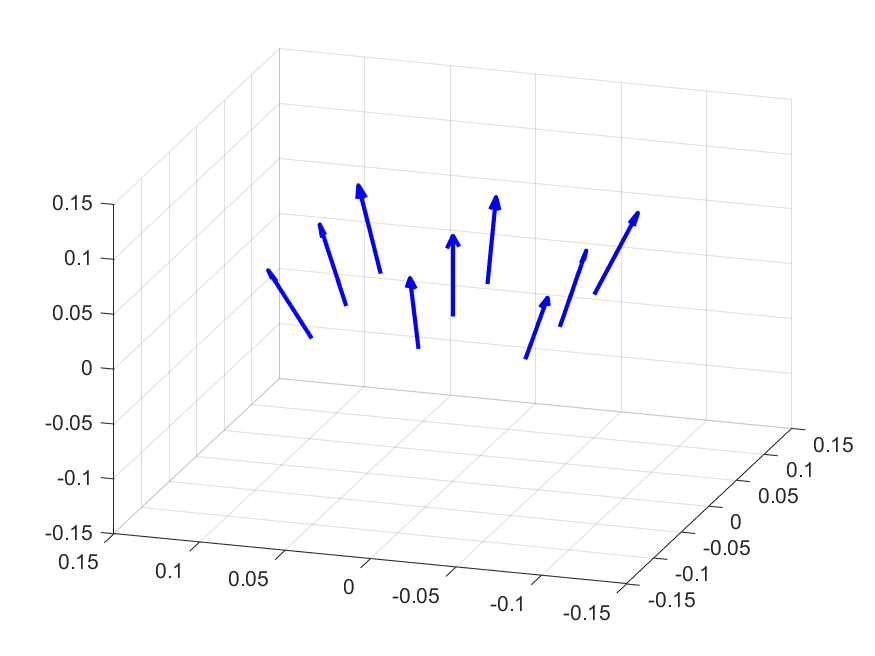}}
\subfigure[$t = 0.001$.]{
\label{Fig11.sub.2}
\includegraphics[width=3.3cm]{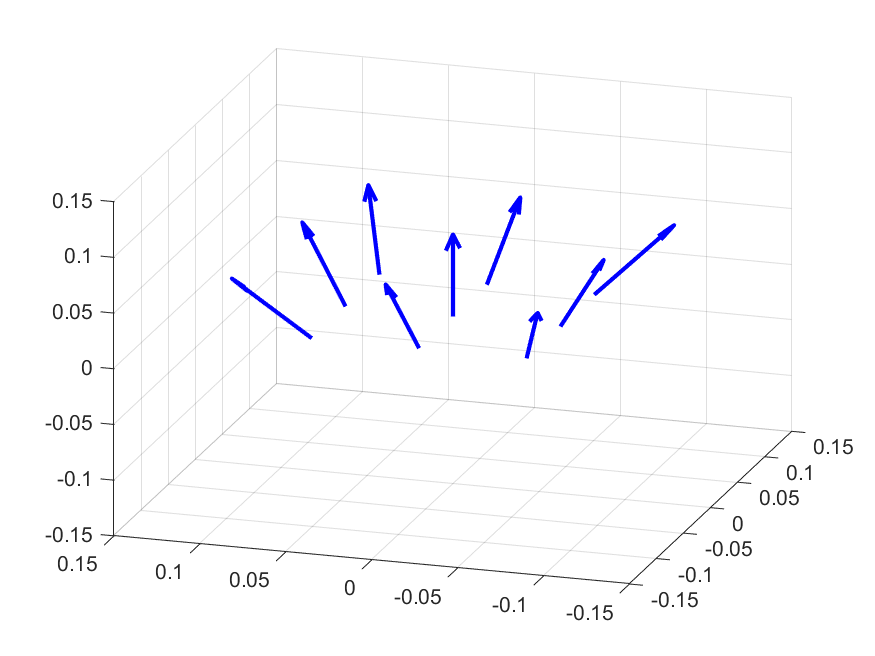}}
\subfigure[$t = 0.05$.]{
\label{Fig11.sub.4}
\includegraphics[width=3.3cm]{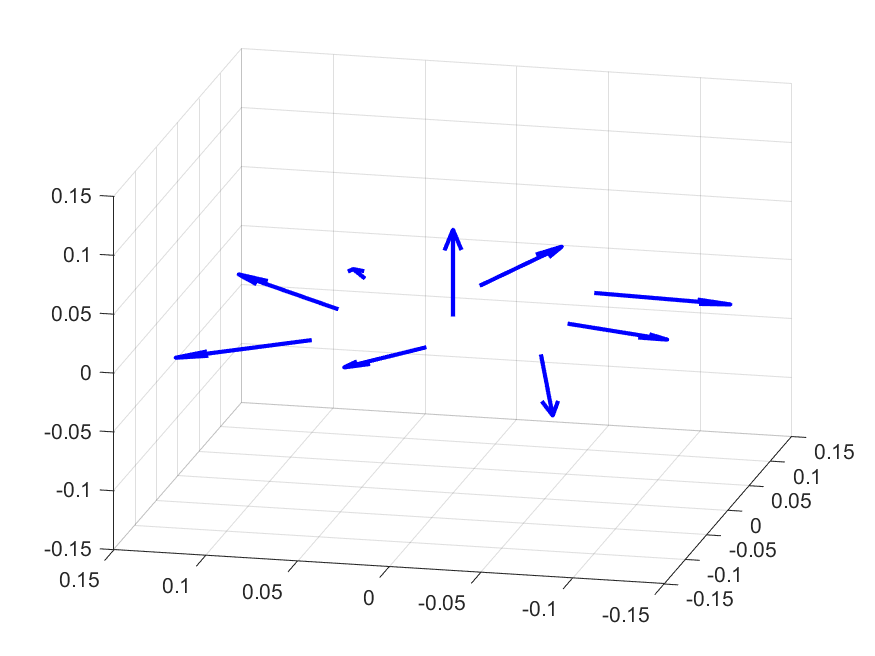}}
\subfigure[$t = 0.1$.]{
\label{Fig11.sub.5}
\includegraphics[width=3.3cm]{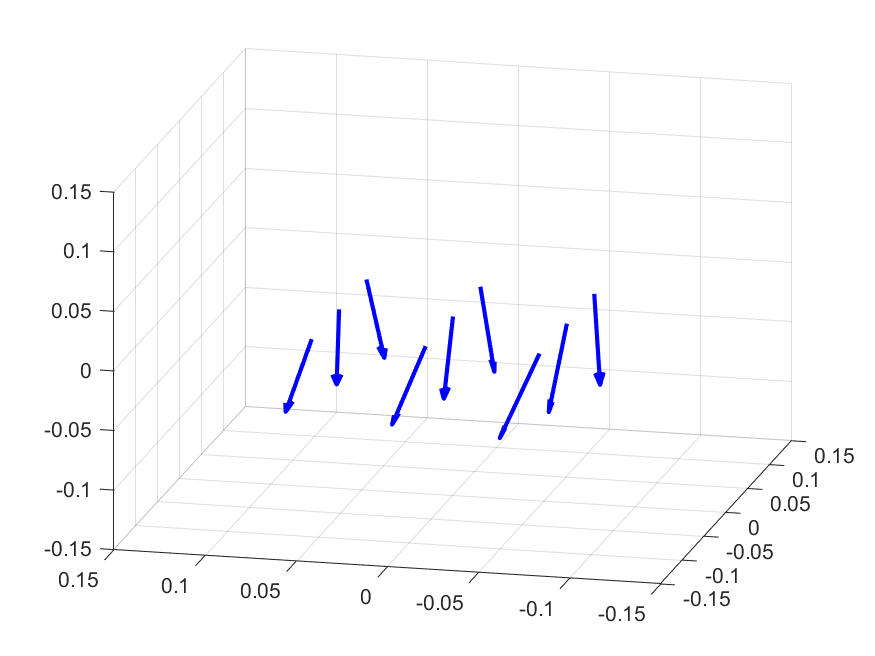}}
\caption{The magnetization around the origin at different times. The magnetization around the original points tends to align uniformly when $t > 0.1$.  The snapshots are therefore omitted here.}
\label{fig11}
\end{figure}

\begin{figure}[htbp]
\centering
\includegraphics[width=0.5\linewidth]{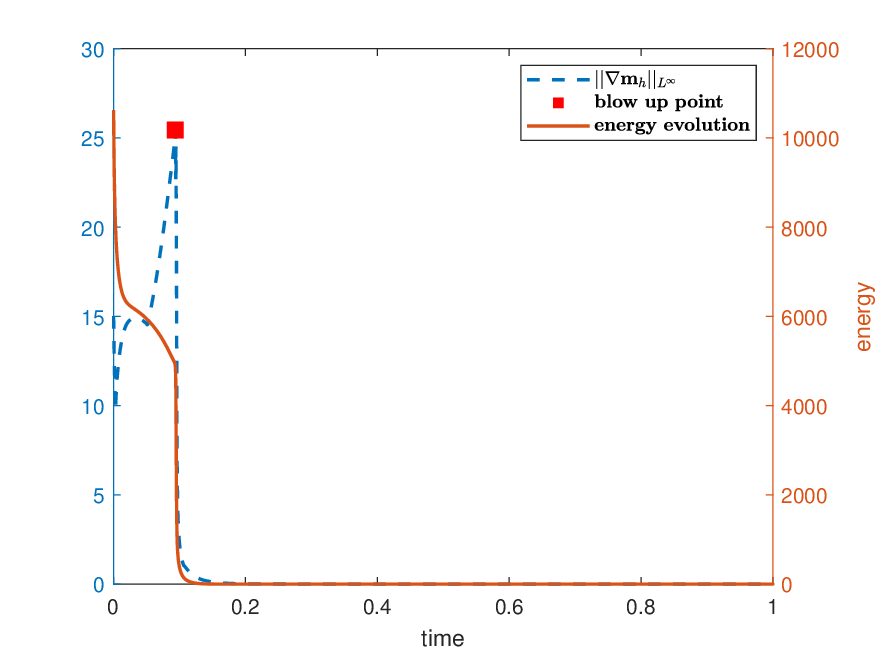}
\caption{Evolutions of the energy and $\|\nabla\bmf{m}_h \|_{\bmf{L}^{\infty}}$ using the proposed method with $\Delta t = 10^{-4}$.}
\label{fig12}
\end{figure}
 \end{example}

\subsection{Micromagnetics simulations}

We continue to investigate the effect of anisotropy on blow-up behavior and simulate several static magnetic textures using the proposed method.

\begin{example}
Consider the 2D ferromagnetic film with the size of $ 1 \, \mu \mr{m} \, \times 1 \, \mu \mr{m}$ and elements with the size of $ 20 \, \mr{nm} \, \times \, 20 \, \mr{nm}$, and choose the material parameters: $M_s=8.0 \times 10^{5} \, \mr{A/m}$, $ A = 1.3 \times 10^{-11} \, \mr{J/m} $, $K_u = 500 \, \mr{J/m^3}$,
and the damping parameter $\alpha = 1.0$.
The simulation starts with the initialization
\begin{equation}
\bmf{m}_0(\bmf{x})=\left\{\begin{aligned}
    &(0,0,-1)^T, \quad &|\bmf{x}-\bmf{x}_0| \geq 0.5, \\
    &\left(\frac{2 (x -0.5) A}{A^2+|\bmf{x} - \bmf{x}_0|^2},\frac{2 (y-0.5) A}{A^2+|\bmf{x} - \bmf{x}_0|^2},\frac{A^2-|\bmf{x} - \bmf{x}_0|^2}{A^2+|\bmf{x} - \bmf{x}_0|^2}\right)^T, & \text{otherwise},
\end{aligned}\right.
\end{equation}
where $A=(1-2|\bmf{x}-\bmf{x}_0|^2)^4$, and $\bmf{x}_0 = (0.5, 0.5)$.  We adopt the time step size $ \Delta t = 1 \, \mr{ps} $, the final time $ T = 10 \, \mr{ns} $. The different anisotropy are along with $\bmf{e}_1$ and $\bmf{e_3}$, respectively.

Snapshots at different times are depicted in \Cref{fig13}, in which the interfaces emerge.
This differs from the model which lacks the inclusion of the anisotropy and the stray field.
Concurrently,
close-up depictions of $\bmf{m}_h$ in the vicinity of the  point $\bmf{x}_0$ are presented at the corresponding time instances in \Cref{fig14}.
These snapshots align with those of the model which lacks the anisotropy and the stray field.

We also record the temporal evolution of the energy and the norm $\|\bmf{m}_h\|_{\bmf{L}^{\infty}}$ as in \Cref{fig16}.
The results indicate that the blow-up may occur within a very short time in realistic simulations.

\begin{figure}[htbp]
\centering
\subfigure[$t = 0 \,\mr{ps}$.]{\label{Fig13.sub.1}
\includegraphics[width=4.0cm]{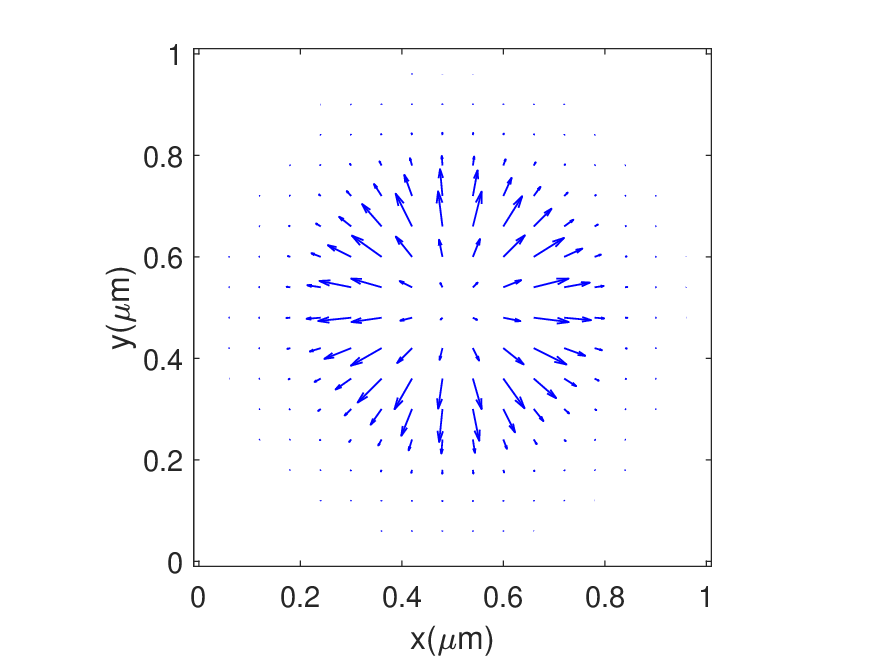}}
\hspace{-10mm}
\subfigure[$ t = 10 \,\mr{ps} $.]{
\label{Fig13.sub.2}
\includegraphics[width=4.0cm]{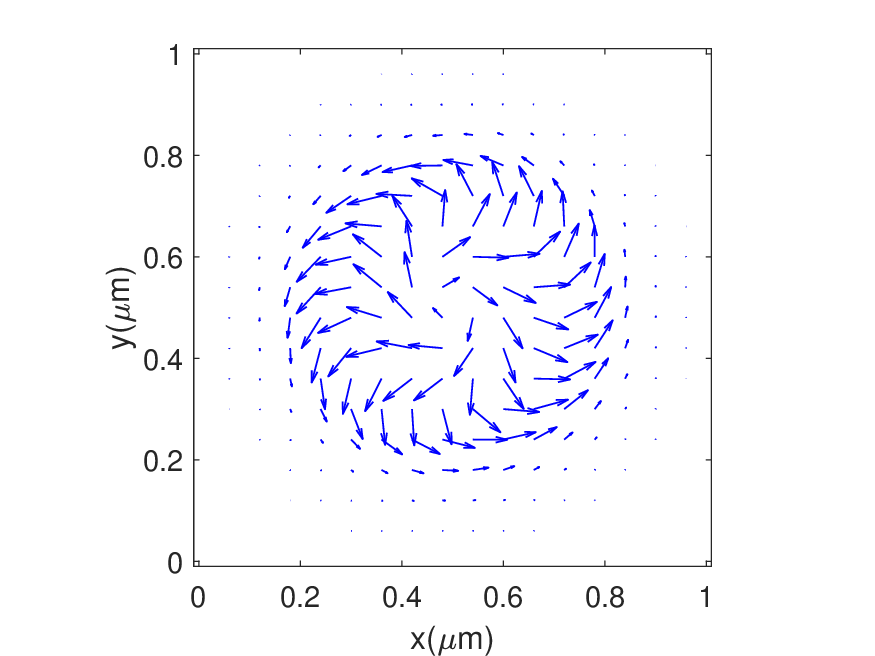}}
\hspace{-10mm}
\subfigure[$ t = 100 \,\mr{ps} $.]{
\label{Fig13.sub.3}
\includegraphics[width=4.0cm]{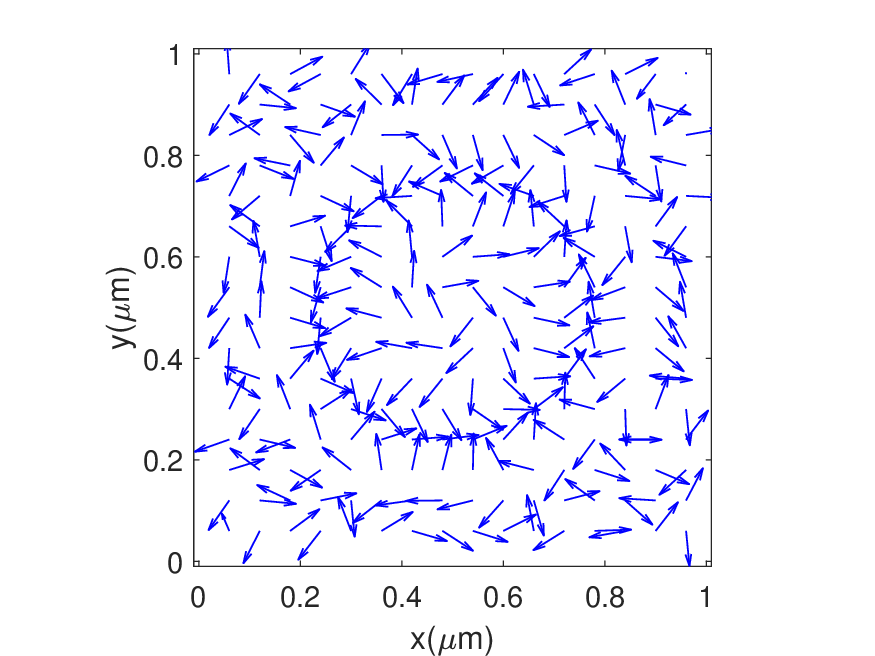}}
\hspace{-10mm}
\subfigure[$t = 500 \,\mr{ps}$.]{
\label{Fig13.sub.4}
\includegraphics[width=4.0cm]{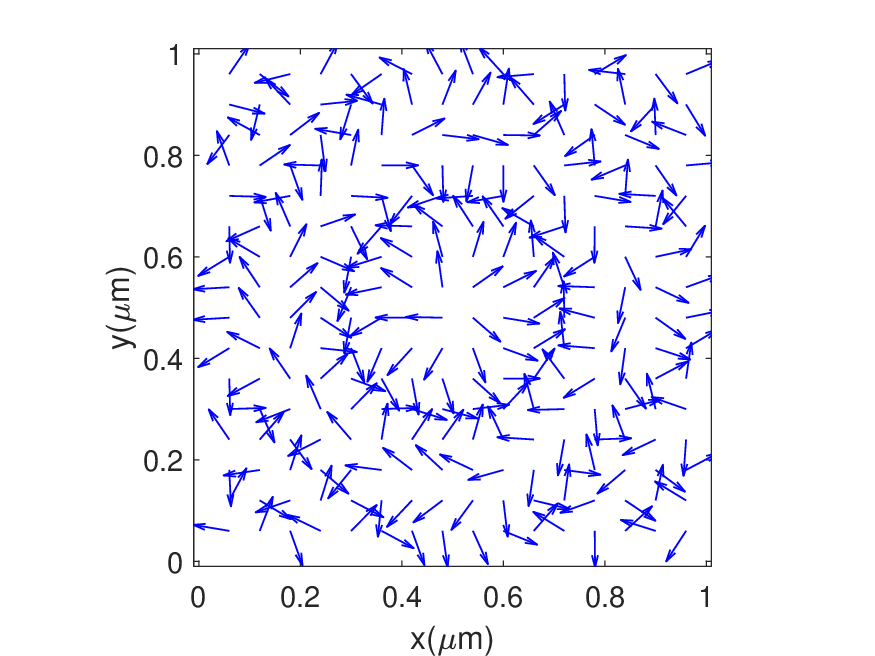}}
\\
\subfigure[$t = 1000 \,\mr{ps}$.]{\label{Fig13.sub.5}
\includegraphics[width=4.0cm]{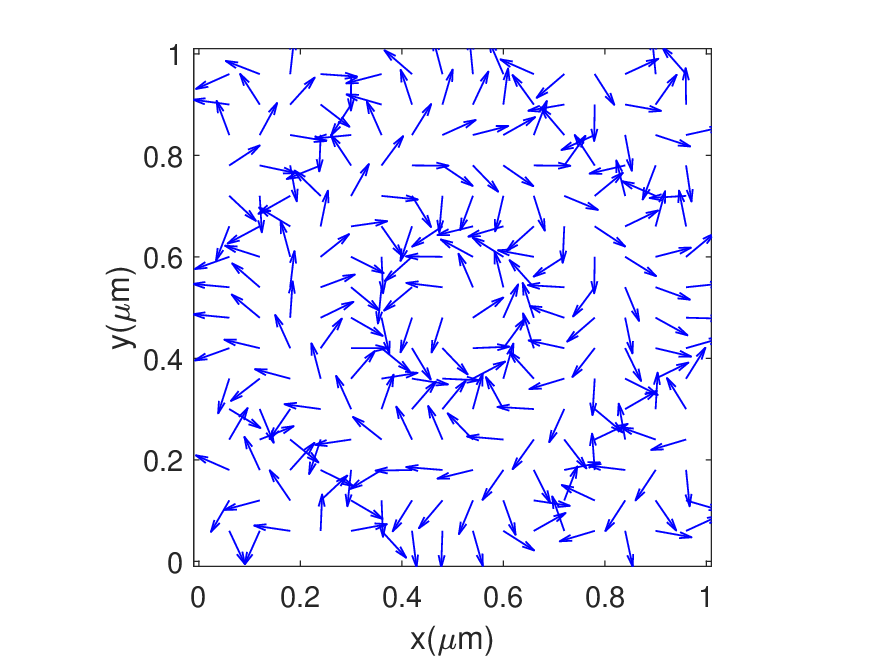}}
\hspace{-10mm}
\subfigure[$t = 2000 \,\mr{ps}$.]{\label{Fig13.sub.6}
\includegraphics[width=4.0cm]{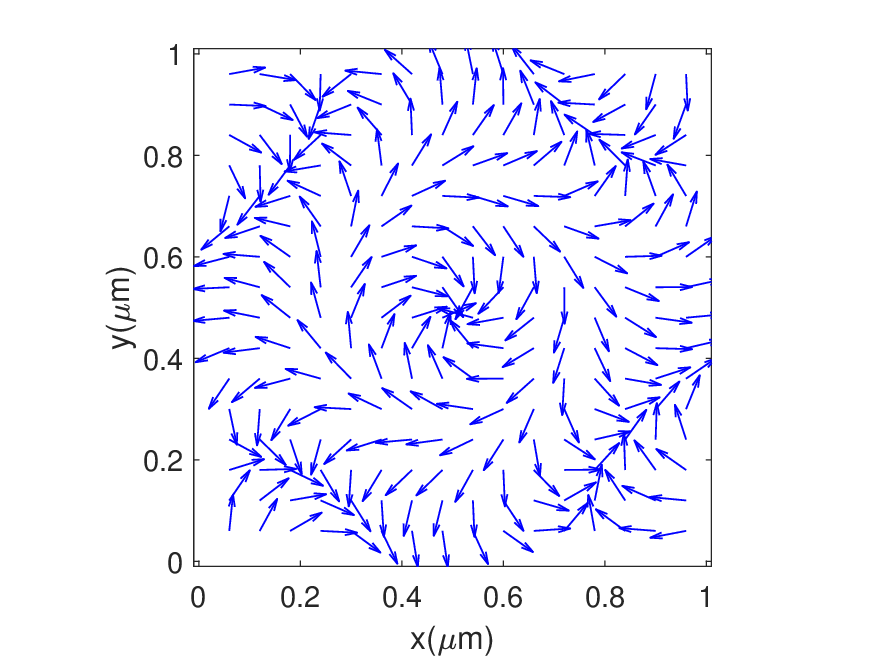}}
\hspace{-10mm}
\subfigure[$t = 4000 \,\mr{ps}$.]{\label{Fig13.sub.7}
\includegraphics[width=4.0cm]{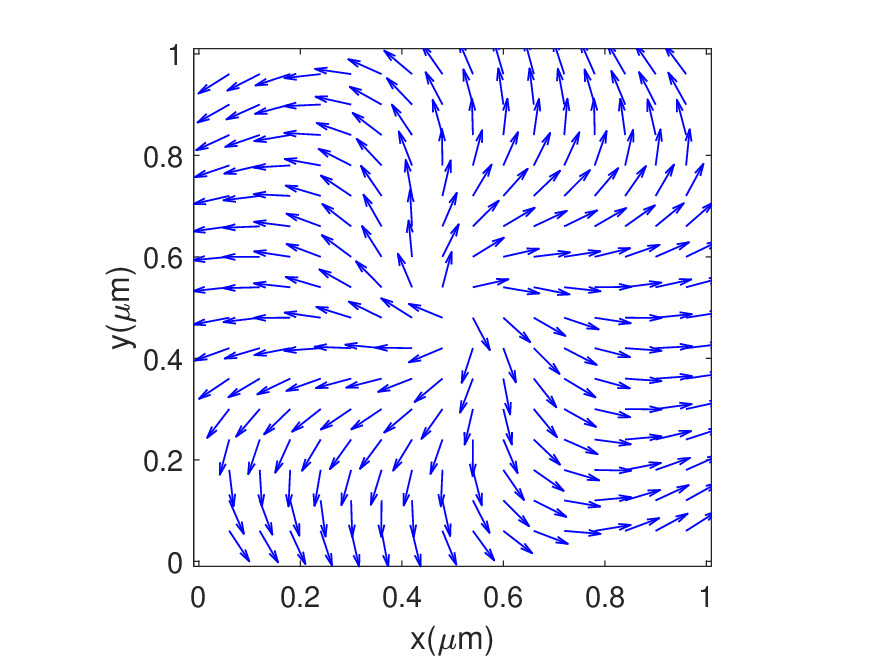}}
\hspace{-10mm}
\subfigure[$t = 6000 \,\mr{ps}$.]{\label{Fig13.sub.9}
\includegraphics[width=4.0cm]{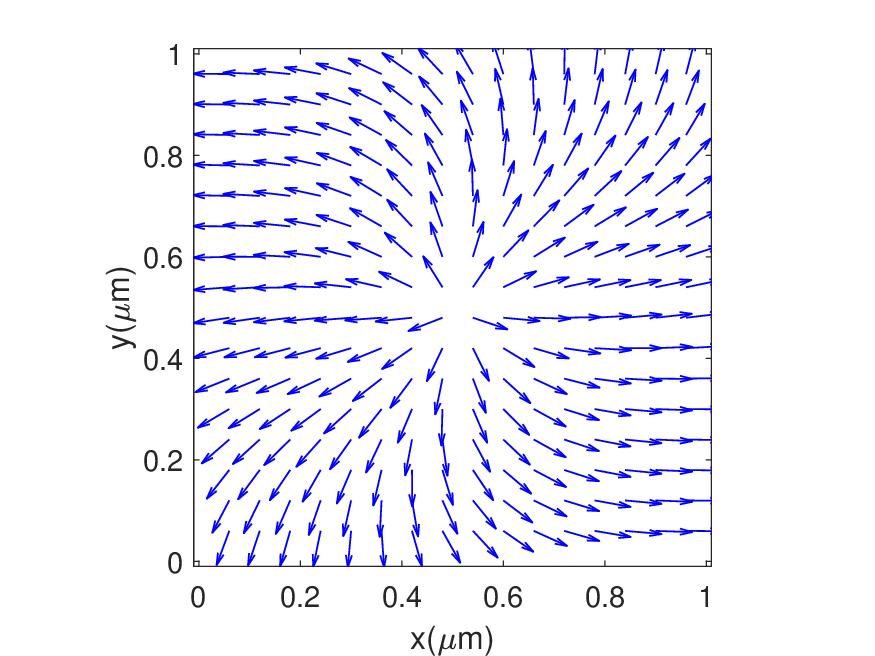}}
\caption{The magnetization textures at different times.}
\label{fig13}
\end{figure}

\begin{figure}[htbp]
\centering
\subfigure[$t = 0 \,\mr{ps}$.]{\label{Fig14.sub.1}
\includegraphics[width=3.3cm]{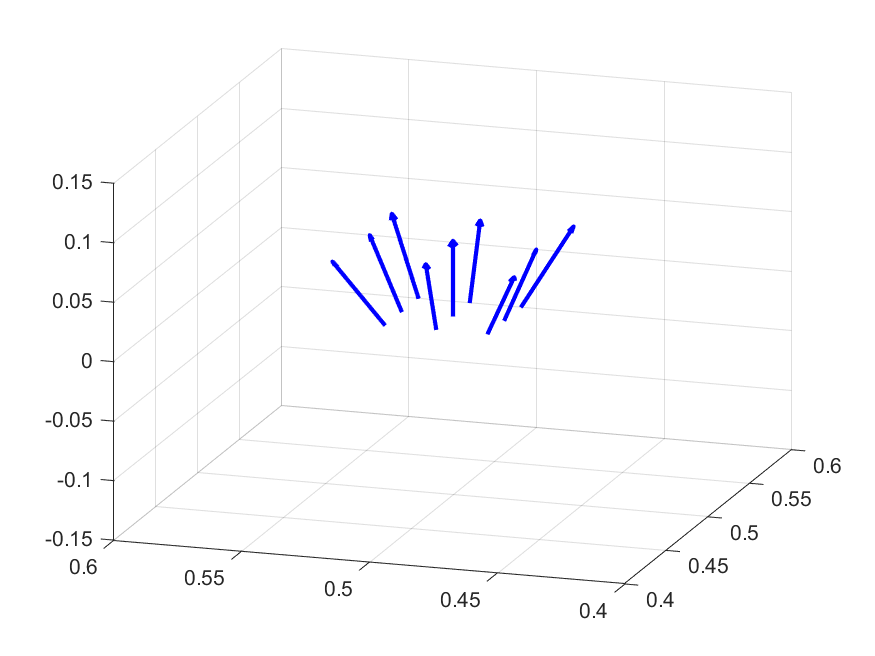}}
\subfigure[$ t = 10 \,\mr{ps} $.]{\label{Fig14.sub.2}
\includegraphics[width=3.3cm]{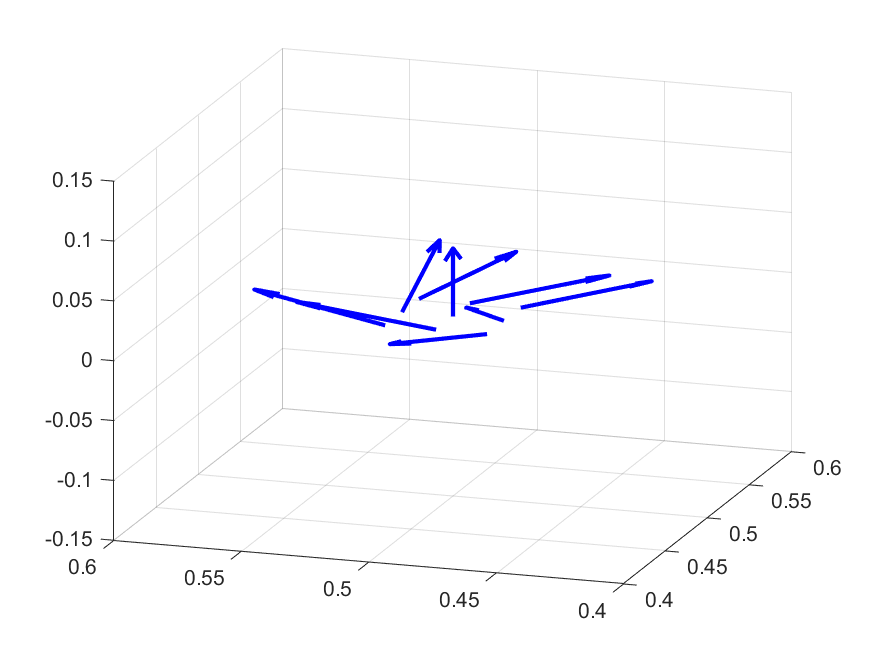}}
\subfigure[$ t = 100 \,\mr{ps} $.]{\label{Fig14.sub.3}
\includegraphics[width=3.3cm]{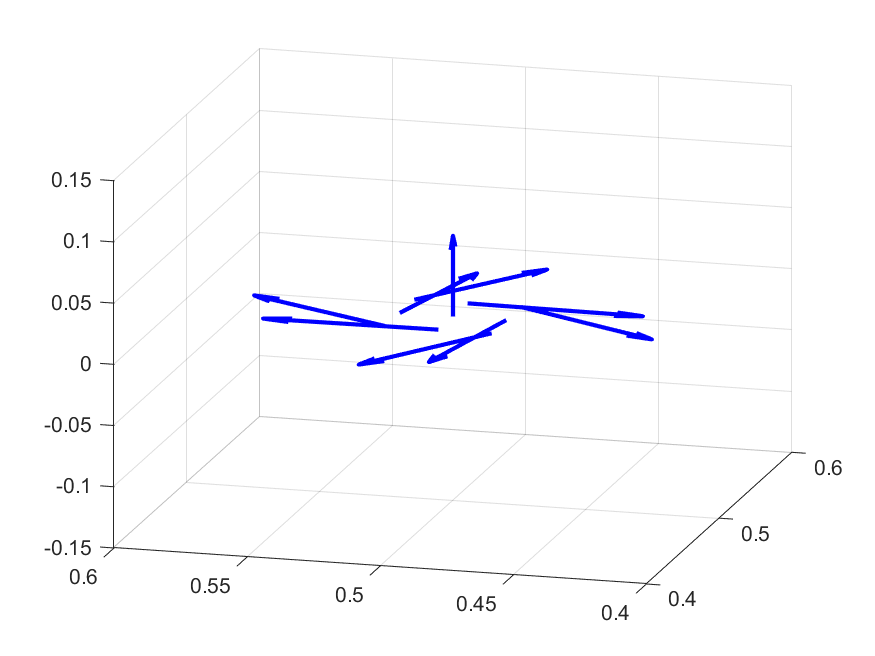}}
\subfigure[$ t = 500 \,\mr{ps} $.]{\label{Fig14.sub.4}
\includegraphics[width=3.3cm]{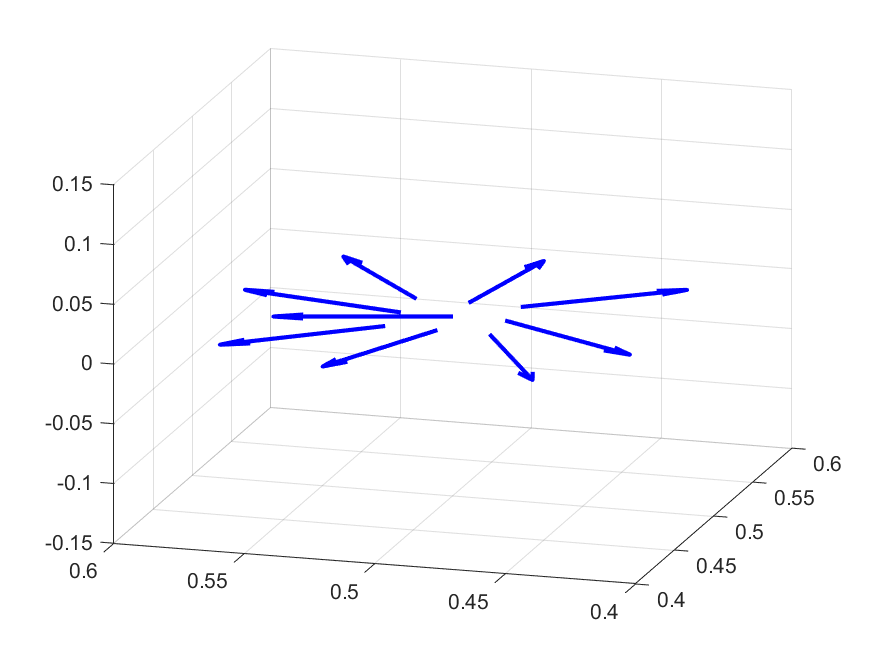}}
\subfigure[$ t = 1000 \,\mr{ps} $.]{\label{Fig14.sub.5}
\includegraphics[width=3.3cm]{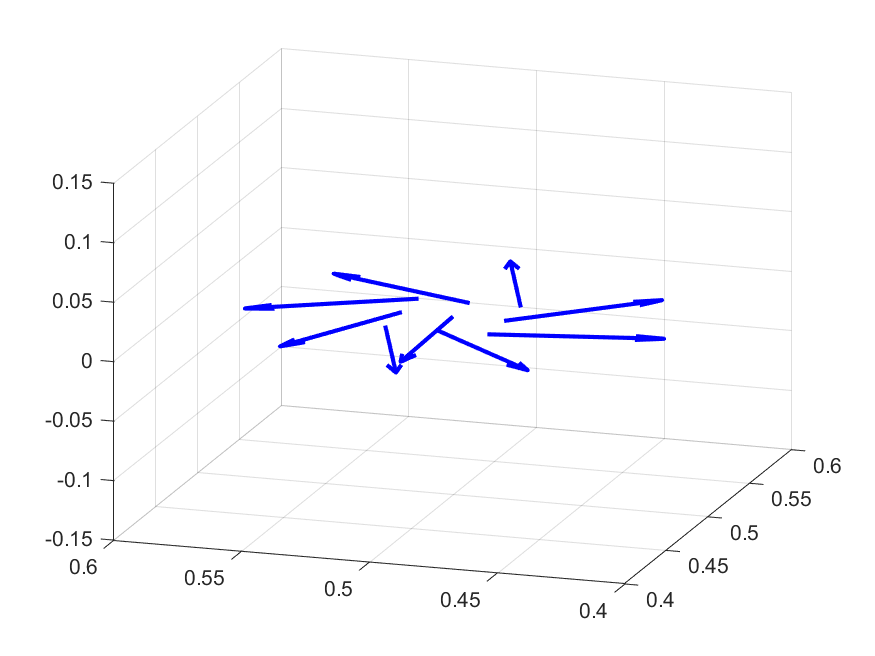}}
\subfigure[$ t = 2000 \,\mr{ps} $.]{\label{Fig14.sub.6}
\includegraphics[width=3.3cm]{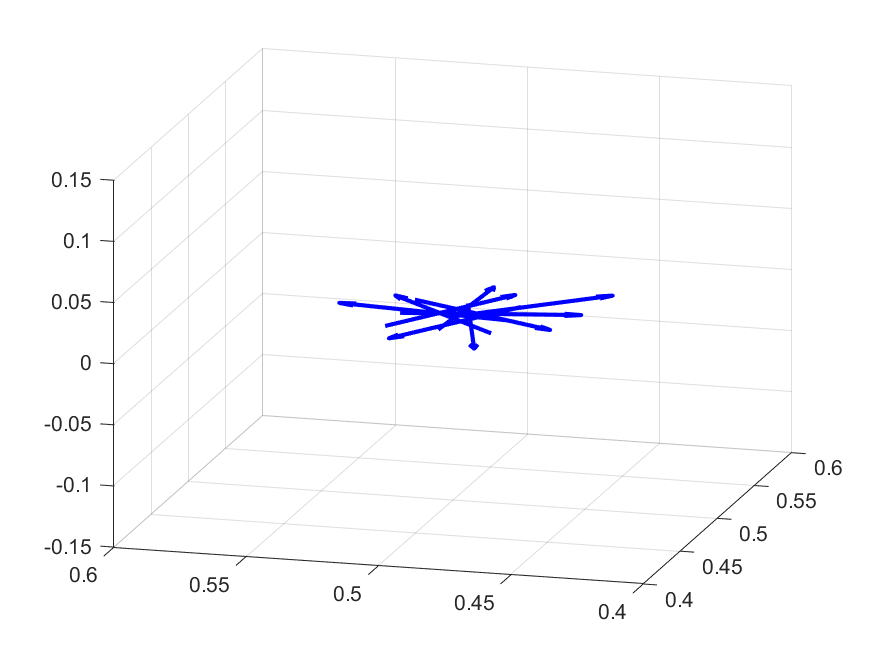}}
\subfigure[$ t = 4000 \,\mr{ps} $.]{\label{Fig14.sub.7}
\includegraphics[width=3.3cm]{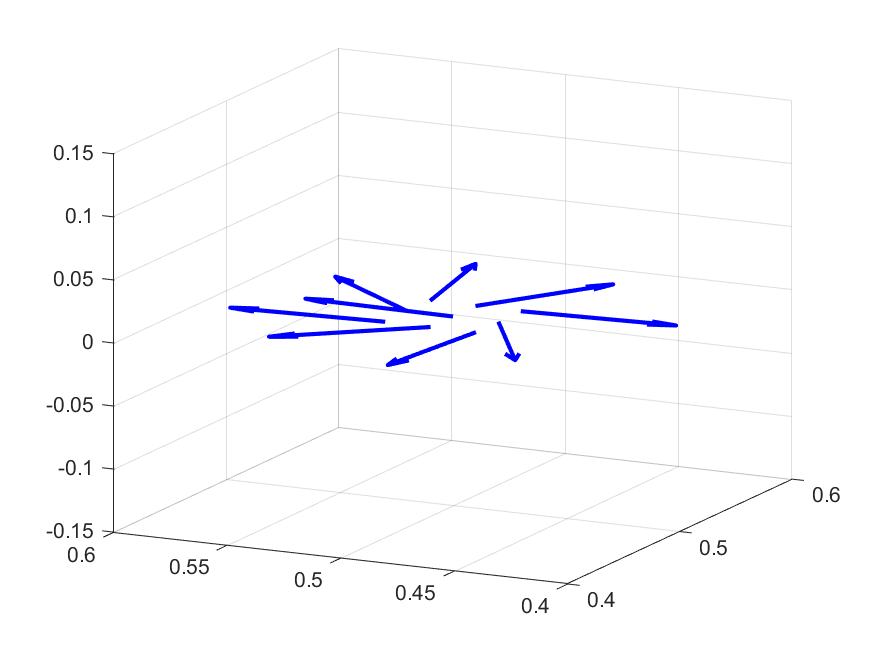}}
\subfigure[$ t = 6000 \,\mr{ps} $.]{\label{Fig14.sub.8}
\includegraphics[width=3.3cm]{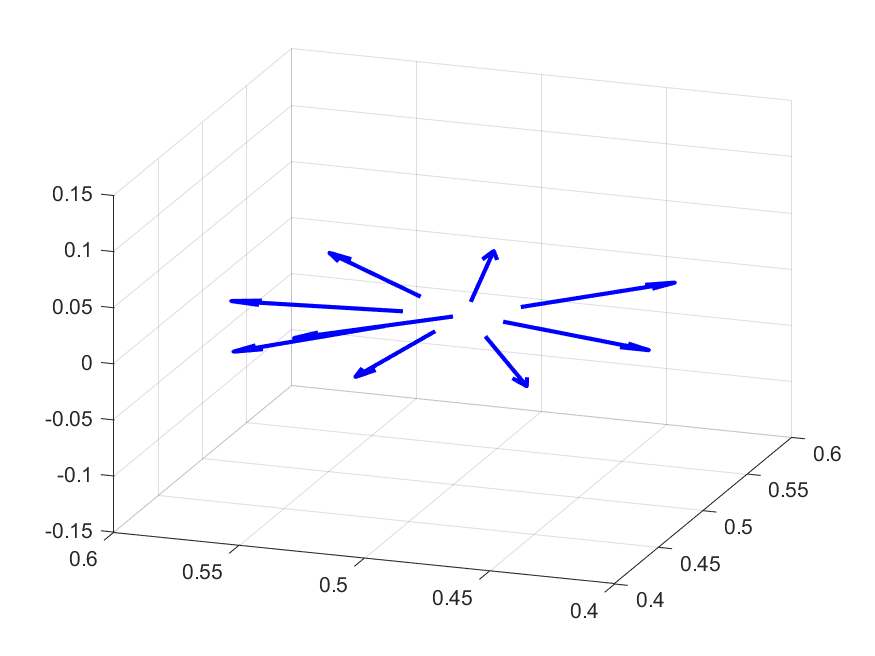}}
\caption{The magnetization around the point $\bmf{x}_0$ at different times.}
\label{fig14}
\end{figure}

\begin{figure}[htbp]
\centering
\subfigure[Anisotropy is along with $\bmf{e}_1$.]{
\label{Fig15.sub.1}
\includegraphics[width=5.0cm]{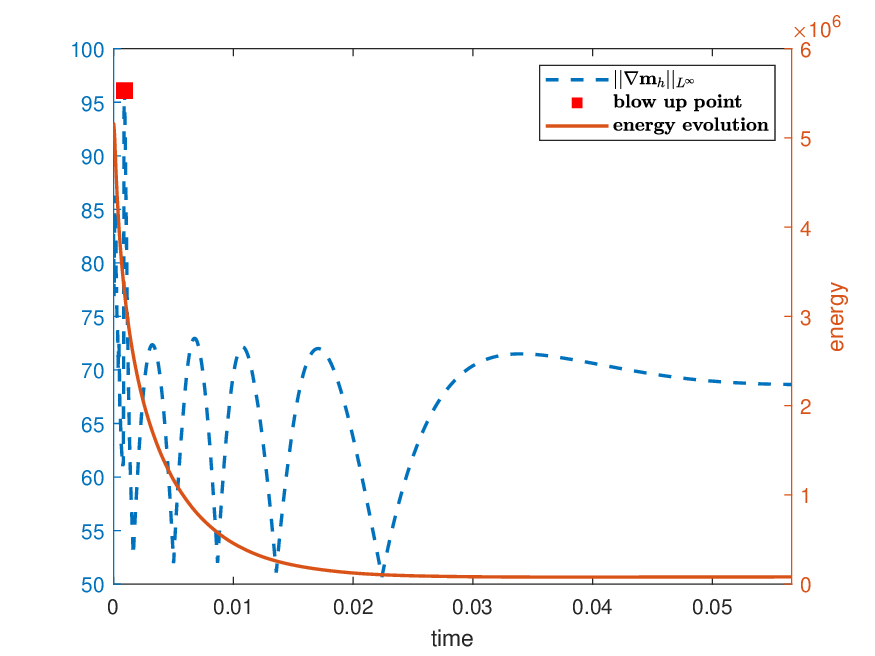}}
\subfigure[Anisotropy is along with $\bmf{e}_3$.]{
\label{Fig15.sub.2}
\includegraphics[width=5.0cm]{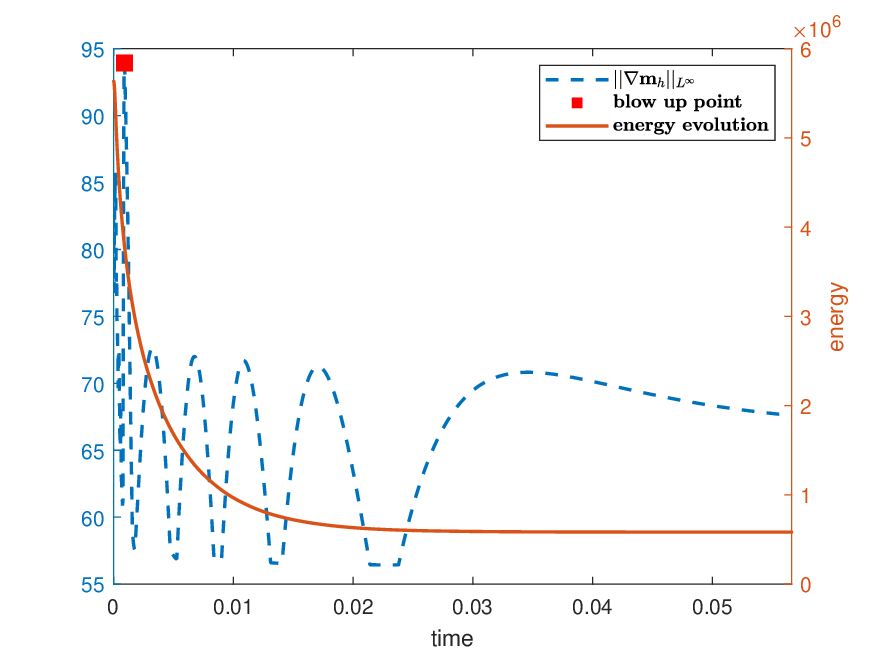}}
\caption{Evolutions of the system energy and $\|\nabla\bmf{m}_h \|_{\bmf{L}^{\infty}}$ using the proposed method with $\Delta t = 1 \,\mr{ps}$.}
\label{fig16}
\end{figure}

\end{example}

\begin{example}
At last, we study the ferromagnetic film with the size of $ 20\;\mr{nm}\times10\;\mr{nm}$ using the mesh size $0.2\;\mr{nm}\times0.2\;\mr{nm}$.
The other parameters are:
$M_s=8.0\times10^{5}\;
\mr{A/m}$, $A = 1.3\times10^{-11}\;
\mr{J/m} $, and $K_u = 100\;
\mr{J/m^3}$.
Here the time-independent Dirichlet boundary condition
\begin{equation}\label{equ:boundary-micromagnetics}
    \bmf{m}(0, y) = (0, 1, 0)^T, \quad \bmf{m}(2, y) = (0, -1, 0)^T, \quad \bmf{m}(x, 0) = (-1, 0, 0)^T, \quad \bmf{m}(x, 1) = (1, 0, 0)^T
\end{equation}
is employed. And the LL equation is evolved to $T = 5\;\mr{ns}$ with  $\Delta t = 0.1\;\mr{ps}$.

The first test is conducted with the initialization
$$\bmf{m}_0(x, y) = (0, 0, 1)^T, \quad (x, y) \in \Omega/\partial\Omega,$$
and different damping parameters $\alpha = 0.8, 1.0$.
Stable vortexes with different diameters are shown in \Cref{fig6}.
\begin{figure}[htbp]
\centering
\includegraphics[width=0.48\linewidth]{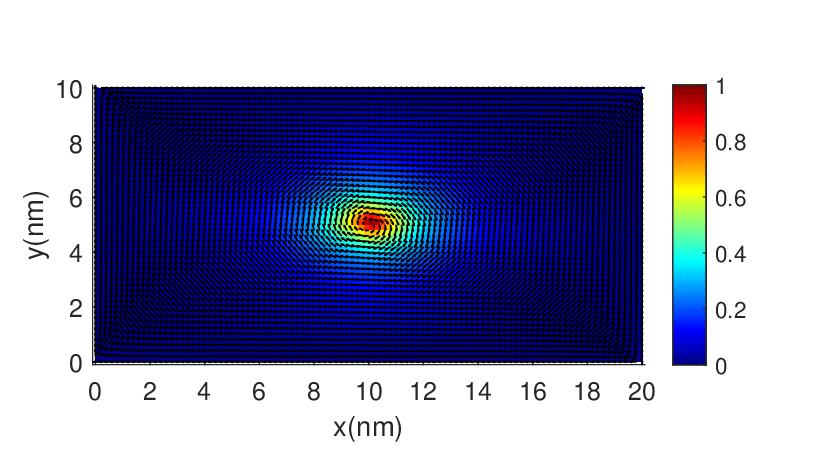}
\includegraphics[width=0.48\linewidth]{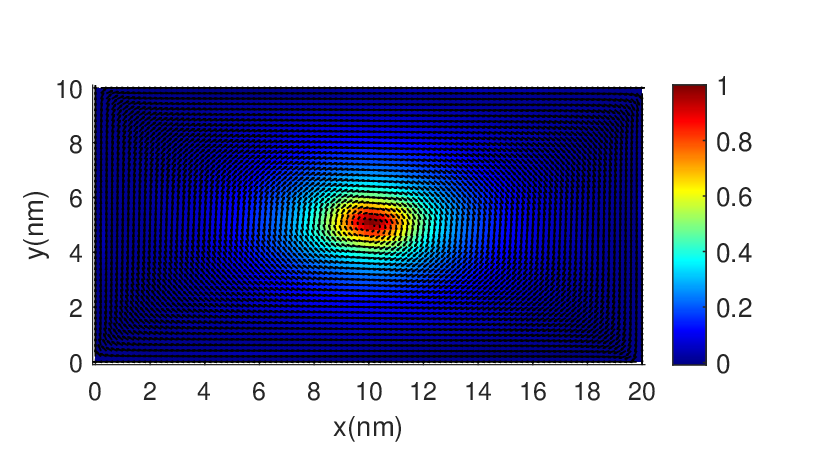}
\caption{The stable vortexes found by the LL equation with different damping parameters. Arrows denote the in-plane magnetization and the background color encodes the out-plane magnetization.}
\label{fig6}
\end{figure}

At the end of this section, we simulate several metastable magnetic states related to different initial magnetization distributions.
In \Cref{fig8}, two metastable magnetic states found by the LL equation with damping parameter $\alpha = 0.01$ are exhibited. We adopt the initial data as Figures 3.3 and 3.7 in \cite{yang2008current}, respectively. It is interesting that with the boundary condition \eqref{equ:boundary-micromagnetics}, two completely different stable states are found.
\begin{figure}[htbp]
\centering
\subfigure[Metastable magnetic state 1.]{\label{Fig8.sub.1}
\includegraphics[width=0.45\linewidth]{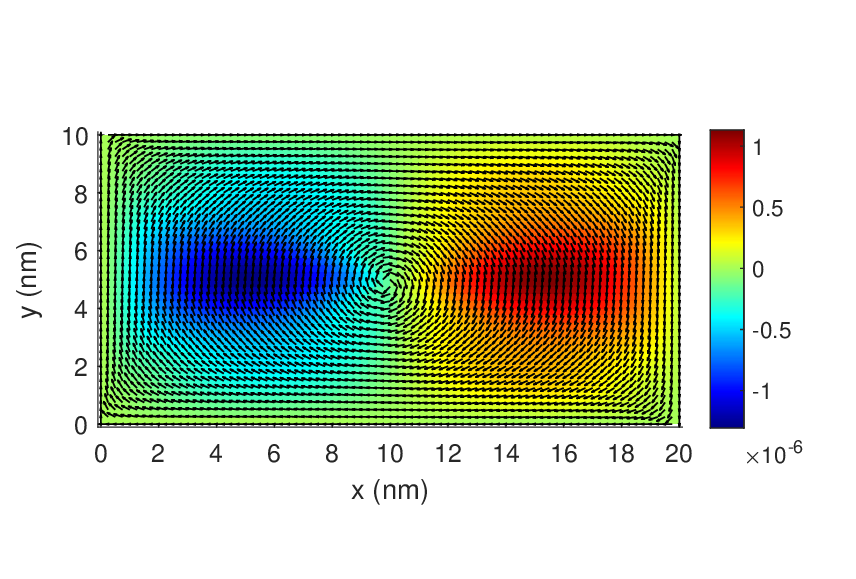}}
\subfigure[Metastable magnetic state 2.]{\label{Fig8.sub.2}
\includegraphics[width=0.45\linewidth]{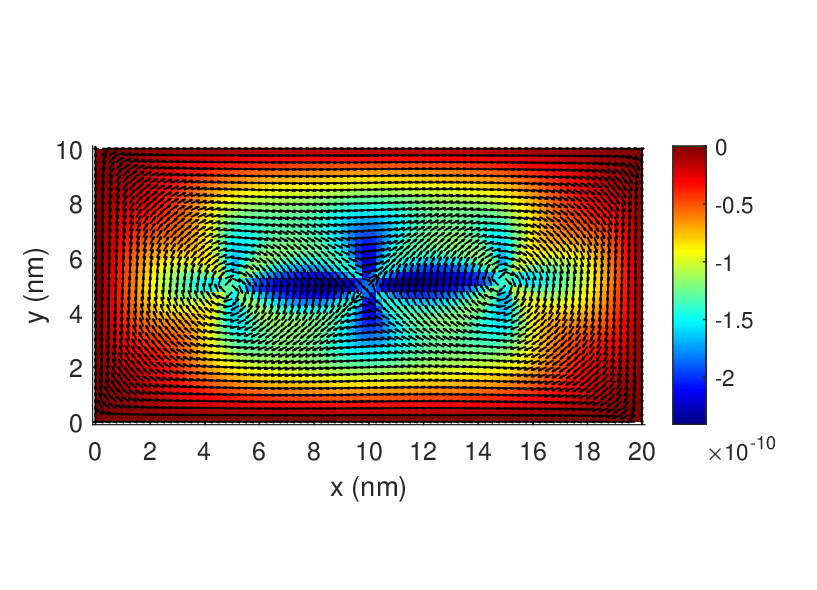}}
\caption{Metastable magnetic states found by the LL equation.
Arrows denote the in-plane magnetization and the background color encodes the out-plane magnetization.}
\label{fig8}
\end{figure}

\end{example}

\section{Conclusions\label{sec5:Conclusions}}
In this paper, we propose an efficient approach for simulating the Landau-Lifshitz (LL) equation, which combines the finite volume element method (FVEM) for spatial discretization and the Gauss-Seidel projection method (GSPM) for time-marching.
We provide a rigorous analysis for the FVEM approximation in both the $L^2$- and $H^1$-sense.
Owing to the employment of the GSPM, the computational costs are comparable to that of implicitly solving a scalar heat equation.
However, the theoretical analysis for the GSPM is still open.
We therefore carry out a series of numerical experiments to verify the convergence analysis and validate the feasibility of the proposed method.
Both error estimates and numerical results indicate that our method provides an efficient and reliable tool for simulating complex magnetization dynamics.

\appendix
\section{The proof of Lemma \ref{lem-con}} \label{app:lem-con}
Let $ \bmf{w}^l = {\bmf{m}}^{l+1} - {\bmf{m}}^l$.
According to \eqref{equ:LL-backward-Euler}, we have
\begin{equation}
\begin{aligned}
   & \frac{\bmf{w}^l}{\tau}
    -
    \alpha \Delta \bmf{w}^l
    +
    \nabla {\bmf{m}}^l \times \nabla \bmf{w}^l
    +
     {\bmf{m}}^l \times \Delta \bmf{w}^l
    +
    \nabla\bmf{w}^{l-1} \times \nabla  {\bmf{m}}^l
    +
    \bmf{w}^{l-1} \times \Delta  {\bmf{m}}^l\\
     = &
    \alpha|\nabla  {\bmf{m}}^l|^2 \bmf{w}^l
    +
    \alpha  {\bmf{m}}^l
    (|\nabla {\bmf{m}}^l|+|\nabla  {\bmf{m}}^{l-1}|) \nabla \bmf{w}^{l-1}.
    \label{equ:sequence-error}
\end{aligned}
\end{equation}
Now, we multiply both sides of \eqref{equ:sequence-error} by $ \bmf{w}^l $ and integrate over the entire domain $\Omega$. Applying Green's formula yields
\begin{equation}
\begin{aligned}
& \frac{1}{\tau} \|\bmf{w}^l\|_{\bmf{L}^2}^2
+
\alpha \|\nabla \bmf{w}^l\|_{\bmf{L}^2}^2 \\
 = &
(\bmf{w}^{l-1} \times \nabla {\bmf{m}}^l,\nabla \bmf{w}^l)
+
(\alpha|\nabla {\bmf{m}}^l|^2 \bmf{w}^l,\bmf{w}^l)
+
(\alpha {\bmf{m}}^l (|\nabla {\bmf{m}}^l|+|\nabla {\bmf{m}}^{l-1}|) \nabla \bmf{w}^{l-1},\bmf{w}^l)\\
 := &
\sum_{i=1}^{3}I_i.
\end{aligned}
\label{error-w1}
\end{equation}
Repeating the above procedure with the test functions $ \Delta \bmf{w}^l$ and $\Delta^2 \bmf{w}^l$,
we obtain
\begin{equation}
\begin{aligned}
& \frac{1}{\tau} \|\nabla \bmf{w}^l\|_{\bmf{L}^2}^2 + \alpha \|\Delta \bmf{w}^l\|_{\bmf{L}^2}^2 \\
 = &
(\nabla {\bmf{m}}^{l} \times \nabla \bmf{w}^l,\Delta \bmf{w}^l)
+
(\nabla \bmf{w}^{l-1} \times \nabla {\bmf{m}}^l,\Delta \bmf{w}^l)
+
(\bmf{w}^{l-1} \times \Delta {\bmf{m}}^l,\Delta \bmf{w}^l) \\
& + (\nabla(\alpha|\nabla {\bmf{m}}^l|^2 \bmf{w}^l),\nabla \bmf{w}^l)
-
(\alpha {\bmf{m}}^l (|\nabla {\bmf{m}}^l|+|\nabla {\bmf{m}}^{l-1}|) \nabla \bmf{w}^{l-1},\Delta \bmf{w}^l)\\
 := &
\sum_{i=4}^{8}I_i,
\end{aligned}
\label{error-w2}
\end{equation}
and
\begin{equation}
\begin{aligned}
\frac{1}{\tau} \|\Delta \bmf{w}^l\|_{\bmf{L}^2}^2
+
\alpha \|\nabla \Delta \bmf{w}^l\|_{\bmf{L}^2}^2
 = &
(\Delta {\bmf{m}}^l \times \nabla \bmf{w}^l, \nabla \Delta \bmf{w}^l)
+
2(\nabla {\bmf{m}}^l \times \Delta \bmf{w}^l, \nabla \Delta \bmf{w}^l) \\
&
+
(\Delta \bmf{w}^{l-1} \times \nabla {\bmf{m}}^l,\nabla \Delta \bmf{w}^{l})
+
2(\nabla \bmf{w}^{l-1} \times \Delta {\bmf{m}}^l,\nabla \Delta \bmf{w}^{l})\\
&
+
(\bmf{w}^{l-1} \times \nabla \Delta {\bmf{m}}^l,\nabla \Delta \bmf{w}^l)
+
(\nabla(\alpha|\nabla {\bmf{m}}^l|^2 \bmf{w}^l),\nabla \Delta \bmf{w}^l)\\
&
- \alpha( \nabla ({\bmf{m}}^l (|\nabla {\bmf{m}}^l|+|\nabla {\bmf{m}}^{l-1}|)) \nabla \bmf{w}^{l-1},\nabla \Delta \bmf{w}^l) \\
&
- (\alpha {\bmf{m}}^l (|\nabla {\bmf{m}}^l|+|\nabla {\bmf{m}}^{l-1}|) \Delta \bmf{w}^{l-1},\nabla \Delta \bmf{w}^l)\\
 := &
 \sum_{i=9}^{16}I_i.
\end{aligned}
\label{error-w3}
\end{equation}
By applying \eqref{Sobolev-inequality}, along with the H\"{o}lder and Young's inequalities, we obtain
\begin{align*}
    & |I_1|
    \leq
    \|\bmf{w}^{l-1}\|_{\bmf{L}^2} \|\nabla {\bmf{m}}^l\|_{\bmf{L}^{\infty}} \|\nabla \bmf{w}^l\|_{\bmf{L}^2}
    \leq
    C \epsilon^{-1} \|\bmf{w}^{l-1}\|_{\bmf{L}^2}^2
    +
    \epsilon \|\nabla \bmf{w}^l\|_{\bmf{L}^2}^2,\\
    & |I_2|
    \leq
    C \|\nabla {\bmf{m}}^l\|_{\bmf{L}^{\infty}}^2 \|\bmf{w}^l\|_{\bmf{L}^2}^2
    \leq
    C \|\bmf{w}^{l}\|_{\bmf{L}^2}^2,\\
    & |I_3|
    \leq
    \alpha \|{\bmf{m}}^l\|_{\bmf{L}^{\infty}}
    (\|\nabla {\bmf{m}}^l\|_{\bmf{L}^{\infty}}+\|\nabla {\bmf{m}}^{l-1}\|_{\bmf{L}^{\infty}}) \|\nabla \bmf{w}^{l-1}\|_{\bmf{L}^2} \|\bmf{w}^l\|_{\bmf{L}^2}
    \leq
    C \|\nabla \bmf{w}^{l-1}\|_{\bmf{L}^2}^2
    +
    C \|\bmf{w}^l\|_{\bmf{L}^2}^2,\\
     &|I_4|
   \leq
    \|\nabla {\bmf{m}}^{l}\|_{\bmf{L}^{\infty}} \|\nabla \bmf{w}^l\|_{\bmf{L}^2} \|\Delta \bmf{w}^l\|_{\bmf{L}^2}
    \leq
    C  \epsilon^{-1} \|\nabla \bmf{w}^l\|_{\bmf{L}^2}^2
    +
    \epsilon \|\Delta \bmf{w}^l\|_{\bmf{L}^2}^2,\\
    &|I_5|
    \leq
    \|\nabla \bmf{w}^{l-1}\|_{\bmf{L}^{2}} \|\nabla {\bmf{m}}^l\|_{\bmf{L}^{\infty}} \|\Delta \bmf{w}^l\|_{\bmf{L}^2}
    \leq
    C  \epsilon^{-1}  \|\nabla \bmf{w}^{l-1}\|_{\bmf{L}^2}^2
    +
    \epsilon \|\Delta \bmf{w}^l\|_{\bmf{L}^2}^2,\\
    &|I_6|
      \leq
    \|\bmf{w}^{l-1}\|_{\bmf{L}^{4}} \|\Delta {\bmf{m}}^l\|_{\bmf{L}^{4}} \|\Delta \bmf{w}^l\|_{\bmf{L}^2}
    \leq
    C  \epsilon^{-1} \|\bmf{w}^{l-1}\|_{\bmf{H}^1}^2
    +
    \epsilon \|\Delta \bmf{w}^l\|_{\bmf{L}^2}^2,\\
     &|I_7|
     \leq
    C \|\nabla {\bmf{m}}^l\|_{\bmf{L}^{\infty}}^2 \|\nabla\bmf{w}^l\|_{\bmf{L}^2}^2
    +
    C \|\nabla {\bmf{m}}^l\|_{\bmf{L}^{\infty}} \|\Delta {\bmf{m}}^l\|_{\bmf{L}^{4}}\|\bmf{w}^l\|_{\bmf{L}^4} \|\nabla \bmf{w}^l\|_{\bmf{L}^2}
    \leq
    C \|\nabla\bmf{w}^{l}\|_{\bmf{H}^1}^2 ,\\
    &|I_8|
    \leq
    \alpha \|{\bmf{m}}^l\|_{\bmf{L}^{\infty}}
    (\|\nabla {\bmf{m}}^l\|_{\bmf{L}^{\infty}}+\|\nabla {\bmf{m}}^{l-1}\|_{\bmf{L}^{\infty}}) \|\nabla \bmf{w}^{l-1}\|_{\bmf{L}^2}\|\Delta \bmf{w}^l\|_{\bmf{L}^2}
    \leq
    C  \epsilon^{-1}  \|\nabla \bmf{w}^{l-1}\|_{\bmf{L}^2}^2
    +
    \epsilon \|\Delta \bmf{w}^l\|_{\bmf{L}^2}^2,
\end{align*}
and
\begin{equation*}
\begin{aligned}
    &|I_9|
     \leq
    \|\Delta {\bmf{m}}^l\|_{\bmf{L}^4}\|\nabla \bmf{w}^l\|_{\bmf{L}^4} \|\nabla \Delta \bmf{w}^l\|_{\bmf{L}^2}
    \leq
    C  \epsilon^{-1} \|\nabla \bmf{w}^l\|_{\bmf{L}^2}^2
    +
    C  \epsilon^{-1} \|\Delta \bmf{w}^l\|_{\bmf{L}^2}^2 + \epsilon \|\nabla \Delta \bmf{w}^l\|_{\bmf{L}^2}^2,\\
   & |I_{10}|
    \leq
    2\|\nabla {\bmf{m}}^l\|_{\bmf{L}^{\infty}} \|\Delta \bmf{w}^l\|_{\bmf{L}^2} \|\nabla \Delta \bmf{w}^l\|_{\bmf{L}^2}
    \leq
    C  \epsilon^{-1} \|\Delta \bmf{w}^l\|_{\bmf{L}^2}^2
    +
    \epsilon \|\nabla \Delta \bmf{w}^l\|_{\bmf{L}^2}^2,\\
   & |I_{11}|
    \leq
    \|\Delta \bmf{w}^{l-1}\|_{\bmf{L}^2} \|\nabla {\bmf{m}}^l\|_{\bmf{L}^{\infty}} \|\nabla \Delta \bmf{w}^{l}\|_{\bmf{L}^2}
    \leq
    C  \epsilon^{-1} \|\Delta \bmf{w}^{l-1}\|_{\bmf{L}^2}^2
    +
    \epsilon \|\nabla \Delta \bmf{w}^{l}\|_{\bmf{L}^2}^2,  \\
   & |I_{12}|
     \leq
    2\|\nabla \bmf{w}^{l-1}\|_{\bmf{L}^4} \|\Delta {\bmf{m}}^l\|_{\bmf{L}^4} \|\nabla \Delta \bmf{w}^{l}\|_{\bmf{L}^2}
    \leq
    C  \epsilon^{-1} \| \nabla \bmf{w}^{l-1}\|_{\bmf{L}^2}^2
    +
    C  \epsilon^{-1} \| \Delta \bmf{w}^{l-1}\|_{\bmf{L}^2}^2
    +
    \epsilon \|\nabla \Delta \bmf{w}^{l}\|_{\bmf{L}^2}^2,\\
     &|I_{13}|
    \leq
    \|\bmf{w}^{l-1}\|_{\bmf{L}^{\infty}} \|\nabla \Delta {\bmf{m}}^l\|_{\bmf{L}^2} \|\nabla \Delta \bmf{w}^l\|_{\bmf{L}^2}
    \leq
    C  \epsilon^{-1} (\|\bmf{w}^{l-1}\|_{\bmf{L}^{2}}^2
    +
    \|\Delta \bmf{w}^{l-1}\|_{\bmf{L}^{2}}^2)
    +
    \epsilon \|\nabla \Delta \bmf{w}^l\|_{\bmf{L}^2}^2,
    \end{aligned}
\end{equation*}
\begin{align*}
    |I_{14}|
    &\leq
    C \|\nabla {\bmf{m}}^l\|_{\bmf{L}^{\infty}} \|\Delta {\bmf{m}}^l\|_{\bmf{L}^{2}} \|\bmf{w}^l\|_{\bmf{L}^{\infty}} \|\nabla \Delta \bmf{w}^l\|_{\bmf{L}^2}
    +
    C \|\nabla {\bmf{m}}^l\|_{\bmf{L}^{\infty}}^2 \|\nabla \bmf{w}^l\|_{\bmf{L}^2} \|\nabla \Delta \bmf{w}^l\|_{\bmf{L}^2}\\
    &\leq
    C (\|\bmf{w}^l\|_{\bmf{L}^2}^2 + \|\Delta \bmf{w}^l\|_{\bmf{L}^2}^2)^{\frac{1}{2}} \|\nabla \Delta \bmf{w}^l\|_{\bmf{L}^2}
    + C \|\nabla \bmf{w}^l\|_{\bmf{L}^2} \|\nabla \Delta \bmf{w}^l\|_{\bmf{L}^2} \\
    &\leq
    C  \epsilon^{-1} \|\bmf{w}^l\|_{\bmf{L}^2}^2
    +
    C  \epsilon^{-1} \|\Delta \bmf{w}^l\|_{\bmf{L}^2}^2
    +
    C  \epsilon^{-1} \|\nabla \bmf{w}^l\|_{\bmf{L}^2}^2
    +
    2 \epsilon \|\nabla \Delta \bmf{w}^l\|_{\bmf{L}^2}^2,\\
    &\leq
    C  \epsilon^{-1} \|\bmf{w}^l\|_{\bmf{H}^2}^2
    + 2 \epsilon \|\nabla \Delta \bmf{w}^l\|_{\bmf{L}^2}^2,
\end{align*}

\begin{align*}
    |I_{15}|
     \leq &
    C \|\nabla {\bmf{m}}^l\|_{\bmf{L}^{\infty}}
    (\|\nabla {\bmf{m}}^l\|_{\bmf{L}^{\infty}} + \|\nabla {\bmf{m}}^{l-1}\|_{\bmf{L}^{\infty}}) \|\nabla \bmf{w}^{l-1}\|_{\bmf{L}^2} \|\nabla \Delta \bmf{w}^l\|_{\bmf{L}^{2}} \\
    & + C \|\nabla {\bmf{m}}^l\|_{\bmf{L}^{\infty}}
    (\|\nabla {\bmf{m}}^l\|_{\bmf{L}^{4}} + \|\nabla {\bmf{m}}^{l-1}\|_{\bmf{L}^{4}}) \|\nabla \bmf{w}^{l-1}\|_{\bmf{L}^4} \|\nabla \Delta \bmf{w}^l\|_{\bmf{L}^{2}} \\
    \leq & C \|\nabla \bmf{w}^{l-1}\|_{\bmf{L}^2} \|\nabla \Delta \bmf{w}^l\|_{\bmf{L}^2}
    +
    C \|\nabla \bmf{w}^{l-1}\|_{\bmf{L}^4} \|\nabla \Delta \bmf{w}^l\|_{\bmf{L}^2} \\
    \leq &
    C  \epsilon^{-1} \|\nabla \bmf{w}^{l-1}\|_{\bmf{L}^2}
    +
    C  \epsilon^{-1} \|\Delta \bmf{w}^{l-1}\|_{\bmf{L}^2}^2
    +
    2 \epsilon \|\nabla \Delta \bmf{w}^l\|_{\bmf{L}^2}^2,
\end{align*}
\begin{equation*}
\begin{aligned}
    |I_{16}|&
    \leq
    \alpha \|{\bmf{m}}^l\|_{\bmf{L}^{\infty}}
    (\|\nabla {\bmf{m}}^l\|_{\bmf{L}^{\infty}}+\|\nabla {\bmf{m}}^{l-1}\|_{\bmf{L}^{\infty}}) \|\Delta \bmf{w}^{l-1}\|_{\bmf{L}^2} \|\nabla \Delta \bmf{w}^l\|_{\bmf{L}^2}\\
    &\leq
    C  \epsilon^{-1} \|\Delta \bmf{w}^{l-1}\|_{\bmf{L}^2}^2
    +
    \epsilon \|\nabla \Delta \bmf{w}^l\|_{\bmf{L}^2}^2.
\end{aligned}
\end{equation*}
We used the bound of $\|\bmf{m}\|_{L^{\infty}(0,T;\bmf{H}^3)}$ in the above derivations, 
and therefore the constant $C$ depends on $C_{\mathrm{reg}}$.  

Using these estimates of $I_i$, $i=1,2,\cdots,16$,
and summing up inequalities \eqref{error-w1}-\eqref{error-w3} yield
\begin{equation}
\frac{1}{\tau}\|\bmf{w}^l\|_{\bmf{H}^2}^2 + \alpha \| \bmf{w}^l\|_{\bmf{H}^3}^2
\leq
10 \epsilon \|\bmf{w}^l\|_{\bmf{H}^3}^2
+
C  \epsilon^{-1} \|\bmf{w}^l\|_{\bmf{H}^2}^2
+
C \epsilon^{-1}  \|\bmf{w}^{l-1}\|_{\bmf{H}^2}^2.
\label{equ:LL-linearized-H2-2}
\end{equation}
Choosing $\epsilon= \frac{\alpha}{10}$ and rearranging \eqref{equ:LL-linearized-H2-2}, we obtain
\begin{equation}
\|\bmf{w}^l\|_{\bmf{H}^2}
\leq
\sqrt{\frac{C\tau}{1-C\tau}} \|\bmf{w}^{l-1}\|_{\bmf{H}^2},
    \label{equ:LL-linearized-H1}
\end{equation}
where the constant $C$ depends on $C_{\mathrm{reg}}$. 

It indicates that for sufficiently small $\tau$, $ \|\bmf{w}^l\|_{\bmf{H}^2} $ converges to 0 as $l \rightarrow \infty$.

\section{The proof of Lemma \ref{lem-Ah}} \label{app:lem-Ah}
Based on the definition of $\mathcal{A}_h(\bmf{\Phi};\cdot, I^*_h\cdot)$,
we divide Lemma \ref{lem-Ah} into the following two lemmas.
\begin{lemma}
Let $\bmf{\Phi} \in \bmf{X}$. For sufficiently small $h$,
there exists a positive constant $C$ such that
for all $\tilde{\bmf{m}}_h, \bmf{v}_h\in \bmf{S}_h$
\begin{gather}
    a_h(\tilde{\bmf{m}}_h,I_h^*\tilde{\bmf{m}}_h)
    \geq
    \alpha \|\tilde{\bmf{m}}_h\|_{\bmf{H}^1}^2,    \label{ah-coe}\\
    |a_h(\tilde{\bmf{m}}_h,I_h^*\bmf{v}_h)|
    \leq
    C\|\tilde{\bmf{m}}_h\|_{\bmf{H}^1}\|\bmf{v}_h\|_{\bmf{H}^1}.
    \label{ah-boudary}
\end{gather}
\label{lem-ah}
\end{lemma}
The proofs of \eqref{ah-coe} and \eqref{ah-boudary} can be found in \cite{lironghua}.
\begin{lemma}
Let $\bmf{\Phi} \in \bmf{X}$. For sufficiently small $h$,
there exists a positive constant $C$ such that,
for all $\tilde{\bmf{m}}_h, \bmf{v}_h\in \bmf{S}_h$,
\begin{gather}
    b_h(\bmf{\Phi};\tilde{\bmf{m}}_h,I_h^*\bmf{m}_h)
    =
    0,    \label{bh_coe}\\
    | b_h(\bmf{\Phi};\tilde{\bmf{m}}_h,I_h^*\bmf{v}_h)|
    \leq
    C \|\tilde{\bmf{m}}_h\|_{\bmf{H}^1}\|\bmf{v}_h\|_{\bmf{H}^1}.    \label{bh-boundary}
\end{gather}
\label{lem-bh}
\end{lemma}
\begin{proof}
We begin the proof by rewriting \( b_h(\bmf{\Phi}; \tilde{\bmf{m}}_h, I_h^* \bmf{v}_h) \) as
\begin{equation*}
b_h(\bmf{\Phi};\tilde{\bmf{m}}_h,I_h^*\bmf{v}_h)
=
\sum_{K_Q \in T_h} I_{K_Q}(\bmf{\Phi};\tilde{\bmf{m}}_h,I_h^*\bmf{v}_h),
\end{equation*}
where
\begin{equation*}
I_{K_Q}(\bmf{\Phi};\tilde{\bmf{m}}_h,I_h^*\bmf{v}_h)
=
\sum_{P \in \dot{K}_Q} \left[\int_{\partial V_i \cap K_Q} \!\!\!\!\!\!\!\!\!\!\bmf{\Phi}\times \frac{\partial \tilde{\bmf{m}}_h}{\partial \bmf{\nu}} \,\mr{d}S \right]\,\bmf{v}_h(P)
\end{equation*}
with $\dot{K}_Q$ denotes the set of vertexes of $K_Q = \triangle P_1P_2P_3$.
 \begin{figure}[htp]
		\centering
		 \begin{tikzpicture}[>=stealth]
		 \draw[dash pattern = on 5pt off 2.5pt, line width = 1.5pt] (0,0) --  (-0.75,0.5);
		 \draw[dash pattern = on 5pt off 2.5pt, line width = 1.5pt] (0,0) -- (-0.25,-1);
		 \draw[dash pattern = on 5pt off 2.5pt, line width = 1.5pt] (0,0) -- (1,0.5);
        \draw[line width = 1.5pt] (0.5,2) to (-2,-1) to (1.5,-1) to (0.5,2);
         \node at (0,0.3){$Q$};
         \node at (0.5,2.2){$P_1$};
         \node at (-2.3,-1){$P_2$};
         \node at (1.8,-1){$P_3$};
         \node at (-1.1,0.6){$M_1$};
         \node at (-0.25,-1.3){$M_2$};
         \node at (1.3,0.6){$M_3$};
        \end{tikzpicture}
		\caption{Triangle element $K_Q$ ($\triangle P_{1}P_{2}P_{3}$).}
  \label{Fig2}
\end{figure}
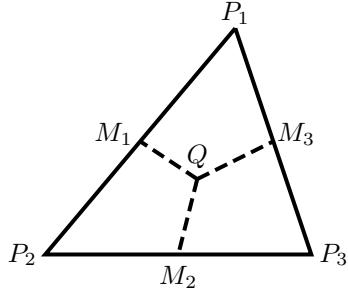	

For each triangular element $K_Q$ (\Cref{Fig2}) in the triangulation $T_h$,
the following equality holds \cite{lironghua}
\begin{equation}
\begin{aligned}
    \frac{\partial \tilde{\bmf{m}}_h}{\partial x}
    & =  \frac{1}{2S_{K_Q}}
    \sum_{i=1}^3 \tilde{\bmf{m}}_h(P_i)(y_{M_{i}}-y_{M_{i+2}}),\\
    \frac{\partial \tilde{\bmf{m}}_h}{\partial y}
    & =   \frac{1}{2S_{K_Q}}
    \sum_{i=1}^3 \tilde{\bmf{m}}_h(P_i)(x_{M_{i+2}}-x_{M_i}),
    \label{triangle_area}
\end{aligned}
\end{equation}
where $S_{K_Q}$ is the area of triangle element $K_Q$, $ \ M_4 = M_1$ and $\ M_5 = M_2$.
Since $ \frac{\partial \tilde{\bmf{m}}_h}{\partial x} $ and $ \frac{\partial\tilde{\bmf{m}}_h}{\partial y}$ are constants in the triangular element $K_Q$,
we have
\begin{equation*}
\begin{aligned}
& I_{K_Q} (\bmf{\Phi};\tilde{\bmf{m}}_h,I_h^*\tilde{\bmf{m}}_h)\\ \nn
 =
 & \sum_{i=1}^3 \tilde{\bmf{m}}_h(P_i) \cdot \int_{\overline{M_iQM_{i+2}}} \bmf{\Phi} \times \frac{\partial \tilde{\bmf{m}}_h(Q)}{\partial \bmf{\nu}} \,\mr{d}S \\
=
& \sum_{i=1}^3 \tilde{\bmf{m}}_h(P_i) \cdot \left[ \int_{\overline{M_iQM_{i+2}}} \bmf{\Phi} \times \frac{\partial \tilde{\bmf{m}}_h(Q)}{\partial x} \,\mr{d}y
- \int_{\overline{M_iQM_{i+2}}} \bmf{\Phi} \times \frac{\partial \tilde{\bmf{m}}_h(Q)}{\partial y} \, \mr{d}x \right] \\
=
& \sum_{i=1}^3 \tilde{\bmf{m}}_h(P_i) \cdot \left[ \left(\bmf{\Phi} \times \frac{\partial \tilde{\bmf{m}}_h(Q)}{\partial x}\right) (y_{M_{i+2}}-y_{M_i})
- \left(\bmf{\Phi} \times \frac{\partial \tilde{\bmf{m}}_h(Q)}{\partial y}\right) (x_{M_{i+2}}-x_{M_i})\right]\\ \nn
=
& - 2 S_{K_Q} \left[\left(\bmf{\Phi} \times \frac{\partial \tilde{\bmf{m}}_h(Q)}{\partial x}\right) \cdot \frac{\partial \tilde{\bmf{m}}_h(Q)}{\partial x} + \left(\bmf{\Phi} \times \frac{\partial \tilde{\bmf{m}}_h(Q)}{\partial y}\right) \cdot \frac{\partial \tilde{\bmf{m}}_h(Q)}{\partial y}\right]\\ \nn
=
& 0.
\end{aligned}
\end{equation*}
It therefore holds $b_h(\bmf{\Phi};\tilde{\bmf{m}}_h,I_h^*\tilde{\bmf{m}}_h) =  0$.

Next, we estimate the bound of $b_h(\bmf{\Phi};\cdot,I_h^{*}\cdot)$.
\begin{equation*}
\begin{aligned}
&  I_{K_Q} (\bmf{\Phi};\tilde{\bmf{m}}_h,I_h^*\bmf{v}_h)\\ \nn
=
& - \sum_{i=1}^3 \bmf{v}_h(P_1)\cdot\int_{\overline{M_iQM_{i+2}}} \bmf{\Phi}\times \frac{\partial \tilde{\bmf{m}}_h}{\partial \bmf{\nu}} \, \mr{d}S
\\
=
& \sum_{i=1}^3  (\bmf{v}_h(P_{i+1})-\bmf{v}_h(P_i))\cdot\int_{\overline{M_iQ}}\bmf{\Phi}\times \frac{\partial \tilde{\bmf{m}}_h}{\partial \bmf{\nu}} \, \mr{d}S,
\end{aligned}
\end{equation*}
where $ P_4 = P_1 $. Denote $l_i=|\overline{P_{i}P_{i+1}}|$,
$\tau_i=\frac{\overrightarrow{P_{i}P_{i+1}}}{l_i}$, ($ i = 1,2,3 $), and
we have
\begin{equation*}
\frac{\bmf{v}_h(P_{i+1})-\bmf{v}_h(P_{i})}{l_i}
=
\frac{\partial \bmf{v}_h}{\partial \tau_i}.
\end{equation*}
We then arrive at
\begin{equation*}
\begin{aligned}
& b_h(\bmf{\Phi};\tilde{\bmf{m}}_h,I_h^*\bmf{v}_h)\\
 = &
 \sum_{{K_Q}\in T_h} \sum_{i=1}^{3} l_i \int_{\overline{M_i Q}}\left(\bmf{\Phi}\times \frac{\partial \tilde{\bmf{m}}_h}{\partial \bmf{\nu}}\right)\cdot \frac{\partial \bmf{v}_h}{\partial \tau_i} \,\mr{d}S\\
 \leq &
\sum_{{K_Q}\in T_h} \sum_{i=1}^{3} l_i |\bmf{\Phi}(Q)| \left|\frac{\partial \tilde{\bmf{m}}_h(Q)}{\partial \bmf{\nu}}\right| \left|\frac{\partial \bmf{v}_h(Q)}{\partial \tau_i}\right| |\overline{M_i Q}|\\
 \leq & \
C \sum_{{K_Q}\in T_h} h^2
\left(\left|\frac{\partial \tilde{\bmf{m}}_h(Q)}{\partial x}\right| + \left|\frac{\partial \tilde{\bmf{m}}_h(Q)}{\partial y}\right|\right)
\left(\left|\frac{\partial \bmf{v}_h(Q)}{\partial x}\right| + \left|\frac{\partial \bmf{v}_h(Q)}{\partial y} \right|\right) \\
 \leq & \
C \,\| \tilde{\bmf{m}}_h \|_{\bmf{H}^1}  \| \bmf{v}_h \|_{\bmf{H}^1}.
\end{aligned}
\end{equation*}
\end{proof}
By Lemma \ref{lem-ah} and Lemma \ref{lem-bh},
we get
\begin{equation*}
\begin{aligned}
&\mathcal{A}_h(\bmf{\Phi};\tilde{\bmf{m}}_h,I_h^*\tilde{\bmf{m}}_h)
=
a_h(\tilde{\bmf{m}}_h,I_h^*\tilde{\bmf{m}}_h)
+
b_h(\bmf{\Phi};\tilde{\bmf{m}}_h,I_h^*\tilde{\bmf{m}}_h)
\geq
C \|\tilde{\bmf{m}}_h\|_{\bmf{H}^1}^2,\\
&|\mathcal{A}_h(\bmf{\Phi};\tilde{\bmf{m}}_h,I_h^*\bmf{v}_h)|
\leq
|a_h(\tilde{\bmf{m}}_h,I_h^*\bmf{v}_h)|+|b_h(\bmf{\Phi};\tilde{\bmf{m}}_h,I_h^*\bmf{v}_h)|
\leq
C \| \tilde{\bmf{m}}_h \|_{\bmf{H}^1}  \| \bmf{v}_h \|_{\bmf{H}^1}.
\end{aligned}
\end{equation*}
These complete the proof.

\section{The proof of Lemma \ref{lem-Ah1}}\label{app:lem-Ah1}
Repeat the proof of Lemma 2.4 in \cite{chou},
we obtain
\begin{equation}
|a_h(\tilde{\bmf{m}}_h,I_h^*\bmf{v}_h)-a_h(\bmf{v}_h,I_h^*\tilde{\bmf{m}}_h)|
\leq
C h\|\tilde{\bmf{m}}_h\|_{\bmf{H}^1}\|\bmf{v}_h\|_{\bmf{H}^1}, \quad \forall \tilde{\bmf{m}}_h,\bmf{v}_h\in \bmf{S}_h.
\label{ah-relation}
\end{equation}
Hence,
we focus solely on estimating the bound of $|b_h(\bmf{\Phi};\tilde{\bmf{m}}_h,I_h^*\bmf{v}_h) - b_h(\bmf{\Phi};\bmf{v}_h,I_h^*\tilde{\bmf{m}}_h)|$.

For $\tilde{\bmf{m}}\in \bmf{H}^2$, $\tilde{\bmf{m}}_h,\bmf{v}_h \in \bmf{S}_h$,
applying Green's formula yields
\begin{equation}
\begin{aligned}
 b(\bmf{\Phi};\tilde{\bmf{m}}-\tilde{\bmf{m}}_h,\bmf{v}_h)
    = &
   -\sum_{K_Q\in T_h} \int_{K_Q} \left(\bmf{\Phi}\times\nabla(\tilde{\bmf{m}}-\tilde{\bmf{m}}_h)\right) \cdot \nabla\bmf{v}_h \, \mr{d}x\mr{d}y \\
    = &
    -\sum_{K_Q\in T_h} \int_{K_Q} \left((\bmf{\Phi}\!-\!\bmf{\Phi}(Q))\times\nabla(\tilde{\bmf{m}}-\tilde{\bmf{m}}_h)\right) \cdot \nabla\bmf{v}_h \mr{d}x\mr{d}y \\
   & +
   \sum_{K_Q\in T_h} \int_{K_Q} \left(\bmf{\Phi}(Q)\times\Delta \tilde{\bmf{m}}\right) \cdot \bmf{v}_h \, \mr{d}x\mr{d}y\\
   & -
   \sum_{K_Q\in T_h} \int_{K_Q} \left(\nabla\bmf{\Phi}(Q)\times\nabla \tilde{\bmf{m}}\right) \cdot \bmf{v}_h \, \mr{d}x\mr{d}y\\
   & +
   \int_{\partial{K_Q}} \left(\bmf{\Phi}(Q)\times \frac{\partial(\tilde{\bmf{m}}-\tilde{\bmf{m}}_h)}{\partial\bmf{\nu}}\right) \cdot\bmf{v}_h \, \mr{d}S.
\end{aligned}
\label{b-1}
\end{equation}
Meanwhile, for $\tilde{\bmf{m}}\in \bmf{H}^2$, $\bmf{v}_h \in \bmf{S}_h$, we have
\begin{equation*}
\begin{aligned}
   & \sum_{K_Q \in T_h} \int_{K_Q} \left(\bmf{\Phi}(Q)\times\Delta \tilde{\bmf{m}}\right) \cdot \bmf{v}_h \, \mr{d}x\mr{d}y \\
   = &
   \sum_{K_Q \in T_h} \sum_{P\in \dot{K}_Q} \int_{{K_Q} \bigcap V_i} \left(\bmf{\Phi}(Q)\times\Delta \tilde{\bmf{m}}\right) \cdot \bmf{v}_h \, \mr{d}x\mr{d}y\\
   = &
   \sum_{K_Q \in T_h} \sum_{P\in \dot{K}_Q} \left[\int_{\partial({K_Q} \bigcap V_i)} \left(\bmf{\Phi}(Q)\times \frac{\partial\tilde{\bmf{m}}}{\partial\bmf{\nu}}\right) \cdot\bmf{v}_h \, \mr{d}S
   -
   \int_{{K_Q} \bigcap V_i} \left(\nabla\bmf{\Phi}(Q)\times\nabla \tilde{\bmf{m}}\right) \cdot\bmf{v}_h \, \mr{d}x\mr{d}y \right]\\
   =  &
   \sum_{K_Q \in T_h} \sum_{P\in \dot{K}_Q} \int_{\partial V_i \bigcap {K_Q}} \left(\bmf{\Phi}(Q)\times \frac{\partial\tilde{\bmf{m}}}{\partial\bmf{\nu}}\right) \cdot\bmf{v}_h \, \mr{d}S
   -
    \sum_{K_Q \in T_h}\int_{{K_Q}} \left(\nabla\bmf{\Phi}(Q)\times\nabla \tilde{\bmf{m}}\right) \cdot \bmf{v}_h \, \mr{d}x\mr{d}y\\
   & +
   \sum_{K_Q \in T_h} \int_{\partial{K_Q}} \left(\bmf{\Phi}(Q)\times \frac{\partial\tilde{\bmf{m}}}{\partial\bmf{\nu}}\right)  \cdot \bmf{v}_h \, \mr{d}S.
\end{aligned}
\end{equation*}
Let $\bmf{v}_h = I_h^*\bmf{v}_h$ in \eqref{b-1} and we obtain
\begin{equation*}
\begin{aligned}
 b_h(\bmf{\Phi};\tilde{\bmf{m}}-\tilde{\bmf{m}}_h,I_h^*\bmf{v}_h)
    = &
   \sum_{K_Q \in T_h} \sum_{P\in \dot{K}_Q}  \int_{\partial V_i \bigcap K_Q} \left(\bmf{\Phi}\times \frac{\partial(\tilde{\bmf{m}}-\tilde{\bmf{m}}_h)}{\partial\bmf{\nu}}\right) \cdot I_h^*\bmf{v}_h \, \mr{d}S \\
    = &
   \sum_{K_Q \in T_h} \sum_{P\in \dot{K}_Q}  \int_{\partial V_i \bigcap K_Q} \left(\left(\bmf{\Phi}-\bmf{\Phi}(Q)\right)\times \frac{\partial(\tilde{\bmf{m}}-\tilde{\bmf{m}}_h)}{\partial\bmf{\nu}}\right) \cdot I_h^*\bmf{v}_h \, \mr{d}S   \\
   & +
   \sum_{K_Q\in T_h}\int_{K_Q} \left(\bmf{\Phi}(Q)\times\Delta \tilde{\bmf{m}}\right) \cdot I_h^*\bmf{v}_h \, \mr{d}x\mr{d}y\\
   &
   -
   \sum_{K_Q\in T_h}\int_{K_Q} \left(\nabla\bmf{\Phi}(Q)\times\nabla \tilde{\bmf{m}}\right) \cdot I_h^*\bmf{v}_h \, \mr{d}x\mr{d}y\\
  & +
   \sum_{K_Q\in T_h}\int_{\partial{K_Q}} \left(\bmf{\Phi}(Q)\times \frac{\partial(\tilde{\bmf{m}}-\tilde{\bmf{m}}_h)}{\partial\bmf{\nu}}\right) \cdot I_h^*\bmf{v}_h \, \mr{d}S.
\end{aligned}
\end{equation*}
Hence, we have
\begin{equation*}
     b(\bmf{\Phi};\tilde{\bmf{m}}-\tilde{\bmf{m}}_h,\bmf{v}_h) - b_h(\bmf{\Phi};\tilde{\bmf{m}}-\tilde{\bmf{m}}_h,I_h^*\bmf{v}_h)
     =
     \sum_{i=1}^5 E_i(\bmf{\Phi};\tilde{\bmf{m}}-\tilde{\bmf{m}}_h,\bmf{v}_h),
\end{equation*}
where
\begin{equation*}
     E_1(\bmf{\Phi};\tilde{\bmf{m}}-\tilde{\bmf{m}}_h,\bmf{v}_h)
     =
     -\sum_{K_Q\in T_h} \int_{K_Q} \left((\bmf{\Phi}\!-\!\bmf{\Phi}(Q))\times\nabla(\tilde{\bmf{m}}-\tilde{\bmf{m}}_h)\right) \cdot \nabla\bmf{v}_h \mr{d}x\mr{d}y,
\end{equation*}
\begin{equation*}
    E_2(\bmf{\Phi};\tilde{\bmf{m}}-\tilde{\bmf{m}}_h,\bmf{v}_h)
    =
    \sum_{K_Q \in T_h} \sum_{P\in \dot{K}_Q}  \int_{\partial V_i \bigcap K_Q} \left(\bmf{\Phi}\times \frac{\partial(\tilde{\bmf{m}}-\tilde{\bmf{m}}_h)}{\partial\bmf{\nu}} \right) \cdot I_h^*\bmf{v}_h \, \mr{d}S,
\end{equation*}
\begin{equation*}
    E_3(\bmf{\Phi};\tilde{\bmf{m}}-\tilde{\bmf{m}}_h,\bmf{v}_h)
    =
     \sum_{K_Q\in T_h} \int_{K_Q} (\bmf{\Phi}(Q)\times\Delta \tilde{\bmf{m}}) \cdot (\bmf{v}_h-I_h^*\bmf{v}_h) \, \mr{d}x\mr{d}y,
\end{equation*}
\begin{equation*}
    E_4(\nabla\bmf{\Phi};\tilde{\bmf{m}}-\tilde{\bmf{m}}_h,\bmf{v}_h)
    =
    -\sum_{K_Q\in T_h} \int_{K_Q} (\nabla\bmf{\Phi}(Q)\times\nabla \tilde{\bmf{m}}) \cdot
    (\bmf{v}_h-I_h^*\bmf{v}_h) \, \mr{d}x\mr{d}y,
\end{equation*}
\begin{equation*}
    E_5(\bmf{\Phi};\tilde{\bmf{m}}-\tilde{\bmf{m}}_h,\bmf{v}_h)
    =
    \int_{\partial{K_Q}} \left(\bmf{\Phi}(Q)\times \frac{\partial(\tilde{\bmf{m}}-\tilde{\bmf{m}}_h)}{\partial\bmf{\nu}}\right) \cdot (\bmf{v}_h-I_h^*\bmf{v}_h) \, \mr{d}S.
\end{equation*}
We next derive the bounds of $E_i$.
For $E_1$, since $\bmf{\Phi} \in \bmf{X}$, it holds
\begin{equation*}
    |E_1(\bmf{\Phi};\tilde{\bmf{m}}-\tilde{\bmf{m}}_h,\bmf{v}_h)| \leq C h \|\tilde{\bmf{m}}-\tilde{\bmf{m}}_h\|_{\bmf{H}^1}\|\bmf{v}_h\|_{\bmf{H}^1}.
\end{equation*}
$E_2(\bmf{\Phi};\tilde{\bmf{m}}-\tilde{\bmf{m}}_h,\bmf{v}_h)$ can be rewritten into
\begin{equation*}
\left|E_2(\bmf{\Phi};\tilde{\bmf{m}}-\tilde{\bmf{m}}_h,\bmf{v}_h) \right|
=
 \left|\sum_{K_Q \in T_h} \sum_{i=1}^3  (\bmf{v}_h(P_{i+1})-\bmf{v}_h(P_i))\cdot\int_{\overline{M_iQ}}(\bmf{\Phi}-\bmf{\Phi}(Q))\times \frac{\partial(\tilde{\bmf{m}}- \tilde{\bmf{m}}_h)}{\partial \bmf{\nu}} \, \mr{d}S\right|,
\end{equation*}
where $P_i$ denotes the vertex of triangle element $K_Q$, and $P_4=P_1$. Since $\bmf{v}_h$ is linear in $K_Q$, and using the Taylor's expansion, we get
\begin{equation*}
    |\bmf{v}_h(P_{i+1})-\bmf{v}_h(P_i)|
    =
    \left|\frac{\partial\bmf{v}_h}{\partial x} (x_{P_i}-x_{P_{i+1}})+ \frac{\partial\bmf{v}_h}{\partial y} (y_{P_i}-y_{P_{i+1}}) \right|
    \leq
    h \left(\left|\frac{\partial\bmf{v}_h}{\partial x}\right| + \left| \frac{\partial\bmf{v}_h}{\partial y} \right|\right).
\end{equation*}
Combined with the inequality \cite{chou}
\begin{equation*}
    \int_{\overline{M_iQ}} \left( \frac{\partial(\tilde{\bmf{m}}- \tilde{\bmf{m}}_h)}{\partial \bmf{\nu}}\right)^2 \, \mr{d}S
    \leq
    C(h^{-1} |\tilde{\bmf{m}} - \tilde{\bmf{m}}_h|_{\bmf{H}^1(K_Q)}^2 + |\tilde{\bmf{m}}|_{\bmf{H}^2(K_Q)}|\tilde{\bmf{m}} - \tilde{\bmf{m}}_h|_{\bmf{H}^1(K_Q)}),
\end{equation*}
it leads to an estimate
\begin{equation*}
    \begin{aligned}
& |E_2(\bmf{\Phi};\tilde{\bmf{m}}-\tilde{\bmf{m}}_h,\bmf{v}_h)|\\
\leq &
 \sum_{K_Q \in T_h} \sum_{i=1}^3  |\bmf{v}_h(P_{i+1})-\bmf{v}_h(P_i)|\int_{\overline{M_iQ}}|\bmf{\Phi}-\bmf{\Phi}(Q)| \left|\frac{\partial(\tilde{\bmf{m}}- \tilde{\bmf{m}}_h)}{\partial \bmf{\nu}}\right| \, \mr{d}S\\
  \leq &
 C h (\|\tilde{\bmf{m}}-\tilde{\bmf{m}}_h\|_{\bmf{H}^1}+h^{1/2}\|\tilde{\bmf{m}}-\tilde{\bmf{m}}_h\|_{\bmf{H}^1}^{1/2}\|\tilde{\bmf{m}}\|_{\bmf{H}^2}^{1/2})\|\bmf{v}_h\|_{\bmf{H}^1}.
\end{aligned}
\end{equation*}
Owing to
\begin{equation*}
    \|\bmf{v}_h-I_h^*\bmf{v}_h\|_{\bmf{L}^2} \leq h \|\bmf{v}_h\|_{\bmf{H}^1},
    \, \forall \, \bmf{v}_h \in \bmf{S}_h,
\end{equation*}
we obtain
\begin{equation*}
    |E_3(\bmf{\Phi};\tilde{\bmf{m}}-\tilde{\bmf{m}}_h,\bmf{v}_h)| \leq C h \|\tilde{\bmf{m}}\|_{\bmf{H}^2}\|\bmf{v}_h\|_{\bmf{H}^{1}},
\end{equation*}
and
\begin{equation*}
    |E_4(\nabla\bmf{\Phi};\tilde{\bmf{m}}-\tilde{\bmf{m}}_h,\bmf{v}_h)|
    \leq C h \|\tilde{\bmf{m}}\|_{\bmf{H}^1}\|\bmf{v}_h\|_{\bmf{H}^{1}}.
\end{equation*}
Since $ \frac{\partial\bmf{v}_h}{\partial\bmf{\nu}} $ is a constant along with edges of the element $K_Q$, i.e.,
\begin{equation}
    \int_{\partial K_Q}(\bmf{v}_h-I_h^*\bmf{v}_h)\,\mr{d}S = 0,
    \label{pk}
\end{equation}
for $E_5$, we have
\begin{equation*}
\begin{aligned}
    E_5(\bmf{\Phi};\tilde{\bmf{m}}-\tilde{\bmf{m}}_h,\bmf{v}_h)
    =&
    \int_{\partial{K_Q}} \left(\bmf{\Phi}(Q)\times \frac{\partial(\tilde{\bmf{m}}-\tilde{\bmf{m}}_h)}{\partial\bmf{\nu}}\right) \cdot (\bmf{v}_h-I_h^*\bmf{v}_h) \, \mr{d}S\\
    =&
    \int_{\partial{K_Q}} \left(\bmf{\Phi}(Q)\times \frac{\partial\tilde{\bmf{m}}}{\partial\bmf{\nu}}\right) \cdot (\bmf{v}_h-I_h^*\bmf{v}_h) \, \mr{d}S.
\end{aligned}
\end{equation*}
Let $N_e$ denote the number of all triangle elements $K_Q$, and let $L$ denote the common edge shared by two adjacent elements $K_{Q_i}$ and $K_{Q_{i+1}}$.
Then, we have
\begin{equation*}
\begin{aligned}
    E_5(\bmf{\Phi};\tilde{\bmf{m}}-\tilde{\bmf{m}}_h,\bmf{v}_h)
    = &
    \sum_{i=1}^{N_e}\int_{L}
    \left(\left(\bmf{\Phi}(Q_{i+1})-\bmf{\Phi}(Q_i)\right)\times \frac{\partial\tilde{\bmf{m}}}{\partial\bmf{\nu}}\right) \cdot \left(\bmf{v}_h-I_h^*\bmf{v}_h\right) \, \mr{d}S\\
    = &
    \sum_{i=1}^{N_e}\int_{L}
    \left(\left(\bmf{\Phi}(Q_{i+1})-\bmf{\Phi}(Q_i)\right)\times \left( \frac{\partial\tilde{\bmf{m}}}{\partial\bmf{\nu}}-C_m\right)\right) \cdot\left(\bmf{v}_h-I_h^*\bmf{v}_h\right) \, \mr{d}S,
\end{aligned}
\end{equation*}
where $C_m$ is a constant defined as
\begin{equation*}
    C_m := \frac{1}{2} \left( \frac{\partial\tilde{\bmf{m}}_h^{i}}{\partial \bmf{\nu}}+\frac{\partial\tilde{\bmf{m}}_h^{i+1}}{\partial \bmf{\nu}} \right)
\end{equation*}
with $\tilde{\bmf{m}}_h^{i}$ and $\tilde{\bmf{m}}_h^{i+1}$ being the bounds of $\tilde{\bmf{m}}_h$ in the triangular element $K_{Q_i}$ and $K_{Q_{i+1}}$, respectively.
Consequently, we obtain
\begin{equation*}
\begin{aligned}
    E_5(\bmf{\Phi};\tilde{\bmf{m}}-\tilde{\bmf{m}}_h,\bmf{v}_h)
    = &
    \sum_{i=1}^{N_e}\int_{L}
    \left(\left(\bmf{\Phi}(Q_{i+1})-\bmf{\Phi}(Q_i)\right)\times \left( \frac{\partial\tilde{\bmf{m}}}{\partial \bmf{\nu}}-\frac{\partial\tilde{\bmf{m}}_h^k}{\partial \bmf{\nu}}\right)\right)\cdot (\bmf{v}_h-I_h^*\bmf{v}_h) \, \mr{d}S,
\end{aligned}
\end{equation*}
where $k = i \,\mr{or}\, i+1$.

Owing to the inequality \cite{chou}
\begin{equation*}
    \left( \int_{\partial K_Q}(\bmf{v}_h-I_h^*\bmf{v}_h)^2\,\mr{d}S \right)^{\frac{1}{2}} \leq C h^{\frac{1}{2}}\|\bmf{v}_h\|_{\bmf{H}^1},
\end{equation*}
we obtain
\begin{equation*}
 |E_5(\bmf{\Phi};\tilde{\bmf{m}}-\tilde{\bmf{m}}_h,\bmf{v}_h)|
 \leq
 C h (\|\tilde{\bmf{m}}-\tilde{\bmf{m}}_h\|_{\bmf{H}^1}+h^{1/2}\|\tilde{\bmf{m}}-\tilde{\bmf{m}}_h\|_{\bmf{H}^1}^{1/2}\|\tilde{\bmf{m}}\|_{\bmf{H}^2}^{1/2})\|\bmf{v}_h\|_{\bmf{H}^1}.
\end{equation*}
Together with triangle inequality, it yields
\begin{equation}
|b_h(\bmf{\Phi};\tilde{\bmf{m}}_h,I_h^*\bmf{v}_h) - b_h(\bmf{\Phi};\bmf{v}_h,I_h^*\tilde{\bmf{m}}_h)|
\leq
C h \|\tilde{\bmf{m}}_h\|_{\bmf{H}^1}\|\bmf{v}_h\|_{\bmf{H}^1}, \quad \forall \tilde{\bmf{m}}_h,\bmf{v}_h\in \bmf{S}_h.
\label{bh-relation}
\end{equation}
These complete the proof.

\section{The proof of Lemma \ref{erfunction}}\label{app:lem-erfun}
To prove Lemma \ref{erfunction}, we denote
$\mathcal{A}_h(\bmf{\Phi};R_h\tilde{\bmf{m}}_h,\bmf{v}_h) = \varepsilon_a (R_h\tilde{\bmf{m}},\bmf{v}_h) - \varepsilon_b(\bmf{\Phi};R_h\tilde{\bmf{m}}_h,\bmf{v}_h)$ with
\begin{equation*}
   \varepsilon_a (R_h\tilde{\bmf{m}},\bmf{v}_h)= a(R_h\tilde{\bmf{m}},\bmf{v}_h)-a_h(R_h\tilde{\bmf{m}}_h,I_h^*\bmf{v}_h), \,\forall \bmf{m}_h,\bmf{v}_h\in \bmf{S}_h,
\end{equation*}
and
\begin{equation*}
\varepsilon_b(\bmf{\Phi};R_h\tilde{\bmf{m}}_h,\bmf{v}_h)
=
b(\bmf{\Phi};R_h\tilde{\bmf{m}}_h,\bmf{v}_h)
-
b_h(\bmf{\Phi};R_h\tilde{\bmf{m}}_h,I_h^*\bmf{v}_h),\,\forall \bmf{m}_h,\bmf{v}_h\in \bmf{S}_h.
\end{equation*}

Repeating the proof as in \cite{chat1,chat2}, we get
\begin{equation*}
|\varepsilon_h(\bmf{f},\bmf{v}_h)|
\leq
Ch^{i+j}\|\bmf{f}\|_{\bmf{H}^i}\|\bmf{v}_h\|_{\bmf{H}^j},
\quad
\bmf{f}\in \bmf{H}^i, i,j=0,1, \, \bmf{v}_h\in \bmf{S}_h,
\end{equation*}
and
\begin{equation}
|\varepsilon_a(R_h \tilde{\bmf{m}},\bmf{v}_h)|
\leq Ch^{i+j}\|\tilde{\bmf{m}}\|_{\bmf{H}^{1+i}}\|\bmf{v}_h\|_{\bmf{H}^j},
\quad
\tilde{\bmf{m}}\in \bmf{H}^{1+i}\cap \bmf{H}^1_0,i,j=0,1,\,\bmf{v}_h\in \bmf{S}_h.
\label{erfun-a1}
\end{equation}
For the bound of $|\varepsilon_b(\bmf{\Phi};R_h \tilde{\bmf{m}},\bmf{v}_h)|$,
we consider the linear operator
$$L_1|_{K_Q} : \bmf{P}_1(K_Q) \rightarrow \bmf{L}^2(K_Q), \quad K_Q \in T_h, $$
which satisfies $\int_{K_Q} L_1 C \mr{d} x \mr{d} y= \int_{K_Q} C \mr{d} x\mr{d}y$.
Then for all constants $C$ and $K_Q \in T_h$,
\begin{equation}
    \|\bmf{v}_h-L_1\bmf{v}_h\|_{\bmf{L}^q(K_Q)} \leq h_{K_Q} |\bmf{v}_h|_{\bmf{W}^{1,q}(K_Q)},
    \, \forall \, \bmf{v}_h \in \bmf{S}_h, \, 1 \leq q < \infty.
    \label{L1}
\end{equation}
Meanwhile, consider the linear operator
$$L_2|_{\partial K_Q} : \bmf{P}_1(K_Q) \rightarrow \bmf{L}^2(\partial K_Q),$$
which satisfies the following properties:
\begin{equation*}
    L_2 C |_{\partial K_Q} = C, \quad \mr{for \, all \, constants} \, C,
    \label{L21}
\end{equation*}
\begin{equation*}
    \int_{\partial K_Q} \bmf{v}_h \, \mr{d} S = \int_{\partial K_Q} L_2 \bmf{v}_h \, \mr{d}S,
    \quad   \forall \, \bmf{v}_h \,\in \bmf{P}_1(K_Q),
\label{L22}
\end{equation*}
\begin{equation*}
    \|L_2 \bmf{v}_h\|_{\bmf{L}^{\infty}(\partial K_Q)} \leq \|\bmf{v}_h\|_{\bmf{L}^{\infty}(\partial K_Q)},
    \, \forall \bmf{v}_h \, \in \, \bmf{S}_h.
    \label{L23}
\end{equation*}
Repeating the proof of Lemma 6.1 in \cite{chat2}, we can obtain:
Given $ \bmf{\Phi} \in \bmf{X}$,
for arbitrary triangle element $K_Q$,
there exists a constant $C$ independent of $K_Q$ such that
\begin{equation*}
    \left|\int_{\partial K_Q} (\bmf{\Phi} \times \nabla R_h\tilde{\bmf{m}})\cdot(\bmf{v}_h-L_2\bmf{v}_h) \,\mr{d}S \right|
    \leq C h_{K_Q} |\tilde{\bmf{m}}|_{\bmf{H}^1(K_Q)}|\bmf{v}_h|_{\bmf{H}^1(K_Q)},
    \quad \forall\,\bmf{m} \in \bmf{H}^1(K_Q), \, \bmf{v}_h \in \bmf{S}_h.
    \label{L11}
\end{equation*}
Applying Green's formula, we obtain
\begin{equation}
\begin{aligned}
    \varepsilon_b(\bmf{\Phi};R_h\tilde{\bmf{m}},\bmf{v}_h)
    =&
    \sum_{K_Q} (\mr{div}(\bmf{\Phi} \times \nabla R_h\tilde{\bmf{m}}), \bmf{v}_h - L_1 \bmf{v}_h)_{K_Q}\\
    & +
    \sum_{K_Q} (\bmf{\Phi} \times (\nabla R_h \tilde{\bmf{m}} \cdot \bmf{n}), \bmf{v}_h - L_2 \bmf{v}_h)_{\partial K_Q}\\
    :=&
    J_1 + J_2.
\end{aligned}
\label{erfun-b1}
\end{equation}
For each triangular element $K_Q$, applying Green's formula yields
\begin{equation*}
\begin{aligned}
     &\int_{K_Q}\mr{div}(\bmf{\Phi} \times \nabla R_h\tilde{\bmf{m}})\, \mr{d}x\mr{d}y\\
    =&
    \sum_{V_i}\int_{V_i \bigcap K_Q} \mr{div}(\bmf{\Phi} \times \nabla R_h\tilde{\bmf{m}}) \, \mr{d}x\mr{d}y\\
    =&
    -\sum_{V_i}\int_{\partial V_i \bigcap K_Q}(\bmf{\Phi} \times \nabla R_h\tilde{\bmf{m}})\cdot \bmf{n} \, \mr{d}S
    -\sum_{V_i}\int_{V_i \bigcap \partial K_Q}(\bmf{\Phi} \times \nabla R_h\tilde{\bmf{m}})\cdot \bmf{n} \, \mr{d}S.
\end{aligned}
\end{equation*}
Combined with \eqref{L1} and H\"{o}lder inequality, we obtain
\begin{equation}
    |J_1|
    \leq C \sum_{K_Q} \|\mr{div}(\bmf{\Phi} \times \nabla\tilde{\bmf{m}})\|_{\bmf{L}^2(K_Q)}\| \bmf{v}_h - L_1 \bmf{v}_h\|_{\bmf{L}^2(K_Q)}
    \leq C \sum_{K_Q} h_{K_Q} |\tilde{\bmf{m}}|_{\bmf{H}^1(K_Q)}|\bmf{v}_h|_{\bmf{H}^1(K_Q)}.
    \label{J1}
\end{equation}
Meanwhile, since $|\bmf{\Phi} \times (\nabla R_h \tilde{\bmf{m}} \cdot \bmf{n})|_{\bmf{H}^1(K_Q)} \leq C |\tilde{\bmf{m}}|_{\bmf{H}^1(K_Q)}$, it holds
\begin{equation}
    |J_2| \leq C \sum_{K_Q} h_{K_Q} |\tilde{\bmf{m}}|_{\bmf{H}^1(K_Q)}|\bmf{v}_h|_{\bmf{H}^1(K_Q)}.
    \label{J2}
\end{equation}
Summing the inequalities \eqref{erfun-b1}-\eqref{J2} yields
\begin{equation}
|\varepsilon_b(\bmf{\Phi};R_h\tilde{\bmf{m}},\bmf{v}_h)|
\leq Ch^{i+j}\|\tilde{\bmf{m}}\|_{\bmf{H}^{1+i}}\|\bmf{v}_h\|_{\bmf{H}^j},
\quad
\bmf{v}\in \bmf{H}^{1+i}\cap \bmf{H}^1_0,i,j=0,1.
\label{erfun-b}
\end{equation}
Together with \eqref{erfun-b} and \eqref{erfun-a1}, it yields \eqref{erfun-A1}.

\section{The proof of Proposition \ref{lem-t}}\label{app:lem-ah-t}
We rewrite $a_h(\bmf{m}_h,I_h^*\partial_t\bmf{m}_h)$ into
\begin{equation*}
a_h(\bmf{m}_h,I_h^*\partial_t\bmf{m}_h)
=
\sum_{K_Q \in T_h} {I}_{K_Q}(\bmf{m}_h,I_h^*\partial_t\bmf{m}_h),
\end{equation*}
where
\begin{equation*}
{I}_{K_Q}(\bmf{m}_h,I_h^*\partial_t\bmf{m}_h)
=
-\sum_{P \in \dot{K}_Q} \left[ \int_{\partial V_i \cap K_Q} \frac{\partial \bmf{m}_h}{\partial \bmf{\nu}} \,\mr{d}S\right]\cdot\partial_t\bmf{m}_h(P).
\end{equation*}
Since $ \frac{\partial \bmf{m}_h}{\partial x} $ and $ \frac{\partial\bmf{m}_h}{\partial y}$ are constants in the triangular element $K_Q$,
we have
\begin{equation*}
\begin{aligned}
&{I}_{K_Q}
 (\bmf{m}_h,I_h^*\partial_t\bmf{m}_h)\\ \nn
 =
 & -\sum_{i=1}^3 \partial_t\bmf{m}_h(P_i)\cdot\left[\int_{\overline{M_iQM_{i+2}}} \frac{\partial \bmf{m}_h}{\partial \bmf{\nu}} \,\mr{d}S\right] \\
=
& -\sum_{i=1}^3 \partial_t\bmf{m}_h(P_i) \cdot \Big[ \int_{\overline{M_iQM_{i+2}}} \frac{\partial \bmf{m}_h}{\partial x} \,\mr{d}y
-  \int_{\overline{M_iQM_{i+2}}} \frac{\partial \bmf{m}_h}{\partial y} \, \mr{d}x \Big] \\
=
& -\sum_{i=1}^3 \partial_t\bmf{m}_h(P_i) \cdot \left[ \frac{\partial \bmf{m}_h}{\partial x} (y_{M_{i+2}}-y_{M_i})
-                                   \frac{\partial \bmf{m}_h}{\partial y}  (x_{M_{i+2}}-x_{M_i})\right]\\ \nn
=
& S_{K_Q} \left[ \frac{\partial \bmf{m}_h}{\partial x} \cdot \partial_t \left( \frac{\partial \bmf{m}_h}{\partial x} \right) + \frac{\partial \bmf{m}_h}{\partial y} \cdot \partial_t \left( \frac{\partial \bmf{m}_h}{\partial y}\right)\right]\\ \nn
=
& S_{K_Q}  \partial_t |\bmf{m}_h|_{\bmf{H}^1(K_Q)}^2.
\end{aligned}
\end{equation*}
Therefore,
\begin{equation*}
a_h(\bmf{m}_h,I_h^*\partial_t\bmf{m}_h)
=
\sum_{K_Q \in T_h} S_{K_Q} \partial_t |\bmf{m}_h|_{\bmf{H}^1(K_Q)}^2
\geq
C \partial_t |\| \bmf{m}_h |\|_{h}^2.
\end{equation*}
\section*{Acknowledgments}
This work was supported in part by the grants NSFC 12271360, 11501399, and Jiangsu Provincial Scientific Research Center of Applied Mathematics under Grant No. BK20233002 (R. Du).

\bibliographystyle{amsplain}
\bibliography{ref}

\end{document}